\documentclass[12pt]{amsart}
\usepackage{fullpage,tikz,scalerel}
\usepackage[utf8]{inputenc}
\usepackage{color,amsfonts,amsmath,amssymb,amsthm,cases,wasysym}
\usepackage[normalem]{ulem}
\usepackage{svg}
\usepackage{enumitem}
\usepackage{booktabs}
\usepackage{mathtools}
\usepackage{float}
\usepackage{hyperref}
\usetikzlibrary{shapes.multipart}

\title{On the common generalization of gentle algebras and framed directed acyclic graphs}
\author{Jonah Berggren}
\date{}

\newtheorem{thm}{Theorem}[section]
\newtheorem{prop}[thm]{Proposition}
\newtheorem{lemma}[thm]{Lemma}
\newtheorem{cor}[thm]{Corollary}

\newtheorem{thmIntro}{Theorem}    
    
\theoremstyle{definition}
    
\newtheorem{defn}[thm]{Definition}
 
\newtheorem{remk}[thm]{Remark}
\newtheorem{example}[thm]{Example}

\newcommand{\K}{\mathcal K}

\renewcommand{\k}{k}
\newcommand{\e}{\varepsilon}

\newcommand{\inn}{\textup{in}}
\newcommand{\out}{\textup{out}}

\renewcommand{\epsilon}{\varepsilon}

\newcommand{\F}{\mathcal F}

\newcommand{\tL}{{\tilde\Lambda}}
\renewcommand{\L}{{\Lambda}}
\newcommand{\I}{\mathcal{I}}

\newcommand{\xx}{\textbf{x}}

\newcommand{\B}{\mathcal{B}}
\newcommand{\R}{\mathcal{R}}

\renewcommand{\epsilon}{\varepsilon}
\newcommand{\bK}{\bar{\mathcal K}}
\renewcommand{\K}{\mathcal{K}}
\renewcommand{\B}{\mathcal{B}}

\renewcommand{\int}{\circ}
\newcommand{\D}{\Delta}

\renewcommand{\R}{\mathfrak{F}}
\renewcommand{\emph}{\textbf}
\newcommand{\te}{\tilde e}
\newcommand{\tf}{\tilde f}
\newcommand{\tg}{\tilde g}
\renewcommand{\th}{\tilde h}
\newcommand{\G}{\Gamma}

\begin{document}

\maketitle

\begin{abstract}
	In the study of flow polytopes, a directed acyclic graph (DAG) with a choice of framing gives a regular unimodular triangulation on its space of unit nonnegative flows.
	In representation theory, a gentle algebra has recently been equipped with a space of unit flows admitting triangulation and subdivision results capturing its $\tau$-tilting theory.
	These theories from different areas of mathematics overlap: flows on gently framed DAGs are the same as flows on paired representation-finite gentle algebras.
	In this article we develop the common generalization of these two theories by defining {(framed) turbulence charts}, which may be thought of as analogs of (framed) DAGs without the conditions of (D)irectedness and (A)cyclicity.
	The space of unit flows on a turbulence chart is its {turbulence polyhedron}. We give presentation and subdivision results on turbulence polyhedra which restrict to known results in the settings of framed DAGs and gentle algebras.
\end{abstract}

\section{Introduction}

Unit flow polytopes model a notion of flow on a directed acyclic graph (DAG). They are a fundamental object of combinatorial optimization~\cite{Hille,MM,MS,RW} with connections to fields such as representation theory~\cite{BV}, Grothendieck polynomials~\cite{LMD}, and algebraic geometry~\cite{EM}.
Danilov, Karzanov, and Koshevoy~\cite{DKK} defined \emph{framings} on DAGs which induce a notion of pairwise compatibility on routes (paths from source to sink). The \emph{clique complex} of sets of pairwise compatible routes of a framed DAG serves as a combinatorial model for a regular unimodular \emph{framing triangulation} of the associated unit flow polytope. Framing triangulations have received considerable attention in recent years~\cite{WIWT,DHY,MMS,vBC,vBDMY} and many important classes of triangulated polytopes appear in this way, including associahedra, generalized permutahedra~\cite{MS}, $s$-permutahedra~\cite{Y2}, and many order polytopes~\cite{LMD}.
Moreover, the dual graphs of these framing triangulations have also seen recent study; in~\cite{vBC} the simplices of a framing triangulation were given a lattice structure generalizing classical lattices such as $\nu$-Tamari and $s$-weak order lattices.

Introduced in~\cite{AIR}, the $\tau$-tilting theory of a finite-dimensional algebra has proven to be a successful generalization of tilting and cluster-tilting theory~\cite{BMRRT}.
Gentle algebras are a particularly rich source of examples in the representation theory of algebras introduced in the 1980's and studied heavily since that time~\cite{AH81,AS87,BR87}.
It was independently shown in~\cite{PPP,BDMTY} that for a gentle algebra $\L$ one may construct a \emph{fringed algebra} $\tL$ of $\L$ by adding some extra ``fringe'' vertices and arrows.
The \emph{clique complex} (or non-kissing complex) of non-kissing walks between fringe vertices on $\tL$ describes the $\tau$-tilting theory of $\L$.
It was observed recently~\cite{WIWT} that the combinatorics of the $\tau$-tilting theory of certain gentle algebras aligns with the combinatorics of the framing triangulations of certain framed DAGs.
In response to this, the author~\cite{BERG} defined the \emph{turbulence polyhedron} $\F_1(\tL)$ of a gentle algebra $\L$ and obtained presentation, triangulation, and subdivision results for $\F_1(\tL)$ modelled by walks on the quiver of $\tL$.
This turbulence polyhedron has a concrete quotient map onto the \emph{g-polyhedron} of the gentle algebra $\L$ which captures its \emph{g-vector fan}~\cite{AIR}, a construction which has received much attention in the representation theory of algebras~\cite{AHIKM,AY,AsaiWall,DIJ,HPS,PPP2023,PPP,PYK}.
When $\L$ is paired and representation-finite, the triangulated turbulence polyhedron $\F_1(\tL)$ is the same as the framing-triangulated flow polytope of an associated \emph{gentle framed} DAG through the recipe first described in~\cite{WIWT}. Hence, the work~\cite{BERG} develops gentle algebras as a generalization of the theory of framing triangulations of gentle framed DAGs.

This raises a question: what is the common generalization of framed DAGs and gentle algebras, when viewing these as two independent generalizations of gentle framed DAGs?
The goal of this article is to answer this question by defining \emph{(framed) turbulence charts}. Turbulence charts will give rise to turbulence polyhedra, generalizing flow polytopes from directed graphs and turbulence polyhedra from gentle algebras, and a framing on a turbulence chart will induce a simplicial subdivision (or triangulation, in bounded cases) on the turbulence polyhedron indexed by collections of ``noncrossing'' (or non-kissing) maximal walks along the chart.

\begin{center}
	\begin{tikzpicture}[every text node part/.style={align=center}]
		\node (A) at (0,0) [draw] {\textbf{Gentle Framed DAGs} \\ ample framing-triangulated \\ flow polytopes};
		\node (B) at (-4,2) [draw] {\textbf{Framed DAGs} \\ framing-triangulated \\ flow polytopes};
		\node (C) at (4,2) [draw] {\textbf{Gentle Algebras} \\ gentle subdivided \\ turbulence polyhedra};
		\node (D) at (0,4) [draw] {\textbf{\textcolor{red}{Framed Turbulence Charts}} \\ \textcolor{red}{framing-subdivided} \\ \textcolor{red}{turbulence polyhedra}};
		\draw (A) -- (B) -- (D) -- (C) -- (A);
	\end{tikzpicture}
\end{center}

We hope that our new constructions will prove useful in unifying the theories of framing triangulations and the $\tau$-tilting theory of gentle algebras, and in transporting results between these settings. As a taste, we pose a few questions here:
\begin{enumerate}
	\item There is a great deal of study on volume and $h^*$-polynomials of flow polytopes~\cite{BV,WIWT,CKM,DHY,MMS,Zeil}.
		What techniques and results generalize to bounded turbulence polyhedra, and is there a representation-theoretic interpretation of their specialization to the g-polytope in the gentle bounded case?
	\item Maximal simplices of the bundle subdivision of a directable framed turbulence chart (i.e., framed DAG) have the structure of a framing lattice~\cite{vBC}, and maximal simplices of the bundle subdivision of a gentle framed turbulence chart (i.e., gentle algebra) have the structure of a $\tau$-tilting poset~\cite{AIR}; these orders agree in the gentle directable acyclic case. Do these generalize to a reasonable order on the maximal cells and/or maximal simplices of the bundle subdivision of an arbitrary framed turbulence chart?
	\item The techniques of this paper realize the bundle subdivision of an turbulence chart $(G,\sim,\R)$ as a restriction of the bundle subdivision of a gentle algebra $\L$ to a face of $\F_1(\tL)$. Can something more structural be said about the relationship between the representation theory of $\L$ and the bundle subdivision $\mathcal S_1(G,\sim,\R)$, perhaps using techniques similar to those used in~\cite{WILT}?
\end{enumerate}

\subsection{Summary of results in more detail}

Define a \emph{turbulence chart} $(G,\sim)$ to be an undirected graph $G$ such that for every vertex $v$ of $G$ with $\text{deg}(v)\geq2$, the half-edges at $v$ are separated into two nonempty equivalence classes by an equivalence relation $\sim_v$. We call the vertices of $G$ with degree 1 \emph{fringe vertices} and draw them as boxes; all other vertices are \emph{internal vertices} and drawn as dots. At each internal vertex $v$, we draw a blue line separating the equivalence classes of $\sim_v$. See, for example, the middle of Figure~\ref{fig:turbchartintro}.

One obtains a turbulence chart from a directed graph by (1) considering each source edge to begin at its own fringed vertex, (2) considering each sink edge to end at its own fringed vertex, (3) letting $\sim_v$ separate the incoming half-edges at each internal vertex $v$ from the outgoing half-edges, and (4) forgetting the orientations of all edges.
Figure~\ref{fig:turbchartintro} features a turbulence chart along with an analogous DAG and gentle algebra (ignoring the red numbers representing the framing for now).
A turbulence chart is \emph{directable} if it comes from a directed graph in this way; a turbulence chart fails to be directable if its edges may not be given orientations such that the incoming edges are separated from the outgoing edges at each internal vertex. 
Moreover, one obtains a turbulence chart from the fringed algebra of a gentle algebra by separating through $\sim_v$ the half-edges of the two relations at each internal vertex $v$, and subsequently forgetting the orientations of arrows.

\begin{figure}
	\centering
	\def\svgscale{.3}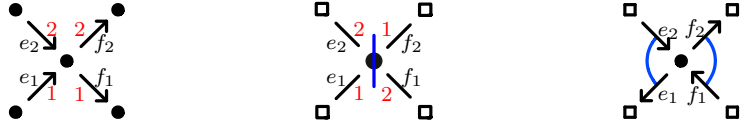
	\caption{A \textcolor{red}{framed} turbulence chart (middle) with an associated \textcolor{red}{framed} DAG (left) and fringed algebra (right).}
	\label{fig:turbchartintro}
\end{figure}

We say that a labelling $F$ of the edges of a turbulence chart \emph{conserves flow} at an internal vertex if the sum of the coefficients of edges in both equivalence classes at that vertex are equal. A \emph{flow} is a labelling which conserves flow at every vertex. A flow is \emph{unit} if the sum of coefficients of fringe edges is 2. The \emph{turbulence polyhedron} $\F_1(G,\sim)$ is the space of unit nonnegative flows. Note that if $(G,\sim)$ is associated to a directed graph (resp. fringed algebra) as in Figure~\ref{fig:turbchartintro}, then $\F_1(G,\sim)$ is equal to the flow polyhedron of the directed graph (resp. turbulence chart of the fringed algebra).

Similarly, we generalize paths on directed graphs and strings of a fringed algebra by defining \emph{strings} on turbulence charts to be walks along $G$ which must cross the equivalence class $\sim_v$ at each vertex $v$ they pass through. A \emph{route} is a maximal string (i.e., a string both of whose endpoints are fringe vertices) and a \emph{band} is a string which forms an infinitely repeatable cycle. We refer to routes and bands collectively as \emph{trails}. A turbulence chart is \emph{acyclic} if it has no bands, generalizing acyclicity of directed graphs.

We define some routes and bands to be \emph{elementary} (see Definitions~\ref{defn:elroute} and~\ref{defn:elband} for the technical definition or Figure~\ref{ELEMENTARYINTRO} for the pictorial summary) to obtain the following presentation of $\F_1(G,\sim)$:
\begin{thmIntro}[{Theorem~\ref{thm:pres-general}}]\label{ithm:pres}
	The map $p\mapsto\I(p)$ bijects elementary routes to vertices of $\F_1(G,\sim)$. The map $B\mapsto\I(B)$ is a map from elementary bands to the extremal rays of $\F_1(G,\sim)$ which is injective on simple bands but 2-to-1 on barbells.
\end{thmIntro}
\begin{figure}
\centering
	\begin{minipage}[b]{.15\textwidth}
  \centering
	\def\svgscale{.21}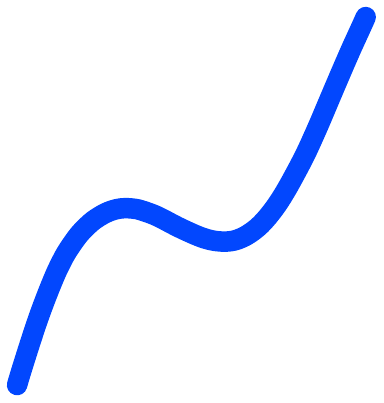

	simple
\end{minipage}
	\begin{minipage}[b]{.15\textwidth}
  \centering
	\def\svgscale{.21}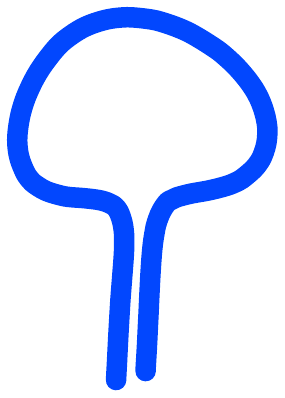

	lollipop
\end{minipage}
	\begin{minipage}[b]{.15\textwidth}
  \centering
	\def\svgscale{.21}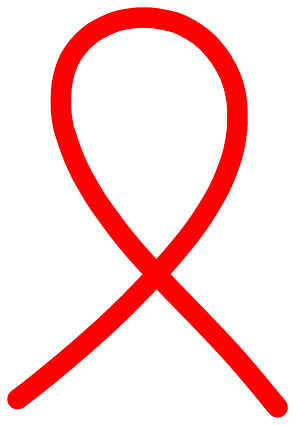

	nonelementary
\end{minipage}
	\begin{minipage}[b]{.15\textwidth}
  \centering
	\def\svgscale{.21}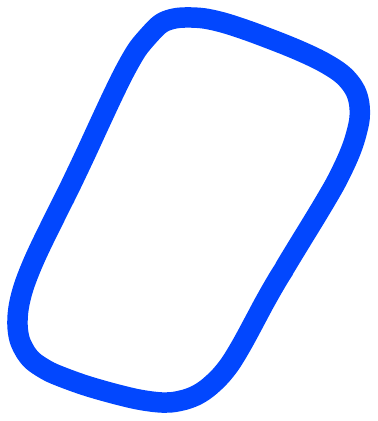

	simple
\end{minipage}
	\begin{minipage}[b]{.15\textwidth}
  \centering
	\def\svgscale{.21}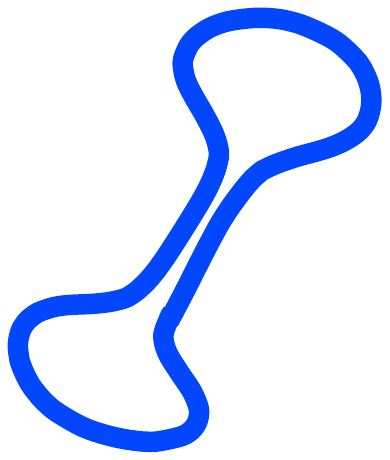

	barbell
\end{minipage}
	\begin{minipage}[b]{.15\textwidth}
  \centering
	\def\svgscale{.21}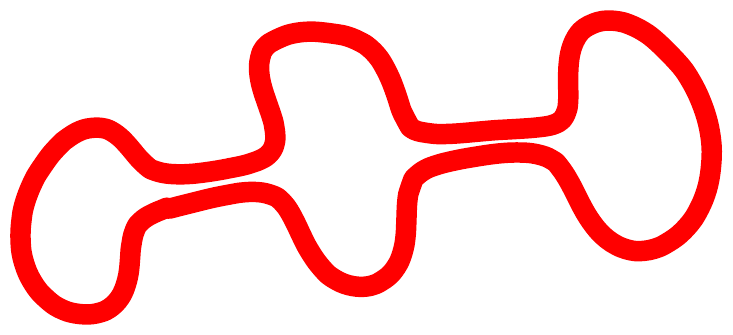

	nonelementary
\end{minipage}
	\caption{Examples of elementary (blue) and nonelementary (red) routes and bands.}
	\label{ELEMENTARYINTRO}
\end{figure}
Note, then, that a turbulence polyhedron is bounded if and only if its turbulence chart has no bands. For examples of this presentation result, see Figure~\ref{fig:trapezoid} (where the unique nonelementary route labels the unique integer point which is not a vertex) and Figure~\ref{fig:bowtie} (where the three vertices are labelled by the three elementary routes, and the unbounded direction represents the indicator vector of the two elementary bands).

We generalize framings on DAGs by defining a \emph{framing} $\R$ on a turbulence chart to be a set of total orders on the elements of every $\sim$-equivalence classes of half-edges.
A framing on a turbulence chart induces a notion of pairwise compatibility on its trails. 
Framings on DAGs give rise to framings on turbulence charts by preserving all incoming orders and reversing all outgoing orders, and the turbulence chart of a fringed quiver is given a framing by ordering all head half-edges high and tail half-edges low. In this way, framed turbulence charts generalize both framed DAGs and fringed quivers.
Call a framed turbulence chart \emph{gentle} if it comes from a fringed quiver of a gentle algebra.

A \emph{clique} is a collection of pairwise compatible routes, and a \emph{bundle} is a collection of pairwise compatible routes and bands.
Given a drawing of a turbulence chart $(G,\sim)$, one may define the \emph{clockwise framing} by ordering half-edges clockwise around each internal vertex as in Figures~\ref{fig:trapezoid} and~\ref{fig:bowtie}. In this case, a bundle is a collection of routes and bands which can be drawn together with no crossings. This generalizes the notion of a {planar framing} on an embedding of a DAG~\cite{DKK,vBC}.

If $\bK=\K\cup \B$ is a bundle of $(G,\sim,\R)$ (where $\K$ consists of routes and $\B$ consists of bands), then a \emph{(unit) $\bK$-bundle combination} is a nonnegative linear combination
$F=\sum_{p\in \bK}a_p\I(p)$
of indicator vectors of trails of $(G,\sim,\R)$ such that $\sum_{p\in\K}a_p=1$.
A bundle combination is a \emph{clique combination} if $\B=\emptyset$ or if $a_p$ is zero for every $B\in \B$. It is \emph{positive} if $a_p$ is positive for every $p\in \bK$.
If $\bK$ is a bundle, then its \emph{bundle simplihedron} $\D_1(\bK)$ is the polyhedron of unit $\bK$-bundle combinations (a ``simplihedron'' is a generalization of simplex to allow for unbounded polyhedra).

\begin{thmIntro}[{Theorem~\ref{thm:bundle}}]
	\label{ithm:bundle}
	Let $(G,\sim,\R)$ be a framed turbulence chart.
	Any nonnegative flow has at most one representation as a positive bundle combination.
	A dense subset of $F\in\F_{\geq0}(G,\sim)$ including all rational flows may be obtained as a bundle combination. If $F$ is an integer flow, then its bundle combination has integral coefficients.
\end{thmIntro}

Theorem~\ref{ithm:bundle} is really a simplicial subdivision result on $\F_1(G,\sim)$. Specifically, for an arbitrary framed turbulence chart $(G,\sim,\R)$ the \emph{bundle subdivision}
	\[\mathcal S_1(G,\sim,\R):=\{\D_1(\bK)\ :\ \bK\text{ is a maximal bundle of }(G,\sim,\R)\text{ with at least one route}\}\]
	is a subdivision of $\F_1(G,\sim)$ into simplihedra which covers every rational point of $\F_1(G,\sim)$.
	When $(G,\sim,\R)$ is bounded, there are no bands the bundle subdivision is a unimodular triangulation called the \emph{clique triangulation}. 
	Versions of bundle subdivisions and clique triangulations also hold for $\F_{\geq0}(G,\sim)$.

When $(G,\sim)$ is directable (resp. gentle), Theorems~\ref{ithm:pres} and~\ref{ithm:bundle} specialize to known presentation, triangulation and subdivision results on the corresponding framed DAG (resp. fringed algebra). See Figure~\ref{fig:sqturb}, which shows the chart from Figure~\ref{fig:turbchartintro} with its turbulence polytope triangulated by the two maximal cliques (which are the same as the maximal cliques of the framed DAG or fringed algebra of Figure~\ref{fig:turbchartintro}; see Figures~\ref{fig:sqdag} and~\ref{fig:sqgent}, respectively).
\begin{figure}
	\centering
	\def\svgscale{.3}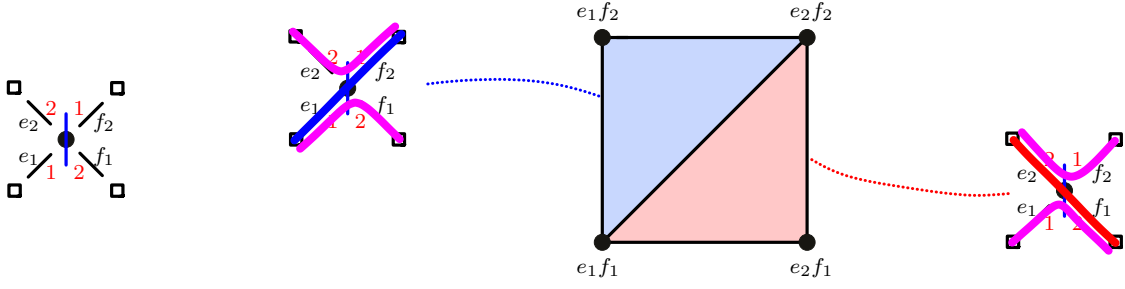
	\caption{The clique-triangulated turbulence polytope of the chart in Figure~\ref{fig:turbchartintro} (which agrees with the triangulated polytopes of the framed DAG and fringed algebra).}
	\label{fig:sqturb}
\end{figure}

See Figure~\ref{fig:trapezoid} for the bundle subdivisions of two framings of the same acyclic turbulence chart. Figure~\ref{fig:bowtie} shows the bundle subdivision of a framed turbulence chart with a cycle. Since both of these figures feature the clockwise framing, the cells of their subdivisions represent maximal collections of trails which can be drawn with no crossings.

	\begin{figure}
		\centering
		\def\svgscale{.25}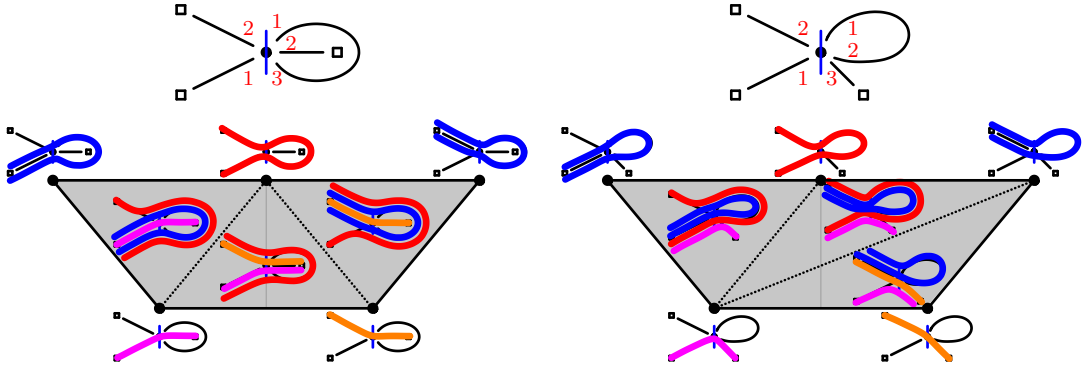
		\caption{Two framings of a turbulence chart which are acyclic, but not directable or gentle. The five integer points are labelled by routes, and the three maximal simplices of the bundle subdivision are labelled by maximal bundles.}
		\label{fig:trapezoid}
	\end{figure}

	\begin{figure}
		\centering
		\def\svgscale{.25}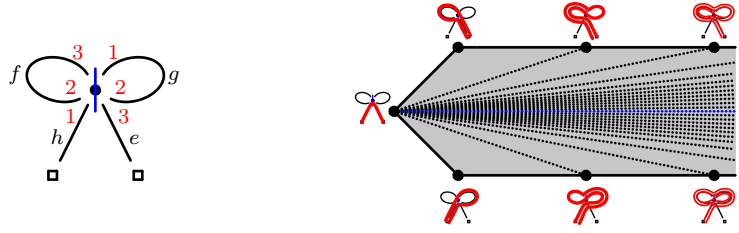
		\caption{A framed turbulence chart which is not directable, acyclic, or gentle.}
		\label{fig:bowtie}
	\end{figure}

We now speak to the methods used to prove Theorems~\ref{ithm:pres} and \ref{ithm:bundle}. Versions of these theorems for flows on fringed algebras of gentle algebras were proven in~\cite{BERG}. Recall that certain \emph{gentle} framed turbulence charts come directly from fringed algebras. The technical core of this paper is a process by which we may take a general framed turbulence chart $(G,\sim,\R)$ and obtain a gentle framed turbulence chart $(G',\sim',\R')$ with a set of edges $W$ such that deleting the edges of $W$ from $(G',\sim',\R')$ and then performing some contractions retrieves the original $(G,\sim,\R)$; this is similar to the construction of \emph{ample envelopes} appearing in~\cite{WILT}. In this way, we show that (up to unimodular equivalences realized by contractions) turbulence polyhedra are precisely faces of gentle turbulence polyhedra. Moreover, we are able to prove Theorems~\ref{ithm:pres} and \ref{ithm:bundle}
by transporting analogous results from the gentle case.

We finish the introduction by discussing the properties of directability, acyclicity, and gentleness of a framed turbulence chart. These properties may be independently varied; see Figure~\ref{fig:CUBEOFGEN}. 
Framed turbulence charts which are directable, acyclic, and gentle correspond to gentle framed DAGs~\cite[Proposition 4.29]{BERG}, which are closely related to the amply framed DAGs of~\cite{DKK,WIWT}.
Directable (resp. framed) turbulence charts are those whose flow may be given consistent direction across the edges, and are described by (resp. framed) directed graphs (which may or may not have oriented cycles). In fact, our presentation and subdivision results are new when specialized to framed directed graphs with cycles as we discuss in Section~\ref{sec:specialization} (note that the clique complexes of certain framed directed graphs with cycles, called \emph{cyclic amply framed DAGs}, were described in~\cite{UQAM}; see Remark~\ref{remk:uqam1}).
Gentle turbulence charts are those which may be associated to fringed algebras of gentle algebras.
Acyclicity is equivalent to the presence of bands, which generalize oriented cycles in the directed graph setting and representation-infiniteness in the gentle algebra setting. A turbulence chart is acyclic if and only if its turbulence polyhedron is bounded, in which case the bundle subdivision is a unimodular triangulation of the turbulence polytope. Turbulence polytopes of acyclic turbulence charts are furthermore examples of signed flow polytopes of signed graphs, defined by M\'esz\'aros and Morales~\cite{MMSigned} as a generalization of flow polytopes from type $A_n$ root systems to types $C_{n+1}$ and $D_{n+1}$.

\begin{figure}
\begin{center}
	\scalebox{.89}
	{
	\begin{tikzpicture}[xscale=1, yscale=1, every text node part/.style={align=center}]
		\node (AAA) at (0,0) [draw] {\textbf{\color{blue}{gentle\ +\ acyclic\ +\ directable}} \\ \textbf{\color{blue}{turbulence charts}}
		\\  \\	\begin{minipage}{170pt}\begin{itemize}
				\item Paired representation-finite gentle algebras
				\item Gentle framed DAGs
		\end{itemize}\end{minipage}
		};
		\node (AAB) at (0,8) [draw] {\textbf{\color{blue}{gentle\ +\ \textcolor{red}{\sout{acyclic}}\ +\ directable}} \\ \textbf{\color{blue}{turbulence charts}}
	\\  \\	\begin{minipage}{160pt}\begin{itemize}
				\item Paired gentle algebras
				\item Gentle framed directed graphs
		\end{itemize}\end{minipage}
		};
		\node (ABA) at (8,0) [draw] {\textbf{\color{blue}{gentle\ +\ acyclic\ +\ \textcolor{red}{\sout{directable}}}} \\ \textbf{\color{blue}{turbulence charts}}
	\\  \\	\begin{minipage}{150pt}\begin{itemize}
				\item Representation-finite gentle algebras
		\end{itemize}\end{minipage}
		};
		\node (ABB) at (8,8) [draw] {\textbf{\color{blue}{gentle\ +\ \textcolor{red}{\sout{acyclic}}\ +\ \textcolor{red}{\sout{directable}}}} \\ \textbf{\color{blue}{turbulence charts}}
		\vspace{5pt}  \\	\begin{minipage}{140pt}\begin{itemize}
				\item Gentle algebras
		\end{itemize}\end{minipage}
		};
		\node (BAA) at (4,4) [draw] {\textbf{\color{blue}{\textcolor{red}{\sout{gentle}}\ +\ acyclic\ +\ directable}} \\ \textbf{\color{blue}{turbulence charts}}
		\vspace{5pt}  \\	\begin{minipage}{140pt}\begin{itemize}
				\item Framed DAGs
		\end{itemize}\end{minipage}
		};
		\node (BAB) at (4,12) [draw, shade, top color=pink, bottom color=pink] {\textbf{\color{blue}{\textcolor{red}{\sout{gentle}}\ +\ \textcolor{red}{\sout{acyclic}}\ +\ directable}} \\ \textbf{\color{blue}{turbulence charts}}
		\vspace{5pt}  \\	\begin{minipage}{140pt}\begin{itemize}
				\item Framed directed graphs
				\item Figure~\ref{fig:kronface}
		\end{itemize}\end{minipage}
		};
		\node (BBA) at (12,4) [draw, shade, top color=pink, bottom color=pink] {\textbf{\color{blue}{\textcolor{red}{\sout{gentle}}\ +\ acyclic\ +\ \textcolor{red}{\sout{directable}}}} \\ \textbf{\color{blue}{turbulence charts}}
		\vspace{5pt}  \\	\begin{minipage}{140pt}\begin{itemize}
				\item Figure~\ref{fig:trapezoid}
		\end{itemize}\end{minipage}
		};
		\node (BBB) at (12,12) [draw, shade, top color=pink, bottom color=pink] {\textbf{\color{blue}{\textcolor{red}{\sout{gentle}}\ +\ \textcolor{red}{\sout{acyclic}}\ +\ \textcolor{red}{\sout{directable}}}} \\ \textbf{\color{blue}{turbulence charts}}
		\vspace{5pt}  \\	\begin{minipage}{140pt}\begin{itemize}
				\item Figure~\ref{fig:bowtie}
		\end{itemize}\end{minipage}
		};

		\draw (AAA) -- (AAB) -- (ABB) -- (ABA) -- (AAA);
		\draw (AAA) -- (BAA) -- (BBA) -- (ABA);
		\draw (BAA) -- (BAB) -- (BBB) -- (BBA) -- (BAA);
		\draw (ABA) -- (BBA);
		\draw (AAB) -- (BAB);
		\draw (ABB) -- (BBB);
	\end{tikzpicture}}
\end{center}
	\caption{The cube of turbulence charts generated by the three independent generalizations to gentle framed DAGs which we treat in this article. Gentle framed DAGs lie at the bottom-left; moving up generalizes by removing the assumption of acyclicity, moving right generalizes by removing the assumption of directability, and moving up-right diagonally generalizes by removing the assumption of gentleness. The three shaded boxes are those for which our presentation and subdivision results are new.}
	\label{fig:CUBEOFGEN}
\end{figure}
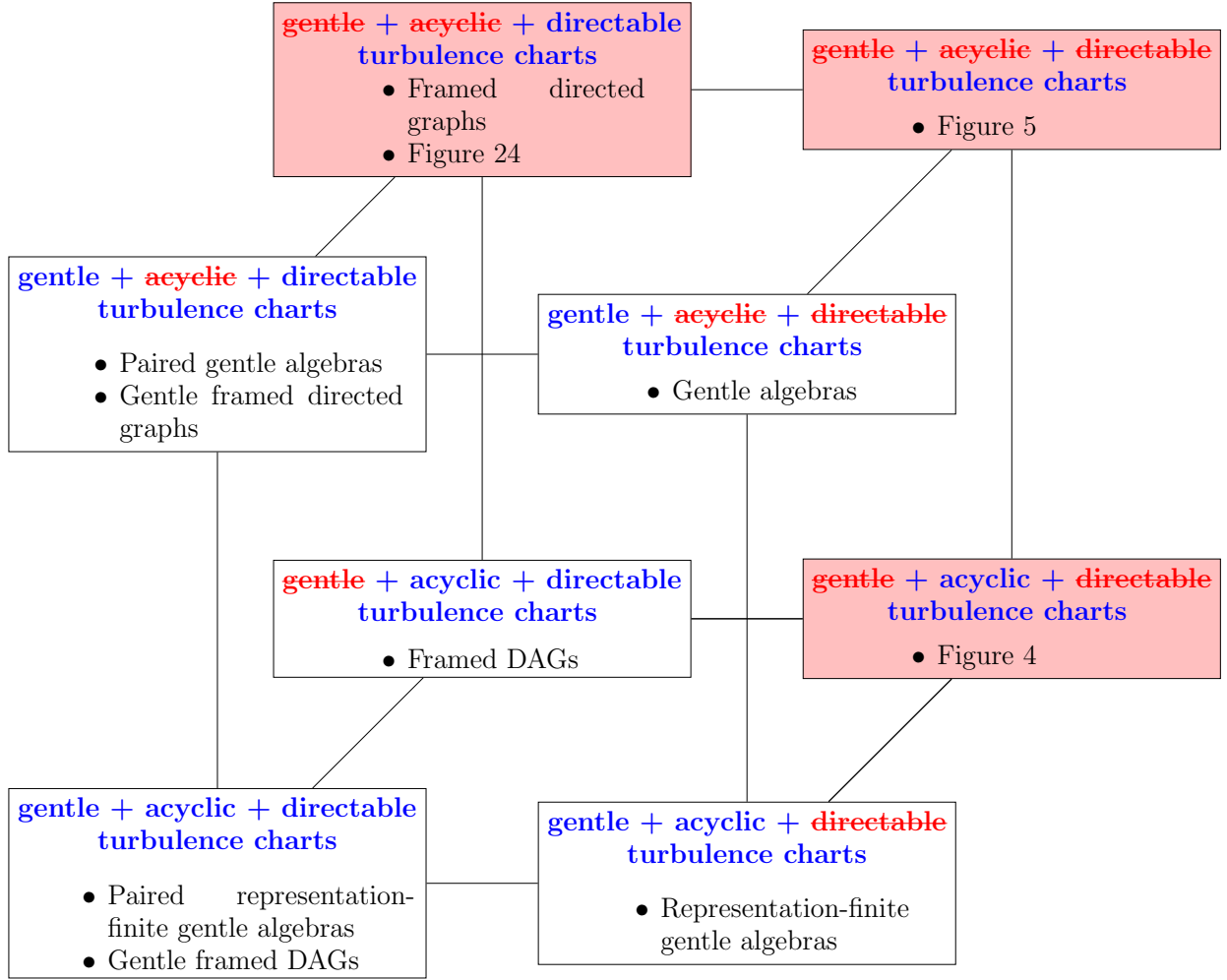

The structure of the paper is as follows.
Sections~\ref{sec:back-dag} and~\ref{sec:back-turb} give background on flow polytopes of (framed) DAGs and turbulence polyhedra of fringed algebras.
In Section~\ref{sec:tc} we define turbulence charts and their turbulence polyhedra, trails, and bundle complex. In Section~\ref{sec:gda} we define and isolate the properties of gentleness, directability, and acyclicity. Section~\ref{sec:ge} is the technical core of the paper, where we reduce an arbitrary turbulence chart to a gentle envelope with some edges deleted.
In Section~\ref{sec:pres-subd} we use gentle envelopes to prove our presentation and subdivision results about turbulence polyhedra.
Section~\ref{sec:specialization} reiterates that our results specialize to known results for gentle algebras and framed DAGs, and that they specialize to new results for framed directed graphs with cycles.
We finish the article in with some examples from the Kronecker quiver in Section~\ref{sec:examples}.

\subsection*{Acknowledgments}

The author was supported by the NSF grant DMS-2451909. This work was supported by a grant from the Simons Foundation International [SFI-MPS-TSM-00013650, KS].

\section{Background on Polyhedra, Subdivisions, and Triangulations}
\label{ssec:subd}

We first give some basic definitions and results on polyhedra. 
A \emph{(convex) polyhedron} $P$ in $\mathbb R^n$ is the solution space $\{\xx\ :\ A\xx\leq0\}$, where $A$ is a real matrix with $n$ columns. A \emph{polytope} is a bounded polyhedron.
A \emph{face} of the polyhedron $P$ is a space of the form $P\cap\{\xx\ :\ A'\xx=0\}$ where $A'$ is a matrix given by deleting some columns from $A$.
A \emph{vertex} of $P$ is a face of dimension 0.
A polyhedron is \emph{pointed} if it contains at least one vertex. 
The polyhedra considered in this paper will all be in $\mathbb R_{\geq0}^n$ and hence pointed.

\subsection{Vertices, extremal rays, and presentations}

Given finite sets $A$ and $B$ of vectors in $\mathbb R^n$ such that $A$ is nonempty, define $P_{(A,B)}:=\textup{conv}(A)+\textup{cone}(B)$.

\begin{thm}[{\cite[Corollary 7.1b]{Schrijver}}]
	\label{thm:decomposition}
	A nonempty subset $P\subseteq\mathbb R^n$ is a polyhedron if and only if $P=P_{(A,B)}$ for some finite sets $A,B\subseteq\mathbb R^n$ with $A\neq\emptyset$.
\end{thm}

We say that $P_{(A,B)}$ is \emph{presented} by $(A,B)$.
The polyhedra considered in this paper will all be in $\mathbb R_{\geq0}^n$ and hence pointed, so we cite a version of~{\cite[Theorem 8.5]{Schrijver}} specific to pointed polyhedra.

\begin{thm}[{\cite[Theorem 8.5]{Schrijver}}]
	\label{thm:sch-pres}
	Let $P$ be a nonempty pointed polyhedron in $\mathbb R^n$.
	Let $A$ be the set of vertices of $P$.
	Let $B\subseteq \mathbb R^n$ be minimal under inclusion with the property that
	\[
		P=\textup{conv}\big(A\big)+\textup{cone}\big(B\big).
	\]
	Then the set $B$ is unique up to positive scaling of its vectors.
\end{thm}

In the notation of Theorem~\ref{thm:sch-pres}, we call the elements of $B$ (up to positive scaling) the \emph{extremal rays} of $P$ and we consider the vertices $A$ and extremal rays $B$ to be the \emph{minimal presentation} of $P$.
We call $\textup{cone}\big(B\big)$ the \emph{recession cone} of $P$.

\subsection{Simplices and simplihedra}

We prove the following proposition.

\begin{prop}\label{prop:simplicial}
	Let $P$ be a pointed polyhedron with vertex set $A$ and extremal ray set $B$. The following are equivalent:
	\begin{enumerate}
		\item Every point $\xx\in P$ has a unique presentation
			\[\xx=\sum_{v\in A}a_vv+\sum_{r\in B}a_rr\]
			where $\sum_{v\in A}a_v=1$ and all coefficients $a_v$ and $a_s$ are nonnegative.
		\item The set of nonempty faces of $P$ is
			\[\{P_{(A',B')}\ :\ \emptyset\subsetneq A'\subseteq A\textup{ and }B'\subseteq B\}.\]
	\end{enumerate}
\end{prop}
\begin{proof}
	Suppose (1) is true with the intent to prove (2). 
	Define a map $\phi$ sending a point 
	$\xx=\sum_{v\in A}a_vv+\sum_{r\in B}a_rr$ of $P$ to the coordinate $(a_s)_{s\in A\cup B}$ of $\mathbb R^{A\cup B}$.
	It may be checked that $P$ is an affine equivalence from $P$ onto its image, the set of points $(a_s)_{s\in A\cup B}\in\mathbb R^{A\cup B}$ such that all coordinates $a_s$ are nonnegative and $\sum_{s\in A}a_s=1$. This image is the Cartesian product of the standard simplex in $\mathbb R^A$ with the nonnegative cone in $\mathbb R^B$. It is immediate that the nonempty faces of this Cartesian product are as described in (2).

	Now suppose (2) with the intent to prove (1).
	For any $r\in B$, the fact that $P_{(A,B\backslash\{r\})}$ is a nonempty face of $P$ implies that there exists a linear functional $\phi_r$ whose minimum on $P$ is achieved precisely on the face $P_{(A,B\backslash\{r\})}$.
	In particular, this implies that $\phi_r(r')=0$ for all $r'\in B\backslash\{r\}$ and that $\phi_r(v_1)=\phi_r(v_2)$ for any $v_1,v_2\in A$. Define the constant $c_r$ to be $\phi_r(v)$ for a vertex $v\in A$.
	Similarly, if $|A|>1$ then for any $v\in A$ there exists a linear functional $\phi_v$ whose minimum on $P$ is achieved precisely on the face $P_{(A\backslash\{v\},v)}$.
	Then $\phi_v(r)=0$ for all $r\in B$ and $\phi_v(v_1)=\phi_v(v_2)$ for all $v_1,v_2\in A$ distinct from $v$. Define the constant $c_v$ to be $\phi_v(v')$ for $v'\in A\backslash\{v\}$.

	Pick a point $\xx\in P$.
	Then it may be presented 
	$\xx=\sum_{v\in A}a_vv+\sum_{r\in B}a_rr$
			where $\sum_{v\in A}a_v=1$ and all coefficients $a_v$ and $a_r$ are nonnegative.
	To show (1), we must show that the coefficients $a_s$ for $s\in A\cup B$ are determined by $\xx$.
	First, take $r\in B$ and observe that
	\begin{align*}
		\phi_r(\xx)&=\sum_{v\in A}a_v\phi_r(v)+\sum_{r'\in B}a_{r'}\phi_r(r')\\
		&=\sum_{v\in A}a_vc_r+a_r\phi_r(r)\\
		&=c_r+a_r\phi_r(r).
	\end{align*}
	Solving for $a_r$ gives $a_r=\frac{\phi_r(\xx)-c_r}{\phi_r(r)}$.

	Now, take $v\in A$. If $|A|=1$, then it is immediate that $a_v=1$ and there is nothing more to show. Otherwise,
	\begin{align*}
		\phi_v(\xx)&=\sum_{v'\in A}a_{v'}\phi_v(v')+\sum_{r\in B}a_r\phi_v(r)\\
		&=a_v\phi_v(v)+(1-a_v)c_v.
	\end{align*}
	Solving for $a_v$ gives
	$a_v=\frac{\phi_v(\xx)-c_v}{\phi_v(v)-c_v}$.

	We have obtained formulas for the coefficients $(a_s)_{s\in A\cup B}$ depending only on $\xx$ and the functionals $\phi_r$ and $\phi_v$, hence the presentation
	$\xx=\sum_{v\in A}a_vv+\sum_{r\in B}a_rr$ is unique. This completes the proof.
\end{proof}

We call a polyhedron \emph{simplicial} if it satisfies the conditions of Proposition~\ref{prop:simplicial}. We also say that $P$ is a \emph{simplihedron}. A \emph{simplex} is a bounded simplihedron.

\subsection{Subdivisions of polyhedra}

We now make clear what we ask of a subdivision of a polyhedron, following~\cite{BERG}.

\begin{defn}\label{defn:subd}
	Let $P$ be a lattice polyhedron. A \emph{subdivision} of $P$ is a set $\mathcal{S}$ of polyhedra (called ``cells'') satisfying:
\begin{enumerate}
	\item \textbf{Density:} $\cup_{Q\in\mathcal{S}}Q$ is a dense subset of $P$,
	\item \textbf{Strong Intersection Property:} for distinct $Q_1,Q_2\in\mathcal{S}$, the polyhedron $Q_1 \cap Q_2$ is a face of $Q_1$ and $Q_2$.
\end{enumerate}
	The subdivision is \emph{simplicial} if all cells are simplihedra.
	The subdivision is \emph{complete} if $P=\cup_{Q\in\mathcal{S}}Q$.
	A \emph{triangulation} is a subdivision in which each cell is a (bounded) simplex of dimension $\text{dim}(P)$.
	A triangulation is \emph{unimodular} if each cell has normalized volume 1 (equivalently, if the traditional volume of each cell within the affine span $\textup{aff}(P)$ is the factorial of $\dim P$).
	As an abuse of notation, we will also use the word (resp. unimodular) triangulation to refer to the cone over a (resp. unimodular) triangulation.
\end{defn}

Other sources may require all subdivisions to be complete, or require each cell of a subdivision to be full-dimensional.

\section{Background on Framed Directed Graphs and Flow Polyhedra}
\label{sec:back-dag}

We give background on framed directed graphs, focusing on framed DAGs~\cite{DKK} and the induced regular unimodular triangulations on flow polytopes.

\subsection{Flow polyhedra}

Let $G=(V,E)$ be a finite directed graph with vertex set $V$ and edge set $E$. 
We consider each edge $e\in E$  to begin at its \emph{tail} $t(e)$ and end at its \emph{head} $h(e)$.
For each $v\in V$, let $\inn(v):=\{e\in E\ :\ h(e)=v\}$ and $\out(v):=\{e\in E\ :\ t(e)=v\}$ denote the incoming and outgoing edges of $v$, respectively. A vertex $v$ is called a \emph{source} if $\inn(v)=\emptyset$ and it is called a \emph{sink} if $\out(v)=\emptyset$.
All other vertices are called \emph{internal vertices}. 
An edge $e\in E$ is \emph{internal} if it is between two internal vertices, and otherwise is a \emph{source/sink edge}. More specifically, it is either a \emph{source edge} (if its tail is a source) and/or a \emph{sink edge} (if its head is a sink). 
A \emph{route} of $G$ is a maximal (directed) path in $G$.
An edge $e\in E$ is \emph{idle} if it is the only incoming edge to its internal head, or the only outgoing edge to its internal tail.

\begin{defn}\label{defn:flow-polytope}
	A \emph{flow} $F$ on a directed graph $G$ is a function $F:E\to\mathbb R$ which preserves flow at each internal vertex, i.e., for every internal vertex $v$ we have
	\[\sum_{e\in\inn(v)}F(e)=\sum_{e\in\out(v)}F(e).\]
	A flow $F$ is \emph{nonnegative} if $F(e)\geq0$ for all edges $e\in E$.
	The \emph{cone of (nonnegative) flows} $\F_{\geq0}(G)$ is the space of nonnegative flows on $G$. The \emph{(unit) flow polyhedron} $\F_1=\F_1(G)$ is the set of all \emph{unit} nonnegative flows of $G$, i.e., nonnegative flows satisfying
	\[\sum_{\substack{v\text{ is a source}\\ e\in\out(v)}}F(e)=1.\]
	If $G$ is acyclic, then $\F_1(G)$ is a polytope and we call it the \emph{flow polytope} of $G$. If $G$ has a cycle, then $\F_1(G)$ is an unbounded polyhedron.
\end{defn}

Refer to the right of Figure~\ref{fig:flow} for examples of a unit flow (labelled in blue) on a directed graph.

\begin{remk}
	More generally, given a directed graph with a netflow vector $\mathbf{a}$ on its vertices one defines an $\mathbf{a}$-flow to be a labelling $F$ of the edges of $E$ satisfying
	\[
		\mathbf{a}_v=\sum_{e\in\out(v)}F(e)-\sum_{e\in\inn(v)}F(e)
	\]
	at each vertex $v\in V$. A unit flow on a DAG $G$ in the sense of Definition~\ref{defn:flow-polytope} is then a $(1,0,\dots,0,-1)$-flow on the DAG $G'$ obtained by identifying all sources of $G$ to a vertex with netflow 1 and all sinks of $G$ to a vertex with netflow -1 (while giving all internal vertices netflow 0). 
\end{remk}

We call a directed graph \emph{convenient} if every source vertex and every sink vertex is incident to only one edge. If $G$ is an arbitrary directed graph, we may separate the source and sink vertices to obtain a convenient $G'$ with the same flow polyhedron.
Because of this, we will often restrict our attention to convenient directed graphs in this article; this does not restrict the combinatorics of framed directed graphs and it makes the directed graph pictures line up better with fringed algebras and with our novel definition of turbulence charts.

Given a route $p$ of $G$, the \emph{indicator vector} $\I(p)$ with 1's in the coordinates of edges used by $p$ and 0's elsewhere is a vertex of $\F_1(G)$. 
All vertices of $\F_1(G)$ are obtained in this way. 
Moreover, if $G$ is a DAG then $\I$ is a bijection from routes of $G$ to vertices of $\F_1(G)$.

\subsection{Framing triangulations of flow polytopes}
\label{ssec:fdag}

Framed DAGs were defined in~\cite{DKK} and used to induce regular unimodular triangulations on the associated flow polytope.
We give here a more general definition allowing oriented cycles in the underlying graphs.

\begin{defn}\label{defn:framingdag}
	Let $G$ be a directed graph.  For each internal vertex $v$ of $G$, assign a linear order to the edges in $\inn(v)$ and assign a linear order to the edges in $\out(v)$. This assignment is called a \emph{framing} on $G$, which we denote by $\mathfrak{F}$. We call a directed graph (resp. DAG) $G$ with a framing $\mathfrak{F}$ a \emph{framed directed graph} (resp. \emph{framed DAG}) and write $(G,\mathfrak{F})$. At times, we will use the single symbol $\Gamma=(G,\R)$ to refer to a framed DAG. If $e$ is less than $f$ in the linear order for $\mathfrak{F}$ on $\inn(v)$, we write $e<_{\mathfrak{F},\inn(v)}f$ (and similarly for $\out(v)$). When $\mathfrak{F}$ and/or $\inn(v)$ or $\out(v)$ is clear, we may drop one or both subscripts.
\end{defn}

In this paper, to denote a framing we label the internal half-edges of a directed graph with integers. See Figure~\ref{fig:sqdag}.
For the rest of this subsection, we focus on framed DAGs.

\begin{defn}\label{defn:compatible}
	Let $(G,\R)$ be a framed DAG.
	Let $R$ be a path in $(G,\R)$ from $v$ to $w$, where both vertices are internal (we may have $v=w$, in which case $R$ is the empty path from $v$ to $v$).
	Let $e_1$ and $e_2$ be edges of $G$ ending at $v$ such that $e_1<_{\mathfrak{F},\inn(v)}e_2$. Let $f_1$ and $f_2$ be edges of $G$ starting at $w$ such that $f_1<_{\mathfrak{F},\out(v)}f_2$.
	We say that $(e_1Rf_2,e_2Rf_1)$ is an \emph{incompatibility}.
	Two paths (or, indeed, routes) $p$ and $q$ are \emph{incompatible} if there exist subpaths $p'$ of $p$ and $q'$ of $q$ such that $(p',q')$ is an incompatibility. Otherwise, $p$ and $q$ are \emph{compatible}.
\end{defn}

A \emph{clique} is a set of pairwise-compatible routes in $(G,\R)$.
The \emph{clique complex} of $(G,\R)$ is the simplicial complex of cliques of $(G,\R)$.

Given an embedding of a DAG $G$ into the page with all edges moving left to right, one may define the \emph{planar framing} $\R_{\text{pl}}$ which orders the incoming (resp. outgoing) edges at each vertex $v$ from bottom to top in the order $<_{\mathfrak{F}_{\textup{pl}}}$. Cliques of $(G,\R_{\text{pl}})$ are collections of routes which can be drawn on the DAG $G$ without any crossings along edges of $G$ -- see Figure~\ref{fig:sqdag}.

Recall that taking indicator vectors gives a bijection from routes of $G$ onto vertices of $\F_1(G)$. Through this correspondence, we may view a maximal clique of $(G,\R)$ as a collection of vertices of $\F_1(G)$ which form a simplex. The set of such simplices forms a regular unimodular triangulation of $\F_1(G)$, called the \emph{framing triangulation}.
We refer to~\cite[\S2.2]{LRS} for the definition of a regular triangulation.

\begin{thm}[{\cite[Theorem 1, Theorem 2]{DKK}}]
	\label{thm:dkk}
	Let $(G,\R)$ be a framed DAG. The set of maximal cliques of $(G,\R)$ forms a regular unimodular triangulation of the flow polytope $\F_1(G)$.
\end{thm}

See the framed DAG on left of Figure~\ref{fig:sqdag} whose flow polytope is the square in the middle of the figure. There are two maximal cliques; the top induces the red simplex, and the bottom induces the blue simplex.

\begin{figure}
	\centering
	\def\svgscale{.3}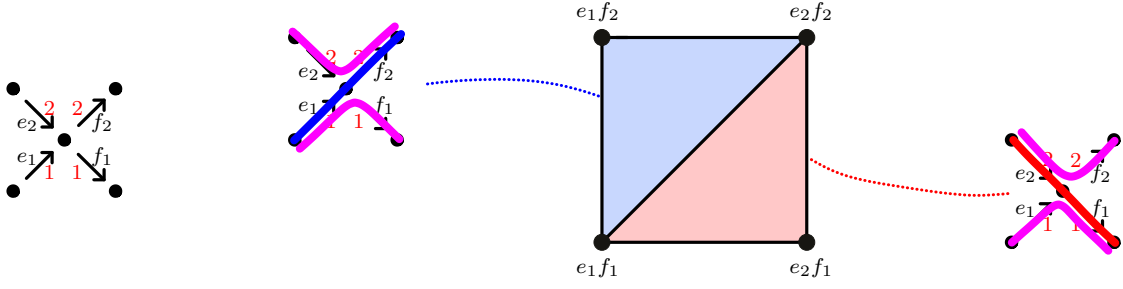
	\caption{A framed DAG with framing marked in red on the left with its maximal cliques and framing-triangulated flow polytope on the right.}
	\label{fig:sqdag}
\end{figure}

\subsection{Bundles on Framed Directed Graphs}
\label{ssec:frameddg}

We now extend the definitions of the previous subsection to framed directed graphs which may have cycles. Framing triangulation results are not known in the literature for general framed directed graphs, but were proven in~\cite{BERG} for a subclass of framed directed graphs known as \emph{gentle}. In Section~\ref{sec:specialization} we will get subdivision and presentation results for general framed directed graphs as a consequence of our results on framed turbulence charts.
Let $(G,\R)$ be a framed directed graph.

\begin{defn}
A \emph{route} of a directed graph is a path from a source vertex do a sink vertex. A \emph{band} $B$ of a directed graph is a path of the form $e_1e_2\dots e_m$, where $h(e_m)=t(e_1)$ and $B$ is not a power of a strictly smaller cycle. We consider bands only up to cyclic equivalence, so that $B$ is equivalent to the cycle $e_je_{j+1}\dots e_me_1\dots e_{j-1}$ for any $j\in[m]$. A subpath of the band $B$ is a subpath of any power of the underlying cycle of $B$. A \emph{trail} of a directed graph is a route or band.
\end{defn}

Two trails $p$ and $q$ are \emph{incompatible} if there exist subpaths $p'$ of $p$ and $q'$ of $q$ which are incompatible as in Definition~\ref{defn:compatible}.
A \emph{clique} is a set of pairwise-compatible routes of $(G,\R)$, and a \emph{bundle} is a set of pairwise-compatible trails of $(G,\R)$ (where bands are considered up to cyclic equivalence).
The \emph{bundle complex} of $(G,\R)$ is the simplicial complex of bundles of $(G,\R)$.

In~\cite{BERG}, it was shown that the clique and bundle complexes of certain framed directed graphs with cycles index subdivisions of the underlying flow polyhedra. We will give more detail about this in Section~\ref{ssec:gentlyframeddag}.

\section{Background on Gentle Algebras and their Turbulence Polyhedra}
\label{sec:back-turb}

We give background on gentle algebras, their fringed algebras, and their turbulence polyhedra, following~\cite{BERG}. Turbulence polyhedra of gentle algebras will be a special case of the turbulence polyhedra defined in this paper.

\subsection{Gentle algebras}

We recall the definition of a gentle algebra from~\cite{AS87}.
A \emph{quiver} is a directed graph.
For an arrow $\alpha$ of a quiver $Q$, we write $h(\alpha)$ for the head (or endpoint) of $\alpha$ and $t(\alpha)$ for the tail (or start point) of $\alpha$. We consider the formal inverse $\alpha^{-1}$ to have head $t(\alpha)$ and tail $h(\alpha)$.

\begin{defn}\label{defn:gentle}
	A finite-dimensional $\k$-algebra $\Lambda$ is \emph{gentle} if it is (Morita equivalent to) a bound path algebra $\k Q/I$, where $Q$ is a quiver and $I$ is an admissible ideal in $\k Q$ such that
	\begin{enumerate}
		\item for each vertex $v\in Q_0$, there are at most two edges starting at $v$ and at most two edges ending at $v$,
		\item for each arrow $\alpha\in Q_1$, there is at most one arrow $\beta$ with $h(\alpha)=t(\beta)$ such that $\alpha\beta\not\in I$, and there is at most one arrow $\gamma$ with $t(\alpha)=h(\gamma)$ such that $\gamma\alpha\not\in I$,
		\item for each arrow $\alpha\in Q_1$, there is at most one arrow $\beta$ with $h(\alpha)=t(\beta)$ such that $\alpha\beta\in I$, and there is at most one arrow $\gamma$ with $t(\alpha)=h(\gamma)$ such that $\gamma\alpha\in I$, and
		\item the ideal $I$ is generated by length-two monomial relations.
	\end{enumerate}
	We say that $(Q,I)$ is a \textit{gentle bound quiver}.
\end{defn}

\begin{remk}\label{remk:locally-gentle}
	A $\k$-algebra $\Lambda$ is \emph{locally gentle} if it satisfies the conditions of Definition~\ref{defn:gentle}, but may or may not be finite-dimensional algebra; i.e., locally gentle algebras are an infinite-dimensional version of gentle algebras. Locally gentle algebras give rise to a non-kissing complex similar to the gentle case~\cite{PPP2019}, and will give examples of our turbulence charts.
\end{remk}

See the left of Figure~\ref{fig:bloss} for an example of a gentle algebra $\L$. The ideal $I$ is generated by the relations $\beta\gamma$ and $\gamma\beta$ of $\L$. The relation $\beta\gamma$ (resp. $\gamma\beta$) is marked visually by a blue arc connecting $\beta$ to $\gamma$ (resp. $\gamma$ to $\beta$) around the vertex $h(\beta)=t(\gamma)$ (resp. $h(\gamma)=t(\beta)$). In the following, we will express the relations of a gentle algebra $\L=\k Q/I$ using this visual convention without writing them in text.

Henceforth, let $\Lambda=\k Q/I$ be a gentle algebra.
We collect some facts about gentle algebras following~\cite{PPP,BERG}.

\begin{defn}
		A \emph{string} of $\Lambda$ is a word of the form $p=\alpha_1^{\e_1}\alpha_2^{\e_2}\dots\alpha_m^{\e_m}$, where
			\begin{enumerate}
				\item $\alpha_i\in Q_1$ and $\e_i\in\{-1,1\}$ for all $i\in[m]$,
				\item $h(\alpha_i^{\e_i})=t(\alpha_{i+1}^{\e_{i+1}})$ for all $i\in[m-1]$,
				\item for any $\beta\gamma\in I$, neither $\beta\gamma$ nor $\gamma^{-1}\beta^{-1}$ appears in $p$, and
				\item no factor $\alpha\alpha^{-1}$ or $\alpha^{-1}\alpha$ appears in $p$, for $\alpha\in Q_1$.
			\end{enumerate}
			The integer $m$ is called the \emph{length} of $p$. We write $t(p):=t(\alpha_1^{\e_1})$ and $h(p):=h(\alpha_m^{\e_m})$. For each vertex $v\in Q_0$, there is a \emph{lazy string} of length zero, denoted by $e_v$, starting and ending at $v$. A \emph{substring} of $p$ is a word of the form $\alpha_i^{\e_i}\dots\alpha_j^{\e_j}$, for some $1\leq i\leq j\leq m$. We also consider a lazy string $e_v$, where $v$ is some vertex through which $p$ passes, to be a substring.
\end{defn}

We consider the the two inverse strings $p$ and $p^{-1}$ to be \emph{equivalent}.
We use the symbol $p^{\pm1}$ to refer to the (equivalent) strings $p$ and $p^{-1}$ simultaneously as an abuse of notation. For example, we may write ``$q$ is a substring of $p^{\pm1}$'' to mean, ``$q$ is a substring of $p$ or $p^{-1}$.''

	The strings of the gentle algebra on the left of Figure~\ref{fig:bloss} up to equivalence are 
	\[e_1\ ,e_2\ ,e_3,\ \alpha,\ \beta,\ \gamma,\ \alpha\beta,\text{ and }\alpha\gamma^{-1}\] (labelling the vertices from 1 to 3, left to right).
	Note, for example, that the string $\alpha$ is equivalent to $\alpha^{-1}$ and $\gamma\alpha^{-1}$ is equivalent to $\alpha\gamma^{-1}$.

In the gentle algebra on the left of Figure~\ref{fig:bloss}, the lazy string $e_2$ at the middle vertex is a bottom substring of $\alpha\gamma^{-1}$ and a top substring of $\beta$.

A \emph{band} of $\Lambda$ is a string $B=\alpha_1^{\e_1}\dots\alpha_m^{\e_m}$ of length at least one such that $h(B)=t(B)$, the walk $\alpha_m^{\e_m}\alpha_1^{\e_1}$ is a string, and $B$ is not itself a power of a strictly smaller string.
The gentle algebra of Figure~\ref{fig:bloss} has no bands.
We consider a \emph{substring} of $B$ to be a substring of any power of $\alpha_1^{\e_1}\dots\alpha_m^{\e_m}$.
Two bands $B$ and $B'$ are \emph{equivalent} if the underlying string of $B'$ is a substring of $B$ or $B^{-1}$ (equivalently, if $B^{\pm1}$ and $B'$ are cyclically equivalent).

Strings and bands of a gentle algebra $\L$ are important because they describe the indecomposable modules over $\L$.
Given a string $p$, one may define the \emph{string module} $M(p)$ over $\Lambda$. We note that $M(p)$ is isomorphic to $M(p^{-1})$, which is why we consider strings up to equivalence.
Similarly, given a band $B$ of a gentle algebra $\Lambda$, a power $n\in\mathbb Z_{\geq0}$, and $\lambda\in\k^\times$, we may define a \textit{band module} $M(B,n,\lambda)$ which depends on $B$ only up to equivalence. The specific constructions of band and string modules are not important in this paper, and hence we refer to~\cite{BR87},~\cite[\S2]{BDMTY},~\cite[\S1]{PPP} for more information on band modules and string modules. For context, we cite the following theorem stating that the combinatorics of strings describes much of the representation theory of $\L$.

\begin{thm}[{\cite[p. 161]{BR87}}]
	The string and band modules are a complete list of the indecomposable $\Lambda$-modules up to isomorphism. Moreover,
	\begin{enumerate}
		\item a string module is never isomorphic to a band module,
		\item two string modules $M(p)$ and $M(p')$ are isomorphic if and only if $p'=p^{\pm1}$, and
		\item two band modules $M(B,n,\lambda)$ and $M(B',n',\lambda')$ are isomorphic if and only if $n=n'$, $\lambda=\lambda'$, and the bands $B$ and $B'$ are equivalent.
	\end{enumerate}
\end{thm}

\subsection{Fringed algebras and the bundle complex of a gentle algebra}\label{ssec:bloss}

It was shown independently in~\cite{BDMTY} and~\cite{PPP} that a gentle algebra $\L$ gives rise to a larger algebra called its fringed algebra $\tL$, such that the $\tau$-tilting theory of $\L$~\cite{AIR} is described by the simplicial complex of maximal strings on $\tL$.

If $v\in Q$, we say that $\textup{indeg}(v)$ is the number of arrows with head $v$ and we say that $\textup{outdeg}(v)$ is the number of arrows with tail $v$. If $\Lambda$ is gentle, then every vertex of $\Lambda$ has $\textup{indeg}(v)\leq2$ and $\textup{outdeg}(v)\leq2$.

\begin{defn}[{\cite[Definition 2.1]{PPP},\cite[Definition 3.2]{BDMTY}}]
	The \emph{fringed algebra} of a gentle algebra $\Lambda=\k Q/I$ is the gentle bound quiver $(\tilde Q,\tilde I)$ obtained by adding at each vertex $v\in Q_0$:
	\begin{itemize}
		\item\label{g1} arrows from new vertices in and out of $v$ such that its in-degree and out-degree both become $2$, and
		\item relations such that vertex $v$ fulfills the gentle bound quiver conditions.
	\end{itemize}
	Any arrow added through this construction is called a \emph{fringe arrow}, and is considered to be incident to $v$ as well as a new \emph{fringe vertex}. Vertices of $\tL$ coming from $\L$ are called \emph{internal vertices}. By construction, all {internal vertices} of $\tL$ have in-degree and out-degree 2, and all fringe vertices are incident to only one (fringe) arrow.
	We use the symbols $V$ and $E$ to refer to the vertices and arrows of $\tL$, respectively.
	The \textit{fringed algebra} of $\Lambda$ is $\tilde\Lambda=\k\tilde Q/\tilde I$.
\end{defn}

See Figure~\ref{fig:bloss} for an example of a gentle algebra and its fringed algebra.
When we draw fringed algebras, we draw fringe vertices as squares.

Abusing notation, we will use the term \emph{fringed algebra} to refer to the fringed algebra $\tL$ as well as its bound quiver $(\tilde Q,\tilde I)$.
We may refer to vertices or arrows of $\tilde Q$ as vertices or arrows of the fringed algebra $\tilde\Lambda$, respectively. Note that each vertex $v$ of $\tilde \Lambda$ corresponding to a vertex of $Q$ is incident to exactly four arrows (counting multiplicities, if it is incident to the same arrow at the head and tail). We call such vertices \emph{internal}. We say that the \emph{relations} of $\tL$ at an internal vertex $v$ are those compositions $\alpha_1\alpha_2$ and $\beta_1\beta_2$ such that $h(\alpha_1)=h(\beta_1)=t(\alpha_2)=t(\beta_2)=v$ and neither $\alpha_1\alpha_2$ nor $\beta_1\beta_2$ is a string of $\tL$. 

\begin{figure}
	\centering
	\def\svgscale{.21}
\begingroup%
  \makeatletter%
  \providecommand\color[2][]{%
    \errmessage{(Inkscape) Color is used for the text in Inkscape, but the package 'color.sty' is not loaded}%
    \renewcommand\color[2][]{}%
  }%
  \providecommand\transparent[1]{%
    \errmessage{(Inkscape) Transparency is used (non-zero) for the text in Inkscape, but the package 'transparent.sty' is not loaded}%
    \renewcommand\transparent[1]{}%
  }%
  \providecommand\rotatebox[2]{#2}%
  \newcommand*\fsize{\dimexpr\f@size pt\relax}%
  \newcommand*\lineheight[1]{\fontsize{\fsize}{#1\fsize}\selectfont}%
  \ifx\svgwidth\undefined%
    \setlength{\unitlength}{1141.06018066bp}%
    \ifx\svgscale\undefined%
      \relax%
    \else%
      \setlength{\unitlength}{\unitlength * \real{\svgscale}}%
    \fi%
  \else%
    \setlength{\unitlength}{\svgwidth}%
  \fi%
  \global\let\svgwidth\undefined%
  \global\let\svgscale\undefined%
  \makeatother%
  \begin{picture}(1,0.16803498)%
    \lineheight{1}%
    \setlength\tabcolsep{0pt}%
    \put(0,0){\includegraphics[width=\unitlength,page=1]{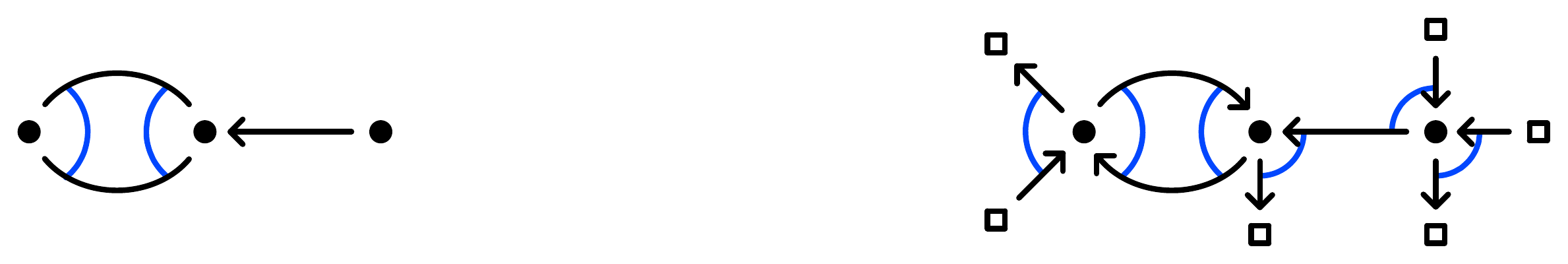}}%
    \put(0.73080615,0.05943422){\color[rgb]{0,0,0}\makebox(0,0)[lt]{\lineheight{1.25}\smash{\begin{tabular}[t]{l}$\scriptstyle{\beta}$\end{tabular}}}}%
    \put(0.72712019,0.13727586){\color[rgb]{0,0,0}\makebox(0,0)[lt]{\lineheight{1.25}\smash{\begin{tabular}[t]{l}$\scriptstyle{\gamma}$\end{tabular}}}}%
    \put(0.84160075,0.09682401){\color[rgb]{0,0,0}\makebox(0,0)[lt]{\lineheight{1.25}\smash{\begin{tabular}[t]{l}$\scriptstyle{\alpha}$\end{tabular}}}}%
    \put(0,0){\includegraphics[width=\unitlength,page=2]{genttoblos.pdf}}%
    \put(0.05774884,0.05944474){\color[rgb]{0,0,0}\makebox(0,0)[lt]{\lineheight{1.25}\smash{\begin{tabular}[t]{l}$\scriptstyle{\beta}$\end{tabular}}}}%
    \put(0.05406279,0.1372855){\color[rgb]{0,0,0}\makebox(0,0)[lt]{\lineheight{1.25}\smash{\begin{tabular}[t]{l}$\scriptstyle{\gamma}$\end{tabular}}}}%
    \put(0.16854335,0.09683453){\color[rgb]{0,0,0}\makebox(0,0)[lt]{\lineheight{1.25}\smash{\begin{tabular}[t]{l}$\scriptstyle{\alpha}$\end{tabular}}}}%
  \end{picture}%
\endgroup%

	\caption{{A gentle algebra (left) and its fringed algebra (right).}}
	\label{fig:bloss}
\end{figure}

\begin{defn}
	A \emph{route} of $\tL$ is a maximal string in $\tilde \Lambda$ (thus beginning and ending at fringe vertices of $\tilde \Lambda$).
	A route is \emph{straight} if it is an oriented walk in $\tilde \Lambda$, and \emph{bending} otherwise.
	A \emph{trail} of $\tL$ is a route or band of $\tL$.
\end{defn}

\begin{defn}\label{defn:tilt-complex}
	Let $\sigma$ be a string of $\tL$ from vertices $v$ to $w$. Let $\alpha\beta$ be a relation at $v$ and $\gamma\delta$ be a relation at $w$ so that $\alpha\sigma\gamma^{-1}$ and $\beta^{-1}\sigma\delta$ are strings. We say that $(\alpha\sigma\gamma^{-1},\beta^{-1}\sigma\delta)$ is a \emph{kiss}, or an \emph{incompatibility}.
	Two trails $p$ and $q$ are \emph{incompatible} if there exist substrings $p'$ of $p$ and $q'$ of $q$ such that $(p',q')$ is an incompatibility. Otherwise, $p$ and $q$ are \emph{compatible}.
\end{defn}

Shown in Figure~\ref{fig:kiss} are two incompatible routes -- the string $\alpha$ is at the top of the red route and at the bottom of the blue route.
\begin{figure}
	\centering
	\def\svgscale{.21}
\begingroup%
  \makeatletter%
  \providecommand\color[2][]{%
    \errmessage{(Inkscape) Color is used for the text in Inkscape, but the package 'color.sty' is not loaded}%
    \renewcommand\color[2][]{}%
  }%
  \providecommand\transparent[1]{%
    \errmessage{(Inkscape) Transparency is used (non-zero) for the text in Inkscape, but the package 'transparent.sty' is not loaded}%
    \renewcommand\transparent[1]{}%
  }%
  \providecommand\rotatebox[2]{#2}%
  \newcommand*\fsize{\dimexpr\f@size pt\relax}%
  \newcommand*\lineheight[1]{\fontsize{\fsize}{#1\fsize}\selectfont}%
  \ifx\svgwidth\undefined%
    \setlength{\unitlength}{426.14199829bp}%
    \ifx\svgscale\undefined%
      \relax%
    \else%
      \setlength{\unitlength}{\unitlength * \real{\svgscale}}%
    \fi%
  \else%
    \setlength{\unitlength}{\svgwidth}%
  \fi%
  \global\let\svgwidth\undefined%
  \global\let\svgscale\undefined%
  \makeatother%
  \begin{picture}(1,0.41174772)%
    \lineheight{1}%
    \setlength\tabcolsep{0pt}%
    \put(0,0){\includegraphics[width=\unitlength,page=1]{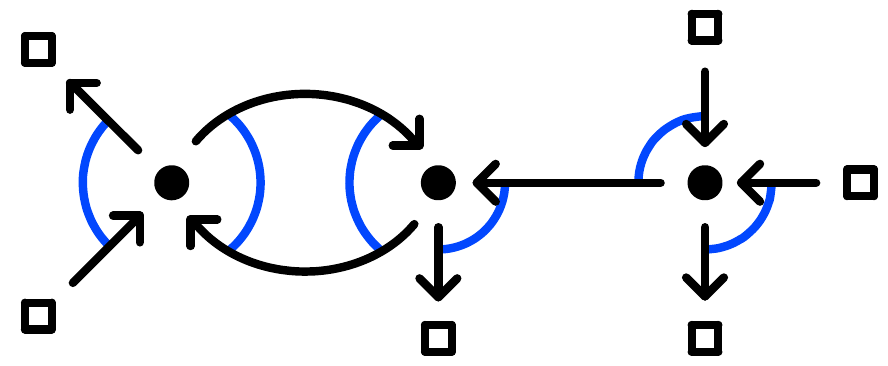}}%
    \put(0.29828508,0.14004721){\color[rgb]{0,0,0}\makebox(0,0)[lt]{\lineheight{1.25}\smash{\begin{tabular}[t]{l}$\scriptstyle{\beta}$\end{tabular}}}}%
    \put(0.28841537,0.34848008){\color[rgb]{0,0,0}\makebox(0,0)[lt]{\lineheight{1.25}\smash{\begin{tabular}[t]{l}$\scriptstyle{\gamma}$\end{tabular}}}}%
    \put(0.5949545,0.24016407){\color[rgb]{0,0,0}\makebox(0,0)[lt]{\lineheight{1.25}\smash{\begin{tabular}[t]{l}$\scriptstyle{\alpha}$\end{tabular}}}}%
    \put(0,0){\includegraphics[width=\unitlength,page=2]{kiss.pdf}}%
  \end{picture}%
\endgroup%

	\caption{Two incompatible routes.}
	\label{fig:kiss}
\end{figure}

A \emph{clique} is a set of pairwise compatible (self-compatible) routes.
The \emph{clique complex} of $\tL$ is the simplicial complex of cliques.
A \emph{bundle} is a set of pairwise compatible (self-compatible) trails. The \emph{bundle complex} of $\tL$ is the simplicial complex of bundles.

The clique complex was defined in~\cite{PPP,BDMTY} as the \emph{non-kissing complex}. We call it the clique complex to differentiate it from the bundle complex, introduced by the author in~\cite{BERG}.
See the right of Figure~\ref{fig:sqgent} for an example of a fringed algebra and its two maximal cliques.

\begin{figure}
	\centering
	\def\svgscale{.3}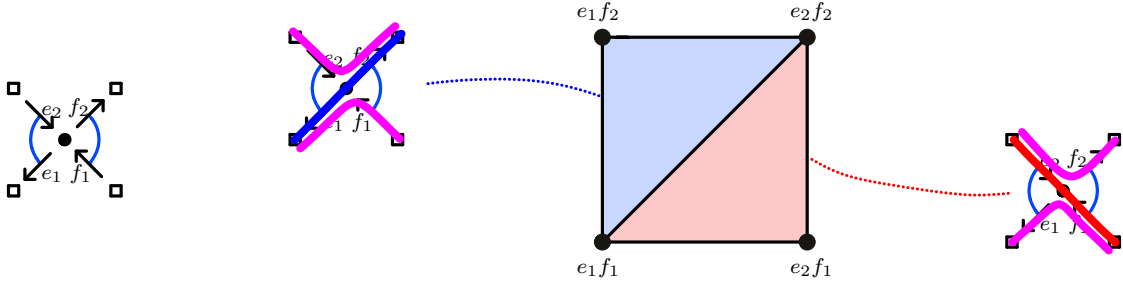
	\caption{A fringed algebra with its two maximal cliques and turbulence polyhedron on the right.}
	\label{fig:sqgent}
\end{figure}

\begin{remk}
	The fringed algebra $\tL$ of a gentle algebra $\L$ and its clique complex were developed in~\cite{PPP,BDMTY} to describe the $\tau$-tilting theory of $\L$, an important generalization of tilting theory introduced in~\cite{AIR}. Self-compatible routes of $\tL$ are in bijection with $\tau$-rigid indecomposable $\L$-modules and shifted projectives, and compatibility of routes describes $\tau$-rigidity of direct sums of the corresponding indecomposables.
\end{remk}

\subsection{Flows and turbulence polyhedra of fringed algebras}
\label{ssec:Bpres}

In~\cite{BERG}, the author introduced the notion of flow on the fringed algebra of a gentle algebra. The unit nonnegative flows make up the turbulence polyhedron, which is triangulated and subdivided by pairwise non-kissing routes and bands. We recall this theory in this subsection, and in the next we will see that it generalizes framing-triangulated flow polytopes of certain framed DAGs.

\begin{defn}\label{defn:flowgent}
	Let $\tL$ be a fringed algebra.
	A function $F:E\to\mathbb R$ is a \emph{flow} on $\tL$ (or on $\L$) if it satisfies
\textbf{conservation of flow:} i.e., if for any internal vertex $v$ of $\tL$ with relations $\alpha_1\alpha_2$ and $\beta_1\beta_2$ we have 
	\[F(\alpha_1)+F(\alpha_2)=F(\beta_1)+F(\beta_2).\]
	$F$ is \emph{nonnegative} if for any arrow $\alpha\in E$, we have $F(\alpha)\geq0$.
	The \emph{strength} of a flow $F$ is the sum $\frac{1}{2}\sum_{\alpha\in E\text{ fringe}}F(\alpha)$. A flow is \emph{unit} if it has strength 1, and it is a \emph{vortex} if it has strength 0.
	The \emph{turbulence polyhedron} $\F_1(\tL)$ is the space of unit nonnegative flows on $\tL$.
\end{defn}

The turbulence polyhedron $\F_1(\tL)$ is a rational polyhedron because it is defined as a finite intersection of rational hyperplanes and half-spaces in $\mathbb R^E$.
Refer to the middle of Figure~\ref{fig:flow} for an example of a flow on a fringed algebra.

\begin{defn}\label{defn:qw}
	Let $W$ be a set of arrows of $\tL$. Define the face $Q_W$ of $\F_1(\tL)$ as
	\[Q_W:=\{F\in\F_1(\tL)\ :\ F(\alpha)=0\text{ for all }\alpha\in W\}.\]
\end{defn}

\begin{lemma}[{\cite[Lemma 5.6]{BERG}}]
	\label{lem:gentface}
	$Q_W$ is a face of $\F_1(\tL)$, and all faces of $\F_1(\tL)$ are of the form $Q_W$.
\end{lemma}

\begin{defn}
	If $p$ is a string of $\tL$, then the \emph{indicator vector} $\I(p)$ is the vector in $\mathbb R^E$ such that the coordinate of an arrow $\alpha$ is the number of times the arrow $\alpha^{\pm1}$ appears in $p$.
	The \emph{indicator vector} of a route or band of $\tL$ is the indicator vector of its underlying string.
\end{defn}

We will use indicator vectors of certain routes and bands to obtain a minimal presentation for the turbulence polyhedron of a fringed algebra.

\begin{defn}\label{defn:elrouteg}
	Let $\tL$ be a fringed algebra.
	We say that a route $p$ of $\tL$ is \emph{simple} if it does not use the same vertex twice. 
	We say that $p$ is a \emph{lollipop} if it is of the form $s\sigma s^{-1}$ for some strings $s$ and $\sigma$ such that every vertex of $s\sigma s^{-1}$ appears exactly twice if it is a vertex of $s$, and once otherwise.
	A self-compatible route is \emph{elementary} if it is simple or it is a lollipop.
\end{defn}

\begin{defn}\label{defn:elbandg}
	Let $\tL$ be a fringed algebra.
	We say that a band $B$ of $\tL$ is \emph{simple} if it does not use the same vertex twice (not counting $t(B)=h(B)$). 
	We say that $B$ is a \emph{barbell} if up to cyclic equivalence it is of the form $s\sigma_1 s^{-1}\sigma_2$ for some strings $s$, $\sigma_1$, and $\sigma_2$ such that every vertex used appears exactly twice if it is a vertex of $s$, and once otherwise.
	A band is \emph{elementary} if it is simple or it is a barbell.
\end{defn}

\begin{remk}
	In~\cite{BERG}, elementary bands were required to be self-compatible, but we drop that requirement here. This is because we will extend this to a notion of elementary trails on turbulence charts, and we do not want the notion of an elementary trail to depend on the choice of framing.
	Note that for each self-compatible barbell $s\sigma_1 s^{-1}\sigma_2$, there is a non-self-compatible barbell $s\sigma_1 s^{-1}\sigma_2^{-1}$.
\end{remk}

See Figure~\ref{fig:ELEMENTARY} for a pictorial summary of elementary and nonelementary trails. The following result shows that elementary routes and bands of a fringed algebra give the vertices and unbounded directions of its turbulence polyhedron.

\begin{thm}[{\cite[Theorem A]{BERG}}]\label{thm:BPres}
	The map $p\mapsto\I(p)$ bijects elementary routes to vertices of $\F_1(\tL)$. The map $B\mapsto\I(B)$ bijects self-compatible elementary bands to the elementary rays of $\F_1(\tL)$.
\end{thm}

\subsection{Subdivisions of turbulence polyhedra of fringed algebras}
\label{ssec:Bsubd}

We now recall subdivision results on turbulence polyhedra of fringed algebras.

\begin{defn}
	Let $\bK=\K\cup \B$ be a bundle. A \emph{$\bK$-bundle combination} (of $\tL$) is a linear combination
	\[F=\sum_{p\in \bK}a_p\I(p)\]
	of indicator vectors of trails, such that each $a_p$ is nonnegative. The \emph{strength} of the bundle combination is $\sum_{p\in \K}a_p$ (note we iterate only over routes $\K$, not all trails $\bK$). Note that this is the strength of the resulting flow $F$. The bundle combination is \emph{unit} if its strength is 1. In other words, a unit $\bK$-bundle combination is a convex combination of indicator vectors of routes in ${\K}$ plus a nonnegative combination of indicator vectors of bands in $\B$.
	It is a \emph{vortex} if its strength is 0 -- i.e., if it is a nonnegative combination of indicator vectors of bands in $\B$, and no routes.
	A bundle combination is a \emph{clique combination} if $a_p=0$ for every $B\in \B$. It is \emph{positive} if $a_p$ is positive for every $p\in \bK$.
\end{defn}

If $\bK$ is a bundle, then the \emph{bundle cone} $\D_{\geq0}(\bK)$ is the polyhedron of nonnegative $\bK$-bundle combinations and the \emph{bundle simplihedron} $\D_{1}(\bK)$ is the polyhedron of unit $\bK$-bundle simplihedra. The maximal bundle simplihedra are precisely the bundle simplihedra of maximal bundles.
When $\bK$ is a clique, $\D_1(\bK)$ is a simplex and we call it a \emph{clique simplex}.
If $\bK$ is a bundle which is not a clique, then we call $\D_1(\bK)$ its \emph{(unit) bundle wall}.
Every maximal clique simplex is full-dimensional in $\F_1(\tL)$, and no maximal bundle wall is full-dimensional in $\F_1(\tL)$.

\begin{thm}[{\cite[Theorem B]{BERG}}]\label{thm:BClique}
	Any nonnegative flow has at most one representation as a positive clique combination.
	A dense subset of $\F_{\geq0}(\tL)$ may be obtained as a clique combination. If $F\in\F_{\geq0}(\tL)$ is an integer flow, then its clique combination has integral coefficients.
	Consequently, the \emph{clique triangulations}
	\begin{align*}
		\mathcal T_{\geq0}(\tL):&=\{\D_{\geq0}(\K)\ :\ \K\text{ is a maximal clique of }\tL\}\text{ and }\\
		\mathcal T_1(\tL):&=\{\D_1(\K)\ :\ \K\text{ is a maximal clique of }\tL\}
	\end{align*}
	are unimodular triangulations of $\F_{\geq0}(\tL)$ and $\F_1(\tL)$, respectively.
\end{thm}

The unimodular triangulation of Theorem~\ref{thm:BClique} is complete when $\tL$ is representation-finite, but is never complete when $\tL$ is representation-infinite. In~\cite{BERG}, we gave two larger subdivisions of $\F_1(\tL)$ given by adding lower-dimensional walls to the clique triangulation. First, we add the maximal bundle walls.

\begin{thm}[{\cite[Theorem C]{BERG}}]\label{thm:BBundle}
	Any nonnegative flow has at most one representation as a positive bundle combination.
	A dense subset of $\F_{\geq0}(\tL)$ including all rational flows may be obtained as a bundle combination. If $F\in\F_{\geq0}(\tL)$ is an integer flow, then its bundle combination has integral coefficients.
	Consequently, the \emph{bundle subdivisions}
	\begin{align*}
		\mathcal S_{\geq0}(\tL):&=\{\D_{\geq0}(\bK)\ :\ \bK\text{ is a maximal bundle of }\tL\} \text{ and}\\
		\mathcal S_{1}(\tL):&=\{\D_1(\bK)\ :\ \bK\text{ is a maximal bundle of }\tL\}
	\end{align*}
	formed by adding the bundle walls to the clique triangulation are simplicial subdivisions of $\F_{\geq0}(\tL)$ and $\F_1(\tL)$, respectively, which cover every rational point.
\end{thm}

\subsection{Gentle framed directed graphs and subdivided flow polyhedra}
\label{ssec:gentlyframeddag}

Recall Section~\ref{ssec:fdag}, where framed DAGs gave rise to regular unimodular triangulations on the associated flow polytopes. The theory of turbulence polyhedra of gentle algebras allows one to extend these results to a certain class of framed directed graphs which may have oriented cycles.

We first isolate the class of fringed algebras whose turbulence polyhedra have been shown to be flow polyhedra.

\begin{defn}
	A gentle algebra $\Lambda=\k Q/I$ is \textit{paired} if there is a map $\psi$ from the arrows of $Q$ to $\{1,2\}$ such that 
		 if $\alpha,\beta\in Q_1$ with $h(\alpha)=t(\beta)$ then $\alpha\beta\in I\iff \psi(\alpha)\neq\psi(\beta)$.
	We say that $\psi$ is a \emph{pairing function} of $\Lambda$.
\end{defn}

It is immediate that $\Lambda$ is paired if and only if $\tL$ is paired.

We will connect paired fringed algebras with a special class of framed directed graphs which we define now.
Recall the definition of a framed directed graph from Definition~\ref{defn:framingdag}.

\begin{defn}\label{defn:amply-framed-dg}
	A directed graph is \emph{full} if every internal vertex is incident to precisely two outgoing edges and two incoming edges.
	A \emph{gentle framed directed graph} $(G,\mathfrak{F})$ is a framed directed graph $G$ such that
	\begin{enumerate}
		\item the directed graph $G$ is full,
		\item $G$ has no edges directly from source to sink,
		\item there is a (necessarily unique) map $\psi_\mathfrak{F}:E\to\{1,2\}$ such that, if ${e}$ and ${f}$ are edges with the same internal source or target, we have ${e}<_{\mathfrak{F}}{f}$ (in the partial order of $\inn(h({e}))$ or $\out(t({e}))$) if and only if $\psi_\mathfrak{F}({e})<\psi_\mathfrak{F}({f})$, and
		\item\label{dgh3} if $e_1e_2\dots e_m$ is an oriented cycle of $G$, then the image of $\{e_1,\dots,e_m\}$ through $\psi_\mathfrak{F}$ is $\{1,2\}$.
	\end{enumerate}
\end{defn}

If $\psi_{\mathfrak{F}}(e)=1$, we call $e$ a \emph{1-edge}; otherwise, it is a \emph{2-edge}.
We now give a map from (convenient) gentle framed directed graphs to paired gentle algebras.

\begin{defn}\label{defn:gent-from-dg}
	Let $\Gamma=(G,\mathfrak{F})$ be a convenient gentle framed directed graph. We associate to $\Gamma$ a fringed algebra $\tL_\Gamma$ as follows.
	The vertices of $\tL_\Gamma$ are in bijection with vertices of $\Gamma$. For each edge $e$ of $\Gamma$ with $\psi_{\mathfrak{F}}(e)=1$, there is an arrow $\tilde{e}:t({e})\to h({e})$ of $\tL_\Gamma$. For each edge ${f}$ of $\Gamma$, there is an arrow $\tilde{f}:h({f})\to t({f})$ of $\tL_\Gamma$. The relations of $I$ are all pairs of the form $\tilde{e}\tilde{f}$, where either ${e}$ is a 1-edge and ${f}$ is a 2-edge, or ${e}$ is a 2-edge and ${f}$ is a 1-edge. Then $\tL_\Gamma$ is paired, realized by the pairing function $\psi_{\tL_\Gamma}$ from arrows of $\tL$ to $\{1,2\}$ defined by $\psi_{\tL_\Gamma}(\tilde{e})=\psi_{\mathfrak{F}}({e})$.
\end{defn}

Both framed directed graphs on the right of Figure~\ref{fig:flow} map to the fringed algebra in the middle of Figure~\ref{fig:flow} through the map of Definition~\ref{defn:gent-from-dg}.

It is easy to check that the flow polyhedron of $\Gamma$ is equal to the turbulence polyhedron of $\tL_\Gamma$.
Note that a route (resp. band) of $\Gamma$ naturally gives rise to a route (resp. band) of $\tL_\Gamma$. This gives a bijection between trails of $\Gamma$ (up to equivalence) and trails of $\tL_\Gamma$ (up to equivalence). Moreover, this correspondence respects compatibility, showing that the bundle complex of $\Gamma$ is equal to the bundle complex of $\tL_\Gamma$. This is covered by the following result.

\begin{prop}[{\cite[Proposition 4.29]{BERG}}] 
	\label{prop:gentframe}
	The map of Definition~\ref{defn:gent-from-dg} is a bijection from convenient gentle framed directed graphs to paired fringed algebras with a choice of pairing function. Moreover, $\F_1(\Gamma)=\F_1(\tL_\Gamma)$ and the natural correspondence of trails of $\Gamma$ and $\tL_\Gamma$ realizes an isomorphism between the bundle complex of $\Gamma$ and the bundle complex of $\tL_\Gamma$.
\end{prop}

Proposition~\ref{prop:gentframe} states that fringed algebras and their bundle-subdivided turbulence polyhedra generalize gentle framed DAGs and their framing-triangulated flow polytopes.
For example, note the similarity between Figure~\ref{fig:sqdag} (featuring the bundle complex of a gentle framed DAG and the triangulated flow polytope) and Figure~\ref{fig:sqgent} (featuring the bundle complex of the analogous fringed algebra and its triangulated turbulence polytope).

\section{Turbulence Charts}
\label{sec:tc}

In Section~\ref{sec:back-turb}, we saw that fringed algebras generalize gentle framed DAGs in a way which respects the clique complex and flow polytope.
On the other hand, general framed DAGs also generalize gentle framed DAGs. We now define \emph{(framed) turbulence charts} and equip them with a bundle complex and turbulence polyhedron.
We will argue in future sections that that framed turbulence charts are the common generalization of framed DAGs and fringed algebras by extending presentation and subdivision results from these settings.

\subsection{Turbulence charts and turbulence polyhedra}
\label{ssec:turbchart}

In the following, we will deal heavily with undirected graphs.
We will assume throughout that all undirected graphs are finite, and that every vertex of an undirected graph is incident to at least one edge.

We call a vertex of an undirected graph $G$ \emph{internal} if it is incident to two or more edges (counting multiplicity), and \emph{fringe} if it is incident to only one edge.
An edge is \emph{fringe} if it is incident to a fringe vertex, and otherwise is \emph{internal}.
A \emph{half-edge} of $G$ at a vertex $v$ is a tuple $(e,v)$, where $e$ is an edge of $G$ incident to a vertex $v$. 
If $e$ starts and ends at the same vertex, then we still consider $e$ to be a part of two distinct half-edges.
Then every edge of $G$ is a part of exactly two half-edges. 
	In the following, we consider any undirected graph to have vertex set $V$, internal vertex set $V_{\textup{int}}$, and edge set $E$. 

\begin{defn}
	A \emph{turbulence chart} is a tuple $(G,\sim)$ where $G$ is an undirected graph and 
	$\sim:=\{\sim_v\ :\ v\in V_{\textup{int}}\}$ is the data of an equivalence relation $\sim_v$ splitting the half-edges of $G$ at $v$ into exactly two nonempty equivalence classes for every internal vertex $v\in V_{\textup{int}}$.
\end{defn}

See Figure~\ref{fig:turbchartexample} for an example of a turbulence chart. Fringe vertices are drawn as squares, and blue lines separate the equivalence classes of $\sim_v$ at each internal vertex $v$.

\begin{defn}\label{defn:tchart-poly}
	A function $F:E\to\mathbb R$ is a \emph{nonnegative flow} on $(G,\sim)$ if it satisfies
	\textbf{conservation of flow:} for any internal vertex $v$ of $(G,\sim)$, notating the equivalence classes of $\sim_v$ as $A$ and $B$ we have
			\[
				\sum_{(e,v)\in A}F(e)=\sum_{(e,v)\in B}F(e).
			\]
			A flow $F$ is \emph{nonnegative} if $F(e)\geq0$ for every edge $e\in E$.
	The \emph{strength} of a flow $F$ is the sum $\frac{1}{2}\sum_{\alpha\in E\text{ fringe}}F(\alpha)$. A flow is \emph{unit} if it has strength 1. A flow is a \emph{vortex} if it has strength 0.
	The \emph{cone of nonnegative flows} $\F_{\geq0}(G,\sim)$ is the space of nonnegative flows on $(G,\sim)$. The \emph{turbulence polyhedron} $\F_1(G,\sim)$ is the space of unit nonnegative flows on $(G,\sim)$.
\end{defn}

Both $\F_1(G,\sim)$ and $\F_{\geq0}(G,\sim)$ are rational polyhedra, since they are defined as finite intersections of rational hyperplanes and half-spaces in $\mathbb R^E$.
For brevity, we will restrict most of our discussion to $\F_1(G,\sim)$, though analogous results will hold for $\F_{\geq0}(G,\sim)$.
A turbulence chart with a unit flow labelled in blue is given in the left of Figure~\ref{fig:flow}.

\begin{figure}
	\centering
	\def\svgscale{.3}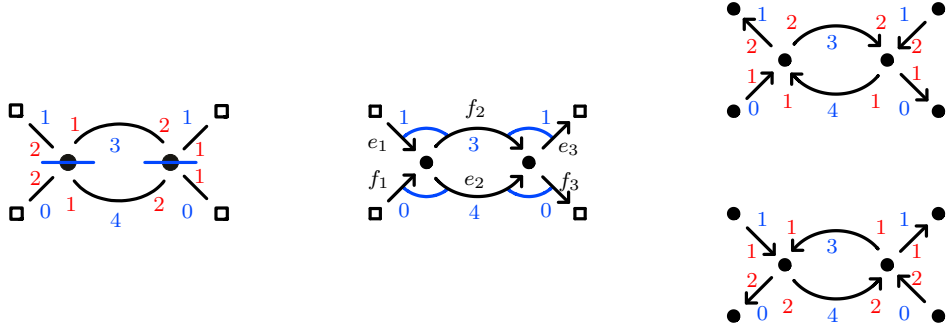
	\caption{A directable gentle turbulence polyhedron (left) with the corresponding fringed algebra (middle) and gentle framed directed graphs (right). Framings are labelled in red and the same unit flow is labelled in blue on all four diagrams.}
	\label{fig:flow}
\end{figure}

Given a convenient directed graph $G$, one may obtain a turbulence chart by taking the undirected graph of $G$ and letting $\sim_v$ separate the (formerly) incoming half-edges from the (formerly) outgoing half-edges. Note that the turbulence polyhedron of the resulting turbulence chart is equal to the flow polyhedron of $G$. For example, the two directed graphs on the right of Figure~\ref{fig:flow} both give rise to the turbulence chart on the left of the same figure. 
Similarly, a fringed algebra gives rise to a turbulence chart by separating the half-edges of the two relations through any internal vertex, and this preserves the turbulence polyhedron. 
See the left and middle of Figure~\ref{fig:flow}.
In Section~\ref{sec:gda} we will expand on the correspondence between turbulence charts and directed graphs, and between turbulence charts and fringed algebras.

\subsection{Routes and bands}

We wish to define routes and bands on a turbulence chart, generalizing routes on a DAG and trails on a fringed algebra.
It will be necessary to work with ``oriented walks'' along an unoriented graph. To do this, we need the notation of oriented edges.

\begin{defn}
	Define an \emph{oriented edge} of an undirected graph $G$ as a double of the form $\te:=e^te^h$, where $e^t$ and $e^h$ are the two distinct half-edges of one (unoriented) edge $e$. We consider $\te$ to \emph{start} at the vertex $t(\te)$ of $e^t$ and \emph{end} at the vertex $h(\te)$ of $e^h$. The \emph{inverse edge} $\te^{-1}$ is the oriented edge $e^he^t$.
\end{defn}

\begin{defn}
	A \emph{string} on a turbulence chart $(G,\sim)$ of length $m\geq0$ is a sequence $s=\te_1\te_2\dots\te_m$, where each $\te_j$ is an oriented edge of $G$ and for any $j\in[m-1]$, we have $h(\te_j)=t(\te_{j+1})$ and the half-edges $e_j^h$ and $e_{j+1}^t$ are in different equivalence classes of $\sim_{h(\te_j)}$.
	We say that $s$ \emph{starts} at $t(s):=t(\te_1)$ and \emph{ends} at $h(s):=h(\te_m)$.
	The \emph{inverse string} $s^{-1}$ is the sequence $\te^{-1}_{m}\te^{-1}_{m-1}\dots\te^{-1}_1$.
	For any $1\leq a\leq b\leq m$, the string $\te_a\te_{a+1}\dots\te_b$ is a \emph{substring} of $s$. 

	For any vertex $v$, there is a \emph{lazy string} $e_v$ which we think of as an empty walk starting and ending at the vertex $v$. We have $e_v=e_v^{-1}$. If $s=\te_1\te_2\dots\te_m$ is a string with $m\geq0$, then we consider $e_{t(\te_1)}$, $e_{t(\te_2)}$, \dots, $e_{t(\te_m)}$, and $e_{h(\te_m)}$ to be substrings of $s$.
\end{defn}

Under our conventions for drawing turbulence charts, a string is a walk on the graph subject to the rule that each time one reaches an internal vertex one may only continue by crossing the blue line at that vertex. See the right of Figure~\ref{fig:turbchartexample}.

\begin{defn}
	A \emph{route} on $(G,\sim)$ is a string which starts and ends at fringe vertices. 
	We consider a route $p$ to be \emph{equivalent} to $p^{-1}$.
	A \emph{band} on $(G,\sim)$ is a string $B$ such that $B^2:=B\circ B$ is a string and that $B$ is not a power of any strictly smaller string $B'$.
	A \emph{substring} of the band $B$ is a substring of any power of its underlying string.
	Two bands $B$ and $B'$ are \emph{equivalent} if the underlying string of $B'$ is a substring of $B$ or $B^{-1}$.
	A \emph{trail} of $(G,\sim)$ is a route or band of $(G,\sim)$.
	The turbulence chart $(G,\sim)$ is \emph{acyclic} if there are no bands of $(G,\sim)$.
\end{defn}

See the right of Figure~\ref{fig:turbchartexample} for an example of two routes and a band on a turbulence chart.
Our depictions of trails do not include a direction and our depictions of bands do not mark a start point, as we only care about trails up to equivalence.

\begin{figure}
	\centering
	\def\svgscale{0.21}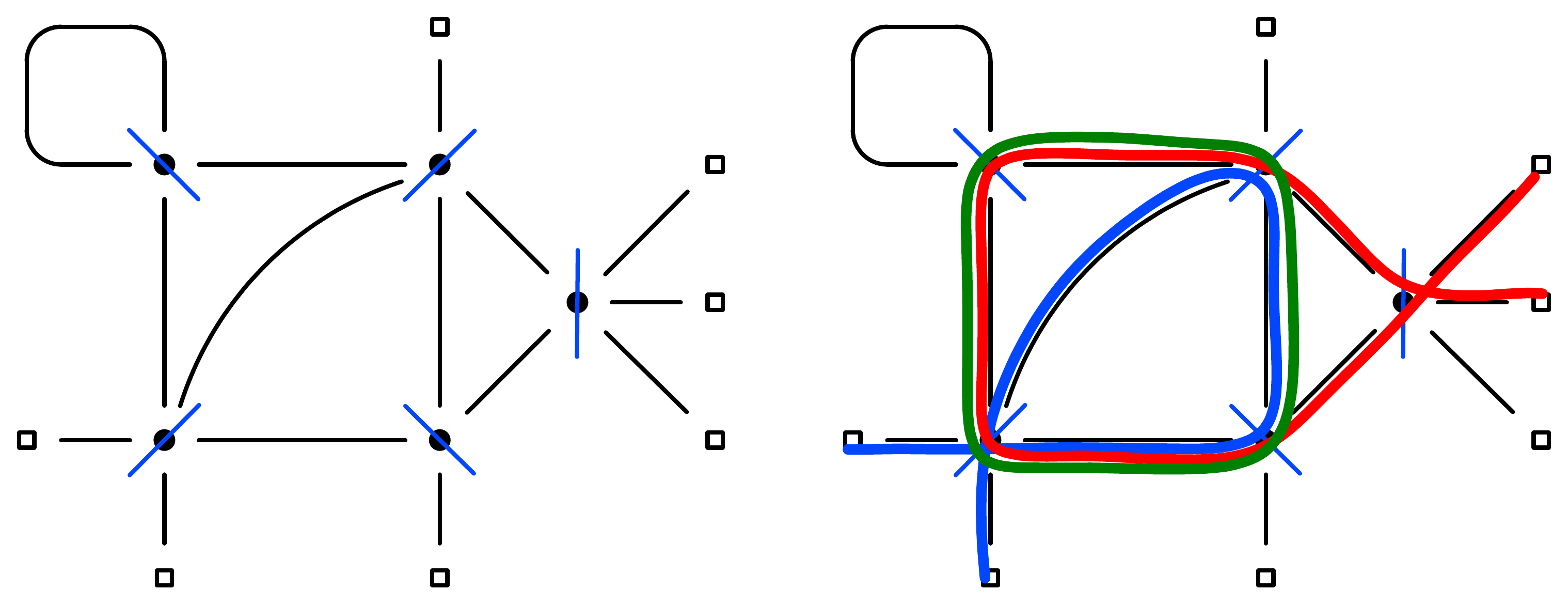
	\caption{On the left is a turbulence chart. On the right is the same turbulence chart with two routes (red and blue) and one band (green).}
	\label{fig:turbchartexample}
\end{figure}

\begin{defn}
	Let $p=\te_1\te_2\dots\te_m$ be a route or band of $(G,\sim)$. The \emph{indicator vector} $\I(p)$ is the map $E\to\mathbb R$ sending any edge $e$ to the number of indices $j\in[m]$ such that the underlying edge of $\te_j$ is $e$ (i.e., to the multiplicity of $e$ in the walk of $p$).
\end{defn}

It is immediate that if $p$ is a route of $(G,\sim)$ then $\I(p)$ is a unit flow, and if $B$ is a band of $(G,\sim)$ then $\I(p)$ is a vortex.

\subsection{Framed turbulence charts and the bundle complex}

We have seen that turbulence charts and their turbulence polyhedra generalize directed graphs and their flow polyhedra. We now generalize framings on directed graphs and their induced notions of compatibility.

\begin{defn}
	Let $(G,\sim)$ be a turbulence chart. Let $v$ be an internal vertex of $G$ and let $S_1$ and $S_2$ be the equivalence classes of half-edges at $v$ given by $\sim_v$. Assign separate linear orders to $S_1$ and $S_2$. This data, ranging over all internal vertices of $G$, is called a \emph{framing} $\R$ on $(G,\sim)$. We write $(G,\sim,\R)$ to denote a \emph{framed turbulence chart}. If $h_1$ and $h_2$ are half-edges in the same equivalence class of $\sim_v$ for some internal vertex $v$, we write $h_1<_{\R}h_2$ to represent that $h_1$ is lesser in the relevant order of $\R$.
\end{defn}

To denote a framing, we label the half-edges of a turbulence chart at internal vertices with red integers.
See Figure~\ref{fig:framedturbchartexample} for an example of a framed turbulence chart.

\begin{defn}
	Let $v$ and $w$ be internal vertices of $G$.
	Let $\te_1$ and $\tilde e_2$ be oriented edges ending at $v$ such that $\te_1^h$ and $\te_2^h$ are equivalent in $\sim_v$ and such that $\te_1^h<_{\R}\te_2^h$.
	Let $\tf_{1}$ and $\tf_{2}$ be oriented edges starting at $w$ such that $\tf_{1}^t$ and $\tf_{2}^t$ are equivalent in $\sim_w$ and such that $\tf_{1}^t<_{\R}\tf_{2}^t$.
	Choose any string $s$ (which may be a lazy string) so that $\te_1 s\tf_{1}$ is a string; then $\te_2 s\tf_{2}$ is also a string.
	We say that $(\te_1 s\tf_{1},\te_2 s\tf_{2})$ is an \emph{incompatibility}.
	Two trails $p$ and $q$ are \emph{incompatible} if without loss of generality there is a substring $p'$ of $p$ and $q'$ of $q$ such that $(p',q')$ is an incompatibility. Otherwise, they are \emph{compatible}.
	A \emph{clique} of $(G,\sim,\R)$ is a set of pairwise compatible routes (up to equivalence of routes).
	A \emph{bundle} of $(G,\sim,\R)$ is a set of pairwise compatible trails (up to equivalence of trails). The \emph{bundle complex} of $(G,\sim,\R)$ is the simplicial complex of bundles of $(G,\sim,\R)$.
\end{defn}

For example, in the middle of Figure~\ref{fig:framedturbchartexample}, the red and green routes agree along their shared substring and the red route leaves with half-edges lower than the green on both sides, hence they are incompatible.

We note that pairwise compatibility of trails of a bundle requires also that every trail of a bundle must be compatible with itself.

Given a drawing of a turbulence chart $(G,\sim)$, define the \emph{clockwise framing} by ordering half-edges clockwise around each internal vertex as in Figures~\ref{fig:trapezoid} and~\ref{fig:bowtie}. In this case, a trail is self-compatible if it can be drawn with no self-crossings, and a bundle is a collection of such trails which can be drawn together with no crossings. This is analogous to the notion of a \emph{planar framing} of an embedding of a DAG~\cite{vBC}, which orders edges bottom-to-top at all internal vertices.

\begin{remk}
	Intuitively, two trails $p$ and $q$ are incompatible if there exists a common (possibly lazy) substring of $p$ and $q$ which, without loss of generality, $p$ departs from $q$ low relative to $\R$ on both sides. We remark that this is opposite to the conventions for compatibility of framed DAGs: two routes $p$ and $q$ of a framed DAG are incompatible if they have a maximal shared subpath which (without loss of generality) $p$ enters low and leaves high relative to $q$.
	When we map framed DAGs to framed turbulence charts, we will reverse all outgoing framing orders (while keeping incoming framing orders the same); this will preserve compatibility of routes and map the planar framing to the clockwise framing.
\end{remk}

	\begin{figure}
		\centering
		\def\svgscale{0.21}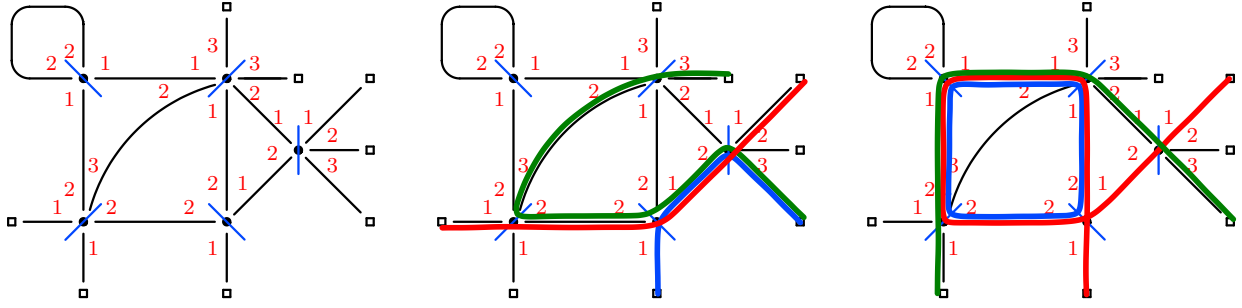
		\caption{On the left is a framed turbulence chart. The middle is not a bundle because the red route and green route are incompatible, and the right is not a bundle because the red route is incompatible with the blue band.}
		\label{fig:framedturbchartexample}
	\end{figure}

\section{Gentle, Directable, and Acyclic Turbulence Charts}
\label{sec:gda}

We now isolate three properties which a framed turbulence chart may have: directability, acyclicity, and gentleness. Directable turbulence charts are those which come from directed graphs. Gentle framed turbulence charts are those which come from fringed algebras. Acyclic turbulence charts are those with no bands -- equivalently, those charts giving bounded turbulence polyhedra. Acyclicity of a turbulence chart amounts to acyclicity of an associated directed graph (in the directable case) and representation-finiteness of an associated fringed algebra (in the gentle case).
Note that directability and acyclicity are really properties of the turbulence chart which do not depend on the framing, while gentleness does depend on the framing.
These three properties may be independently varied; see Figure~\ref{fig:CUBEOFGEN}.
We finish by discussing ampleness of framed turbulence charts, which is a slight generalization of gentleness studied in the framed DAG literature.

\subsection{Gentle turbulence charts}

We now connect certain framed turbulence charts to fringed algebras of gentle algebras.

\begin{defn}
	A turbulence chart $(G,\sim)$ is \emph{sub-full} if for every internal vertex $v$, each equivalence class of $\sim_v$ has cardinality one or two. Moreover, $(G,\sim)$ is \emph{full} if for every internal vertex $v$, each equivalence class of $\sim_v$ has cardinality exactly two. Note that this forces the degree of each internal vertex to be 4.
\end{defn}

\begin{defn}\label{defn:lonely}
	Let $(G,\sim,\R)$ be a framed turbulence chart.
	A half-edge $e$ of $G$ at a vertex $v$ is \emph{lonely} if $v$ is a fringe vertex or if $e$ is the only member of its equivalence class in $\sim_v$.
\end{defn}

\begin{defn}\label{defn:l}
	Let $(G,\sim,\R)$ be a sub-full framed turbulence chart.
	We say that a half-edge $h$ of $G$ at $v$ is \emph{high} if there is another half-edge of $G$ in the same equivalence class of $\sigma_v$ which is under $h$ in $<_{\R}$. Similarly, a half-edge $h$ of $G$ is \emph{low} if there is another half-edge of $G$ at $v$ above $h$ in $<_{\R}$.
	If a half-edge of $G$ at $v$ is fringe or is the only member of its equivalence class of $\sim_v$, then it is lonely and we consider it neither high nor low.
	If $\te=e^te^h$ is an oriented edge such that $e^t$ is low or lonely and $e^h$ is high or lonely, then we say that $\te$ is \emph{ascending}. Similarly, if $e^t$ is high or lonely and $e^h$ is low or lonely, then $\te$ is \emph{descending}.
	Note that if both half-edges are lonely, then $\te$ is both ascending and descending.
	If $\te$ is ascending and/or descending, then we say that the (unoriented) edge $e$ is \emph{steep}. Note that all fringe half-edges are lonely, so all fringe edges are steep.
\end{defn}

We may always consider a framing $\R$ on a full turbulence chart $(G,\sim)$ to be given by a labelling of the half-edges in $\{1,2\}$. In this case, the high half-edges are those which are forced to be labelled with 2 and the low half-edges are those which are forced to be labelled with 1.
Lonely half-edges may be given either label, and we may drop their labels in figures.

Given a framed full turbulence chart all of whose edges are steep, we wish to turn it into a gentle algebra by orienting each edge from its low half-edge to its high half-edge. For this to work, we need to assume some extra conditions as in the following definition.

\begin{defn}\label{defn:gentchart}
	A framed turbulence chart $(G,\sim,\R)$ is \emph{gentle} if
	\begin{enumerate}
		\item\label{gc1} $G$ is full,
		\item\label{gc1.5} every edge of $G$ is incident to at least one internal vertex,
		\item\label{gc2} every edge of $G$ is steep, and
		\item\label{gc3} no band of $(G,\sim)$ consists entirely of ascending oriented edges or entirely of descending oriented edges.
	\end{enumerate}
	We may also say that $\R$ is a \emph{gentle framing} on $(G,\sim)$.
	We may also say that the turbulence polyhedron $\F_1(G,\sim)$ is \emph{gentle} in this case.
	A framed turbulence chart satisfying (1), (2), and (3), but not necessarily (4), is \emph{locally gentle}.
\end{defn}

The following lemma is not strictly used in the following of the paper, but is an important observation regardless.

\begin{lemma}\label{lem:gentfringe}
	If $(G,\sim,\R)$ is a gentle turbulence chart, then every connected component of $G$ contains at least one fringe edge.
\end{lemma}
\begin{proof}
	Let $C$ be a connected component of $G$. Recall that we assume throughout that every vertex of $G$ is incident to at least one edge, so choose an oriented edge $\te$ of $C$. By Definition~\ref{defn:gentchart}~\eqref{gc2}, the oriented edge $\te$ is ascending or descending; if necessary, replace $\te$ with $\te^{-1}$ so that it is ascending. 

	Set $\te_0:=\te$.
	Now, suppose we have chosen some ascending oriented edge $\te_j$.
	If $\te_j$ is fringe, then we have shown that $C$ contains a fringe edge and we are done. Otherwise, since $G$ is full there are exactly two half-edges at the internal vertex $v_j:=h(\te_j)$ not equivalent to $\te_j^h$ in $\sim_v$; define $\te_{j+1}$ to be the oriented edge beginning with the lesser of these two half-edges. In other words, $\te_{j+1}$ is the ascending oriented edge starting at $v$ such that $\te_j\te_{j+1}$ is a string.

	Suppose at some point that $\te_a=\te_b$ as oriented edges for $a<b$. Then $\te_a\te_{a+1}\dots\te_{b-1}$ is a band of $(G,\sim,\R)$ consisting only of ascending edges, contradicting Definition~\ref{defn:gentchart}~\eqref{gc3}. Because $G$ is finite, this shows that the process must terminate with some fringe edge $\te_j$ and the proof ends.
\end{proof}

We now biject gentle framed turbulence charts with fringed algebras of gentle algebras.

\begin{defn}\label{defn:charttogent}
	Let $(G,\sim,\R)$ be a gentle framed turbulence chart.
	We define a fringed algebra $\tL(G,\sim,\R)$ as follows
	\begin{enumerate}
		\item The internal (resp. fringe) vertices of $\tL(G,\sim,\R)$ are the internal (resp. fringe) vertices of $\tL(G,\sim,\R)$.
		\item For each internal edge $e$ of $G$, there is an arrow $\alpha_e$ of $\tL(G,\sim,\R)$ oriented from the low half-edge of $e$ to the high half-edge of $e$.
		\item For each fringe edge $e$ of $G$, 
			\begin{itemize}
				\item if the internal half-edge of $e$ is high then there is an arrow $\alpha_e$ of $\tL(G,\sim,\R)$ oriented from the fringe half-edge of $e$ to the high half-edge of $e$, and
				\item if the internal half-edge of $e$ is low then there is an arrow $\alpha_e$ of $\tL(G,\sim,\R)$ oriented from the low half-edge of $e$ to the fringe half-edge of $e$.
			\end{itemize}
		\item The relations of $\tL(G,\sim,\R)$ are the paths $\alpha_e\alpha_f$ such that the high half-edge of $e$ and the low half-edge of $f$ are at the same vertex $v$ and equivalent in $\sim_{v}$.
	\end{enumerate}
\end{defn}
We now go in the other direction:
\begin{defn}\label{defn:genttochart}
	Let $\tL$ be a fringed algebra. We define a gentle framed turbulence chart $(G_\tL,\sim_\tL,\R_\tL)$ as follows:
	\begin{enumerate}
		\item The graph $G$ is the underlying undirected graph of $\tL$.
		\item At any internal vertex $v$ with relations $\alpha_1\alpha_2$ and $\beta_1\beta_2$, the equivalence relation $(\sim_\tL)_v$ equates the (head) half-arrow $(\alpha_1,v)$ with the (tail) half-arrow $(\alpha_2,v)$ and the framing $\R_\tL$ rates $(\alpha_2,v)<(\alpha_1,v)$. Similarly, the equivalence relation $(\sim_{\tL})_v$ equates $(\beta_1,v)$ with $(\beta_2,v)$ and $\R_{\tL}$ rates $(\beta_2,v)<(\alpha_1,v)$.
	\end{enumerate}
\end{defn}

\begin{prop}\label{prop:gentandchart}
	The maps $(G,\sim,\R)\mapsto\tL(G,\sim,\R)$ and $\tL\mapsto(G_\tL,\sim_\tL,\R_\tL)$ defined above are mutual inverses realizing a bijection between gentle framed turbulence charts and fringed algebras.
	Moreover, the natural identification of edges of $(G,\sim,\R)$ to arrows of $\tL(G,\sim,\R)$ induces
	\begin{enumerate}
		\item an equivalence of turbulence polyhedra $\F_1(G,\sim)\cong\F_1(\tL(G,\sim))$,
		\item a bijection between routes (resp. bands) of $(G,\sim,\R)$ and routes (resp. bands) of $\tL(G,\sim,\R)$ which preserves indicator vectors and pairwise compatibility, and hence
		\item an isomorphism between the bundle complex of $(G,\sim,\R)$ and the bundle complex of $\tL(G,\sim,\R)$.
	\end{enumerate}
\end{prop}
\begin{proof}
	It is not hard to check that these maps are mutual inverses.
	We remark that Condition~\eqref{gc3} of Definition~\ref{defn:gentchart} is precisely the condition that $\tL(G,\sim,\R)$ has no oriented cycles which are bands, and hence that it is a finite-dimensional algebra.
	Condition~\eqref{gc1.5} amounts to the condition that a fringed algebra $\tL$ may not have an edge directly between fringed vertices, as this would form a connected component with no vertex of the original gentle algebra $\L$.

	With this in mind, it is evident that the map of Definition~\ref{defn:gentchart} gives a valid fringed algebra of a gentle algebra.
	It follows immediately from the definitions that $\F_1(G,\sim)\cong\F_1(\tL(G,\sim))$ and that the notion of compatibility, hence the bundle complex, is preserved by this map.
\end{proof}

Proposition~\ref{prop:gentandchart} effectively states that (resp. flows and bundles on) gentle turbulence charts are the same as (resp. flows and bundles on) fringed algebras.
For example, see the same flow labelled on the fringed algebra (middle) and turbulence chart (right) of Figure~\ref{fig:flow}. See also
Figures~\ref{fig:sqturb} and~\ref{fig:sqgent}, which give the turbulence polyhedra and bundle complexes of the simplest gentle turbulence chart and its fringed algebra.

\begin{remk}\label{remk:gentle-case}
	Recall the background sections~\ref{ssec:Bpres} and~\ref{ssec:Bsubd}, where presentation (Theorem~\ref{thm:BPres}) and subdivision (Theorems~\ref{thm:BClique} and \ref{thm:BBundle}) results were proven about the turbulence polyhedron of a fringed algebra $\tL$ using the combinatorics of bundles on $\tL$.
	In light of Proposition~\ref{prop:gentandchart}, these results may be thought of as giving presentations and subdivisions of turbulence polyhedra of gentle framed turbulence charts based on the combinatorics of their bundles. In particular, it is immediate that turbulence polyhedra of gentle framed turbulence charts admit clique triangulations and bundle subdivisions.
	In fact, we will prove that our presentation and subdivision results work for general framed turbulence charts by reducing to the gentle case and applying these fringed algebra results.
\end{remk}

\subsection{Directable turbulence charts}
\label{ssec:directable}

Recall from the end of Section~\ref{ssec:turbchart} that any directed graph $G$ gives rise to a turbulence chart $(G',\sim)$ such that $\F_1(G)=\F_1(G',\sim)$. We build on this idea by describing the class of turbulence charts which may be obtained in this way. Moreover, we extend it to a map from framed directed graphs to framed turbulence charts and show that this correspondence respects bundle complexes.

First, we describe which turbulence charts we hope to model through convenient framed directed graphs.

\begin{defn}\label{defn:directable}
	A turbulence chart $(G,\sim,\R)$ is \emph{directable} if the edges of $G$ may be given orientations such that at every internal vertex $v$, the equivalence relation $\sim_v$ separates the head half-edges from the tail half-edges.
\end{defn}

It is immediate that the directable turbulence charts of Definition~\ref{defn:directable} are related to directed graphs. We expand this connection to framed directed graphs and framed turbulence charts.

\begin{defn}\label{defn:chart-from-dg-}
	Let $\Gamma=(G,\R)$ be a convenient framed directed graph.
	We define a framed turbulence chart $(G_\Gamma,\sim_\Gamma,\R_\Gamma)$ as follows:
	\begin{enumerate}
		\item The graph $G_\Gamma$ is the underlying undirected graph of $G$.
		\item At each internal vertex $v\in G_\Gamma$, the equivalence relation $\sim$ separates the half-edges of $G_\Gamma$ coming from incoming edges to $v$ from those half-edges coming from outgoing edges from $v$.
			\begin{enumerate}
				\item The equivalence class coming from the incoming edges is ordered by $\R_\Gamma$ with the same order as $\R$, and
				\item the equivalence class coming from the outgoing edges is ordered by $\R_\Gamma$ with the opposite order as $\R$.
			\end{enumerate}
	\end{enumerate}
\end{defn}

If a framed directed graph $\Gamma=(G,\R)$ is not convenient, then we may consider the framed directed graph given by separating the source and sink vertices so that every source and sink vertex is incident to exactly one edge; this does not change the flow polytope or bundle complex. One can then apply the map defined above to obtain a framed turbulence chart.

\begin{prop}\label{prop:dg-to-chart}
	The map $\Gamma\mapsto(G_\Gamma,\sim_\Gamma,\R_\Gamma)$ of Definition~\ref{defn:chart-from-dg-} gives a map from convenient framed directed graphs to directable framed turbulence charts which restricts to a surjective map from convenient gentle framed directed graphs to gentle directable framed turbulence charts.
	Moreover, the natural identification of edges of $\Gamma$ to edges of $(G_\Gamma,\sim_\Gamma,\R_\Gamma)$ induces
	\begin{enumerate}
		\item an equivalence $\F_1(\Gamma)\cong\F_1(G_\Gamma,\sim_\Gamma)$,
		\item a bijection between routes (resp. bands) of $\Gamma$ and routes (resp. bands) of $(G_\Gamma,\sim_\Gamma,\R_\Gamma)$ which preserves indicator vectors and pairwise compatibility, and hence
		\item an isomorphism between the bundle complex of $\G$ and the bundle complex of $(G_\Gamma,\sim_\Gamma,\R_\Gamma)$.
	\end{enumerate}
\end{prop}
\begin{proof}
	This result follows immediately from the definitions.
\end{proof}

The map of Proposition~\ref{prop:dg-to-chart} is not injective. It is immediate that any connected turbulence chart is the image of precisely two convenient framed directed graphs because switching the direction of all edges of $G$ and the ordering of the framing $F$ at every vertex preserves the associated framed turbulence chart $(G_\Gamma,\sim_\Gamma,\R_\Gamma)$.
See the two framed directed graphs on the right of Figure~\ref{fig:flow}, which both correspond to the same framed turbulence chart in the middle of the figure.

\subsection{Acyclic turbulence charts}

Recall that a turbulence chart $(G,\sim)$ is \emph{acyclic} if it has no bands. 
When $(G,\sim)$ is directable, this corresponds to  acyclicity of any corresponding directed graph. When $(G,\sim)$ has a gentle framing $\F$, this corresponds to representation-finiteness of the associated fringed algebra.

Note that the indicator vector of any band of $(G,\sim)$ is a vortex, and hence generates a ray of the recession cone of $\F_1(G,\sim)$.
It is then immediate that if $(G,\sim)$ is not acyclic, then $\F_1(G,\sim)$ is unbounded. In fact, the converse holds as well, so that $(G,\sim)$ is acyclic if and only if $\F_1(G,\sim)$ is unbounded. We could prove this now, but for brevity reference Corollary~\ref{cor:acyclic}.
For example, compare the acyclic
Figure~\ref{fig:trapezoid}
with the nonacyclic Figure~\ref{fig:bowtie}.

We now connect turbulence polytopes of acyclic turbulence charts with flow polytopes of signed graphs, which were defined by M\'esz\'aros and Morales~\cite{MMSigned} as a generalization of flow polytopes of (signless) graphs from type $A_n$ root systems to type $C_{n+1}$ and/or $D_{n+1}$.
The article~\cite{MMSigned} and further works~\cite{CKM,Kim,Zeilberger}  went on to study volume formulas for flow polytopes of signed graphs, and in particular the type-$C_{n+1}$ and type-$D_{n+1}$ analogs of the Chan-Robbins-Yuen Polytope.
For more detail and information on signed graphs, we refer to~\cite[\S2]{MMSigned}.

\begin{defn}[{\cite[\S2]{MMSigned}}]
	\label{defn:signed-graph}
	A \emph{signed graph} is an undirected graph $G$ on vertex set $[n+1]$ with a \emph{sign} $\epsilon\in\{+,-\}$ assigned to each of its edges. Loops and multiple edges are allowed. The sign of a loop is always $+$.
	A positive edge $e$ between $i$ and $j$ is \emph{positively incident} to both of its endpoints, written $\text{inc}(e,i)=\text{inc}(e,j)=+$. A negative edge between $i$ and $j$ for $i<j$ is \emph{positively incident} to $i$, written $\text{inc}(e,i)=+$, and \emph{negatively incident} to $j$, written $\text{inc}(e,j)=-$.
\end{defn}

\begin{defn}[{\cite[\S2]{MMSigned}}]
	Let $\bf a$ be a \emph{netflow vector} in $\mathbb Z_{\geq0}^{n+1}$.
	An \emph{$\bf a$-flow} is a vector $F\in\mathbb R_{\geq0}^{n+1}$ such that for all vertices $v$ of $G$, the following equality is satisfied:
	\[
		\sum_{e\in E\ :\ \text{inc}(e,v)=-}F(e)+a_v=\sum_{e\in E\text{ not a loop}\ :\ \text{inc}(e,v)=+}F(e)
		+
		\sum_{e\in E\text{ loop}\ :\ \text{inc}(e,v)=+}2F(e).
	\]
\end{defn}

Note that if edges have negative sign, then we recover traditional (signless) flow polytopes.

\begin{prop}\label{prop:signed}
	The set of turbulence polyhedra of acyclic turbulence charts is equal to the set of signed flow polytopes of signed graphs with netflow vector $(2,0,\dots,0)$.
\end{prop}
\begin{proof}
	We will prove the proposition as follows: Given a signed graph with netflow vector $(2,0,\dots,0)$, we will obtain an acyclic turbulence chart with the same turbulence polytope, and given an acyclic turbulence chart we will obtain a signed graph with netflow vector $(2,0,\dots,0)$ with the same signed flow polytope.

	Start with a signed graph $G$ with netflow vector ${\bf a}=(2,0,\dots,0)$.
	Let $G'$ be the graph given by splitting vertex $1$ of $G$ into multiple (fringe) vertices, such that every fringe vertex has degree 1. For each (internal) vertex of $\{2,\dots,n+1\}$, let $\sim_v$ separate the half-edges with positive incidence from the half-edges with negative incidence. It is then immediate from the definitions that the signed flow polytope of $(G',{\bf a})$ is equal to $\F_1(G',\sim)$.

	For the other direction, start with an acyclic turbulence chart $(G,\sim)$ with $n$ internal vertices. We will obtain a signed graph $G'$ by combining all fringe vertices of $G$ into a single vertex $1$, and choosing an bijection from the internal vertices of $G$ to $\{2,\dots,n+1\}$ and an assignment in $\{+,-\}$ to the equivalence classes $\sim_v$ at each internal vertex to realize $G'$ as a signed graph.

	\textbf{\underline{Claim:}} If $V$ is a nonempty subset of the internal vertices of $G$, then there exists a vertex $v\in V$ and an equivalence class $v^-$ of $\sim_{v}$ such that all edges with one half-edge in $v^-$ have their second half-edge incident to a vertex outside of $V$. Indeed, if this is not the case, then start with any $w_1\in V$. Since the claim is false, we may choose an edge $e_1$ from $w_1$ to a vertex $w_2\in V$. Again, since the claim is false, we may choose an edge $e_2$ from $w_2$ to a vertex $w_3\in V$ such that $e_1e_2$ is a string. We repeat this process indefinitely to get an infinite string $e_1e_2e_3\dots$. Since $G$ is finite, this string must have a subset which is a band, contradicting acyclicity of $G$. This proves the claim.

	By the claim with $V$ being the set of all internal vertices, we may choose a vertex $v_2$ of $G$ and an equivalence class $v_2^-$ of $\sim_{v_2}$ such that all edges with a half-edge in $v_2^-$ are fringe. Let $v_2^+$ be the other equivalence class of $\sim_{v_2}$. Moreover, we may apply the claim inductively with $V_j$ being the set of internal vertices missing $\{v_1,v_2,\dots,v_{j-1}\}$ to choose an order $\{v_2,\dots,v_{n+1}\}$ on the internal vertices of $G$ with equivalence classes $v_j^-$ of $\sim_{v_j}$ for $j\in\{2,\dots,n+1\}$ such that every edge with a half-edge in $v_j^-$ has a second half-edge incident to a fringe vertex or $v_i^+$ with $i<j$.
	In particular, every edge of $G$ is either
	\begin{itemize}
		\item incident to a fringe vertex and some $v_j^-$ (\emph{negative}),
		\item incident to $v_i^+$ and $v_j^-$ for $i<j$ (\emph{negative}), or
		\item incident to $v_i^+$ and $v_j^+$ for $i<j$ (\emph{positive}).
	\end{itemize}

	Then define $G'$ to be the signed graph on vertex set $[n+1]$ whose edges are in bijection with those of $G$ as induced by $\{v_2,\dots,v_{n+1}\}$ and whose signs are induced by the above characterization. It is now immediate from the definitions that the signed flow polytope of this signed graph with netflow vector $(2,0,\dots,0)$ is equal to the turbulence polytope $\F_1(G)$.
\end{proof}

The main result Theorem~\ref{ithm:bundle} of this paper, then, gives unimodular ``framing triangulations'' to flow polytopes of signed graphs with netflow vector $(2,0,\dots,0)$.

\subsection{Amply framed turbulence charts}
\label{ssec:ample}

Amply framed DAGs are a slight generalization of gentle framed DAGs introduced by Danilov, Karzanov, and Koshevoy~\cite{DKK} in their original paper on framing triangulations.
In this section, we will generalize amply framed DAGs by defining amply framed turbulence charts and show that they generalize gentle framed turbulence charts.
We show that amply framed turbulence charts are precisely those whose ``reduced nonnegative space cones'' are complete.

\begin{defn}
	A trail of a framed turbulence chart $(G,\sim,\R)$ is \emph{exceptional} if it is compatible with every trail.
	An edge $e$ of $(G,\sim)$ is called \emph{valid} if there exists a flow which is nonzero on $e$.
	The turbulence chart $(G,\sim,\R)$ is \emph{amply framed} if every valid edge is part of an exceptional trail.
\end{defn}

In~\cite[\S6]{DKK}, amply framed DAGs were connected with the reduced nonnegative spaces of nonnegative flows.
We give the turbulence chart versions of these definitions here.

\begin{defn}
	Let $(G,\sim,\R)$ be a turbulence chart and let $\mathcal E$ be its set of exceptional bundles. Define the \emph{reduced space} $\F_{\text{red}}$ to be the space of (not necessarily nonnegative) flows quotiented by the span $\mathbb R\mathcal E$ of the indicator vectors of exceptional bundles.
	Define the \emph{reduced nonnegative space}  $\F_{\geq0,\text{red}}:=\F_{\geq0}/\mathbb R\mathcal E$ to be the space of nonnegative flows quotiented by $\mathbb R\mathcal E$.
\end{defn}

The following generalizes~\cite[Proposition 5]{DKK}.

\begin{prop}\label{prop:amp}
	The following are equivalent for a framed turbulence chart $(G,\sim,\R)$:
	\begin{enumerate}
		\item $(G,\sim,\R)$ is amply framed;
		\item The reduced nonnegative space $\F_{\geq0,\text{red}}$ is the whole space $\F_{\text{red}}$.
	\end{enumerate}
\end{prop}
\begin{proof}
	Suppose $(G,\sim,\R)$ is amply framed. For any flow $F$ on $(G,\sim)$, set $M:=\text{max}_{e\in E}|F(e)|$. Then $F+\sum_{p\in\mathcal E}M\mathcal I(p)$ is a nonnegative flow which is equivalent to $F$ modulo $\mathbb R\mathcal E$, so $F+\mathbb R\mathcal E\in\F_{\text{red}}$ is in $\F_{\geq0,\text{red}}$. This shows that $\F_{\geq0,\text{red}}=\F_{\text{red}}$.

	On the other hand, suppose $(G,\sim,\R)$ is not amply framed. Then there is a valid edge $e$ of $(G,\sim,\R)$ which is part of no exceptional trail. Since $e$ is valid, there exists a flow $F$ which is nonzero on $e$. If $F(e)>0$, then replace $F$ with $-F$ so that $F(e)<0$.
	Since $e$ is part of no exceptional trail, it is impossible to realize $F$ as a nonnegative flow added to an element of $\mathbb R\mathcal E$, hence the equivalence class $F+\mathbb R\mathcal E\in\F_{\text{red}}$ is not in $\F_{\geq0,\text{red}}$ so $\F_{\geq0,\text{red}}\subsetneq\F_{\text{red}}$.
\end{proof}

\begin{lemma}\label{lem:gentisamp}
	A locally gentle framed turbulence chart $(G,\sim,\R)$ is amply framed.
\end{lemma}
\begin{proof}
	If $(G,\sim,\R)$ is locally gentle, then any edge $e$ of $G$ is steep. If $e$ is ascending, then one obtains an exceptional trail $p_e$ by continuing along ascending edges forward and descending edges backwards, until reaching fringe edges or looping back to become a band. This shows that every edge is part of an exceptional trail, hence $(G,\sim,\R)$ is amply framed.
\end{proof}

It was shown in~\cite{WIWT} that a framed DAG $(G,\R)$ with no idle edges is amply framed if and only if
	\begin{enumerate}
		\item $G$ is full (i.e., each internal edge has exactly two incoming and two outgoing edges), and
		\item there is a map $\psi_\mathfrak{F}:E\to\{1,2\}$ such that, if ${e}$ and ${f}$ are edges with the same internal source or target, we have ${e}<_{\mathfrak{F}}{f}$ (in the partial order of $\inn(h({e}))$ or $\out(t({e}))$) if and only if $\psi_\mathfrak{F}({e})<\psi_\mathfrak{F}({f})$.
	\end{enumerate}

Note that this is the same as Definition~\ref{defn:amply-framed-dg} without the condition that $G$ has no edges directly from source to sink (and without the final condition, which is vacuous since $G$ has no directed cycles). In other words, an amply framed DAG is precisely a gentle framed DAG which allows edges directly from source to sink.
Note that an edge of $G$ directly from source to sink simply takes a cone over the triangulated flow polytope, and does not affect the clique triangulation. Hence, the combinatorics of amply framed DAGs is almost entirely captured by the combinatorics of gentle framed DAGs. We expect that, more generally, amply framed turbulence charts are only a slight generalization of gentle framed turbulence charts.

\begin{remk}\label{remk:uqam1}
	The class of ``amply framed directed graphs all of whose cycles are monolabelled'' (see, e.g., Example~\ref{ex:kron2}) was studied in depth in~\cite{UQAM} under the name \emph{cyclic amply framed DAGs}, where they obtained triangulation results, volume results, 
	and gave connections to the g-vector fans of locally gentle algebras (see Remark~\ref{remk:locally-gentle}).
	Through the map of Definition~\ref{defn:chart-from-dg-}, these cyclic amply framed DAGs give examples of amply framed turbulence charts.
\end{remk}

We finally remark that for a framed $(G,\R)$, the reduced nonnegative space $\F_{\geq0,\text{red}}$ being complete translates to the dual graph of the framing triangulation being a simplicial sphere. 
This gives an interpretation for how exactly general framed DAGs generalize amply framed DAGs -- they allow framing triangulations to fail to be simplicial spheres.
It is possible that there may be a similar interpretation for how exactly general turbulence charts generalize amply framed turbulence charts, or even gentle framed turbulence charts.

\section{Gentle Envelopes}
\label{sec:ge}

In this section, we will realize an arbitrary turbulence polyhedron as a face of a gentle turbulence polyhedron in a way which respects the bundle complex. On the level of turbulence charts, this will amount to starting with an arbitrary turbulence chart $(G,\sim,\R)$ and finding a gentle turbulence chart $(G',\sim',\R')$ such that deleting some fringe edges of $(G',\sim',\R')$ and performing some contractions yields $(G,\sim,\R)$. 
We will use this in Section~\ref{sec:pres-subd} to transport presentation and subdivision results about $(G,\sim',\R')$ to $(G,\sim,\R)$.

First, we need to define contraction steps.

\begin{defn}
	Let $(G,\sim,\R)$ be a framed turbulence polyhedron. An edge $\alpha$ of $(G,\sim,\R)$ is \emph{idle} if it connects to two distinct vertices and has a lonely half-edge at an internal vertex (recall Definition~\ref{defn:lonely}).
	Let $\alpha$ be an idle edge of $(G,\sim,\R)$ with half-edges $(\alpha,v_1)$ and $(\alpha,v_2)$, where $(\alpha,v_1)$ is lonely and $v_1\neq v_2$. We define a new framed turbulence chart $(G',\sim',\R')$ as follows:
	\begin{enumerate}
		\item For every vertex $v$ of $G$, there is a vertex $v'$ of $G'$, with the caveat that we identify $v'_1=v'_2$.
		\item For every edge $\beta\neq\alpha$ of $G$ with half-edges $(\beta,w_1)$ and $(\beta,w_2)$, there is an edge $\beta'$ of $G'$ with half-edges $(\beta',w'_1)$ and $(\beta',w'_2)$.
		\item At every vertex $v$ of $G'$ other than $v'_1=v'_2$, the equivalence relation $\sim'_v$ and orders of $\R'$ are inherited from $(G,\sim,\R)$.
		\item Let $f_1<_{\R}\dots<_{\R}f_{m_1}$ be the equivalence class of $\sim_{v_1}$ not containing $(\alpha,v_1)$. Let $g_1<_{\R}<\dots<_{\R}g_{m_2}$ be the equivalence class of $\sim_{v_1}$ containing $(\alpha,v_2)$, and choose $a\in[m_2]$ so that $(\alpha,v_2)$ is $g_a$. Let $h_1<_{\R}\dots<_{\R}h_{m_3}$ be the equivalence class of $\sim_{v_2}$ not containing $(\alpha,v_2)$. Then the equivalence classes of $\sim'_{v'_1}$ and their orders by $\R'$ are $h'_1<_{\R'}\dots<_{\R'}h'_{m_3}$ and 
			\[g'_1<_{\R'}\dots<_{\R'}g'_{a-1}<_{\R'}f'_1<_{\R'}\dots<_{\R'}f'_{m_1}<_{\R'}g'_{a+1}<_{\R'}\dots<_{\R'}g'_{m_2}.\]
	\end{enumerate}
	We say that $(G',\sim',\R')$ is a \emph{contraction of $(G,\sim,\R)$ at $\alpha$.}
\end{defn}

See Figure~\ref{fig:contractchart} for an example of a contraction at an idle edge.

\begin{figure}
	\centering
	\def\svgscale{.21}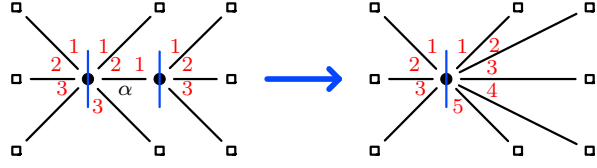
	\caption{Contracting the left framed turbulence chart at $\alpha$ gives the right framed turbulence chart.}
	\label{fig:contractchart}
\end{figure}

\begin{lemma}\label{lem:contractchart}
	Let $(G',\sim',\R')$ be a contraction of $(G,\sim,\R)$ at an idle edge $\alpha$. Then 
	\begin{enumerate}
		\item there is a natural bijection $\phi$ from trails of $(G,\sim,\R)$ to trails of $(G',\sim',\R')$ given by removing all occurrences of (any orientation of) the edge $\alpha$,
		\item the map $\phi$ preserves compatibility: if $S$ is a collection of trails of $(G,\sim,\R)$, then $\phi(S)$ is a bundle if and only if $S$ is a bundle, and
		\item there is a unimodular equivalence from $\F_1(G,\sim)$ to $\F_1(G',\sim')$ which sends $\mathcal I(p)$ to $\mathcal I(\phi(p))$ for any trail $p$ of $(G,\sim,\R)$.
	\end{enumerate}
\end{lemma}
\begin{proof}
	(1) and (2) are immediate. The unimodular equivalence of (3) is the projection given by forgetting the coordinate of $\mathbb R^E$ associated to $\alpha$.
\end{proof}

We now describe a process by which one may start with a general turbulence chart $(G,\sim,\R)$ and perform several kinds of moves to pull it closer to being gentle. We begin with the following move which pulls $(G,\sim,R)$ closer to being sub-full.

\begin{defn}\label{move:decont}
	Let $(G,\sim,\R)$ be a framed turbulence chart and suppose $v$ is an internal vertex of $G$ such that one of the equivalence classes of $\sim_v$ has cardinality greater than two. Notate this equivalence class as $h_1<_{\R} h_2<_{\R}\dots<_{\R}h_m$ for $m>2$.
	Choose $a\in\{1,\dots,m-2\}$.
	We define a new framed turbulence chart $(G',\sim',\R')$ as follows:
	\begin{itemize}
		\item For any vertex $w$ of $G$ there is a corresponding vertex $w'$ of $G'$. In addition, there is a vertex $v''$ of $G'$.
		\item For any edge $e$ of $G$, there is a corresponding edge $e'$ of $G'$. Say $(e,v_1)$ and $(e,v_2)$ are the two distinct half-edges of $e$. Then the half-edges of $e'$ are $(e',u_1)$ and $(e',u_2)$, where for $i\in\{1,2\}$ we have
			\[
				u_i=\begin{cases}v'_i & v_i\neq v \\
					v' & (e,v_i)=h_j\text{ for }j\leq a\\
					v'' & (e,v_i)=h_j\text{ for }j>a.
				\end{cases}
			\]
		\item There is an additional \emph{decontraction edge} $\delta$ from $v'$ to $v''$.
		\item The equivalence classes of $\sim'$ and the framing $\R'$ at vertices other than $v'$ and $v''$ are inherited from $(G,\sim,\R)$.
		\item At vertex $v'$, the half-edge $(\delta,v')$ is added to the equivalence class $\{h'_1,\dots,h'_a\}$ and the framing $\R'$ orders $h'_1<_{\R'}\dots<_{\R'}h'_a<_{\R'}(\delta,v')$.
		\item At vertex $v''$, the half-edge $(\delta,v'')$ is its own equivalence class in $\sim'_{v''}$ and the other equivalence class is $h'_{a+1}<_{\R'}\dots<_{\R'}h'_{m}$.
	\end{itemize}
	We say that $(G',\sim',\R')$ is a \emph{degree-reduction of $(G,\sim,\R)$ at $v$}.
\end{defn}

See Figure~\ref{fig:moveone} for an example of a degree-reduction with $a=1$.
\begin{figure}
	\centering
	\def\svgscale{.21}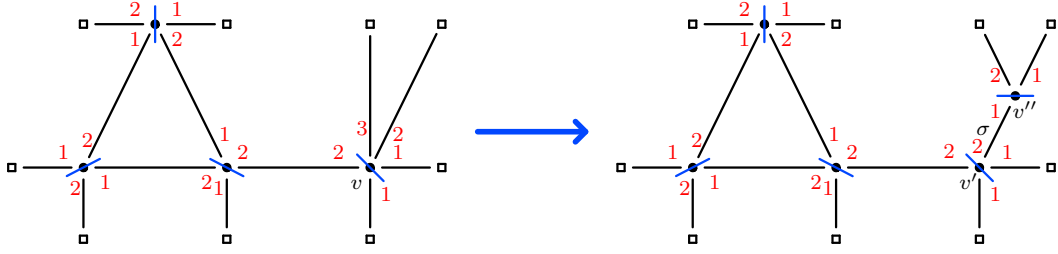
	\caption{Performing a degree-reduction at $v$ with $a=1$.}
	\label{fig:moveone}
\end{figure}

The following remark along with Lemma~\ref{lem:contractchart} shows that degree-reduction moves preserve the turbulence polyhedron and bundle complex.
\begin{remk}\label{remk:undocont}
	If $(G',\sim',\R')$ is a degree-reduction of $(G,\sim,\R)$, then the decontraction edge $\delta$ is idle and we may contract $(G',\sim',\R')$ at $\delta$ to recover the original framed turbulence polyhedron $(G,\sim,\R)$.
\end{remk}
We may take a general framed turbulence chart $(G,\sim,\R)$ and repeatedly apply degree-reduction moves until we reach a framed turbulence chart $(G_1,\sim_1,\R_1)$ which is sub-full. 
Our next move will allow us to make all edges of $(G,\sim,\R)$ steep.

\begin{defn}\label{move:steepening}
	Let $(G,\sim,\R)$ be a sub-full framed turbulence chart. Let $e$ be an edge which is between two fringe vertices (hence is steep), or which is between two internal vertices and is not steep (i.e., neither edge of $e$ is lonely, and either both half-edges are high or both half-edges are low).
	We define a new framed turbulence chart $(G',\sim',\R')$ as follows:
	\begin{itemize}
		\item For any vertex $w$ of $G$ there is a corresponding vertex $w'$ of $G'$. In addition, there is a vertex $v_e'$ of $G'$.
		\item For any edge $f\neq e$ of $G$, there is a corresponding edge $f'$ of $G'$ with the same endpoints.
		\item Replacing the edge $e$ of $G$ with half-edges $(e,v_1)$ and $(e,v_2)$ are the \emph{decontraction edges} $e'_1$ between $v'_1$ and $v'_e$, and $e'_2$ between $v'_2$ and $v'_{e}$.
		\item The equivalence classes of $\sim'$ at the vertices other than $v'_e$ and their orders by $\R'$ are inherited as expected from $(G,\sim,\R)$. The new vertex $v'_e$ is incident to only two half-edges, each of which forms its own equivalence class in $\sim'_{v'_e}$.
	\end{itemize}
	We say that $(G',\sim',\R')$ is a \emph{steepening of $(G,\sim,\R)$ at $e$}.
\end{defn}

See Figure~\ref{fig:movetwo} for an example of a steepening move at an internal edge.
\begin{figure}
	\centering
	\def\svgscale{.21}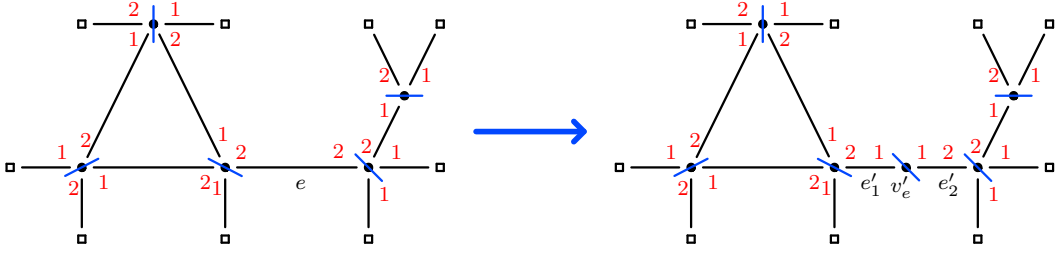
	\caption{Performing a steepening move at $e$.}
	\label{fig:movetwo}
\end{figure}

\begin{remk}\label{remk:undosteep}
	If $(G',\sim',\R')$ is a steepening of $(G,\sim,\R)$, then the decontraction edges are idle and we may contract $(G',\sim',\R')$ at either decontraction edge to recover $(G,\sim,\R)$.
\end{remk}

Note that after a steepening move at an internal edge $e$, the non-steep edge $e$ has been replaced with two steep edges $e'_1$ and $e'_2$. Then the number of non-steep edges decreases upon application of a steepening move. After a steepening move at an edge $e$ between two fringe vertices, the edge $e$ has been replaced with two steep edges $e'_1$ and $e'_2$ which do not lie between two fringe vertices.
Hence, by using degree-reduction moves and steepening moves we can reduce to the case where our turbulence chart is sub-full, has only steep edges, and has no edges between two fringe vertices. We now add extra edges to realize our turbulence chart inside of a full turbulence chart.

\begin{defn}
	Let $(G,\sim,\R)$ be a sub-full framed turbulence chart such that every edge is steep. If $(G,\sim)$ is not full, then let $v$ be an internal vertex of $G$ such that one of the equivalence classes of $\sim_v$ consists of a single half-edge $h_1$. Let $f$ be the edge containing $h_1$ and let $h_2$ be the other half-edge of $f$. We define a framed turbulence chart $(G',\sim',\R')$ as follows:
	\begin{enumerate}
		\item For every vertex $v\in G$, there is a vertex $v'\in G'$.
		\item For every edge $e\in G$, there is an edge $e'\in G'$ with the same start and end points.
		\item The equivalence relations $\sim'_{v'}$ and the orders of $\R'$ are inherited as expected from $\sim$ and $\R$ with respect to these edges.
		\item There is one additional fringed vertex $v''$ of $G'$ and one additional edge $e'_{v''}$ from $v''$ to $v'$. Say $h'_{v'}$ is its half-edge at $v'$; then $h'_{v'}$ and $h'_1$ are equivalent under $\sim'_{v'}$.
		\item If $h_2$ (equivalently, $h'_2$) is high, then the framing $\R'$ orders $h'_1<_{\R'}h'_{v}$. Otherwise, $h'_2$ is low and the framing $\R'$ orders $h'_1>_{\R'}h'_v$.
			Note that the edge $f'$ is steep in either case.
	\end{enumerate}
	We say that $(G',\sim',\R')$ is obtained from $(G,\sim,\R)$ by a \emph{filling step} at $v$.
\end{defn}

Note that if $(G',\sim',\R')$ is obtained from $(G,\sim,\R)$ by a filling step, then $(G',\sim',\R')$ is still a sub-full framed turbulence chart such that every edge is steep. Hence, we can continue applying filling steps until we reach a \emph{full} framed turbulence chart $(G'',\sim'',\R'')$ with only steep edges, which we call a \emph{filling} of $(G,\sim,\R)$. 
Note that if $W$ is the set of fringe edges added during filling steps, then removing all edges of $W$ recovers $(G,\sim,\R)$.
See Figure~\ref{fig:movethree} for an example of a filling, where the new fringe edges $W$ are drawn in purple.
\begin{figure}
	\centering
	\def\svgscale{.21}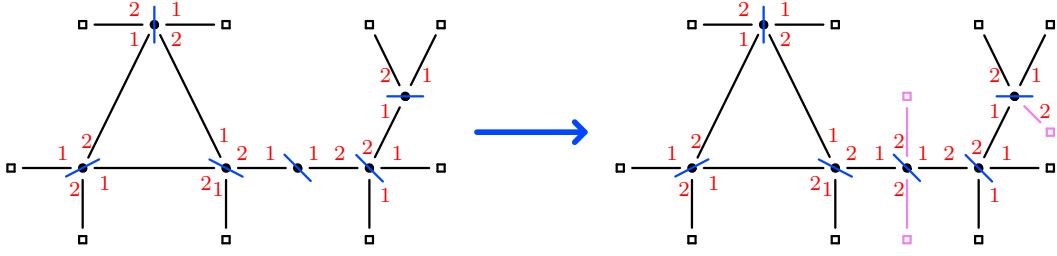
	\caption{Filling a framed turbulence chart.}
	\label{fig:movethree}
\end{figure}

We remark that if there is an oriented edge $\te$ of $(G,\sim,\R)$ with two lonely half-edges, then there will be multiple fillings of $(G,\sim,\R)$; some which make $\te$ ascending, and some which make $\te$ descending.

As Figure~\ref{fig:movethree} shows, the turbulence chart $(G'',\sim'',\R'')$ resulting from a filling may still fail to be gentle.
This is because we may have problematic bands as forbidden in Definition~\ref{defn:gentchart}~\eqref{gc3}. We correct such bands with a final reduction move.

\begin{defn}
	Let $(G,\sim,\R)$ be a framed turbulence chart such that $G$ is full and every edge is steep. A band $B=\te_1\te_2\dots \te_{m}$ is \emph{steep} if every $\te_j$ is ascending, or every $\te_j$ is descending.
	Suppose there exists a steep band $B=\te_1\te_2\dots\te_{m}$ of $(G,\sim,\R)$; if necessary, replace $B$ with $B^{-1}$ so that every edge of $B$ is ascending.
	Define a new turbulence chart $(G',\sim',\R')$ from $(G,\sim,\R)$ by adding two new (internal) vertices $v'_1$ and $v'_2$, and replacing the edge $e_1$ with
	\begin{enumerate}
		\item an edge $f_1$ between $t(\te_1)$ and $v'_1$,
		\item an edge $f_2$ between $v'_1$ and $v'_2$, and
		\item an edge $f_3$ between $v'_2$ and $h(\te_1)$.
	\end{enumerate}
	Say $\tf_1$ is the orientation of $f_1$ from $t(\te_1)$ to $v'_1$, $\tf_2$ is the orientation of $f_2$ from $v'_1$ to $v'_2$, and $\tf_3$ is the orientation of $f_3$ from $v'_2$ to $h(\te_1)$.
	The two half-edges incident to $v'_1$ (resp. $v'_2$) are each in their own equivalence class of $\sim'_{v'_1}$ (resp. $\sim'_{v'_2}$).
	This results in a new framed turbulence chart $(G',\sim',\R')$, but it is not full because $v'_1$ and $v'_2$ each have a degree of two. Let $(G'',\sim'',\R'')$ be a filling of $(G'',\sim',\R')$ such that $\tf_2$ is descending (note that $\tf_1$ and $\tf_3$ are forced to be ascending).
	We say that $(G'',\sim'',\R'')$ is a \emph{band-correction} of $(G,\sim,\R)$ with respect to the band $B$. Note that the steep band $B$ has been replaced with the band $\tf_1\tf_2\tf_3\te_2\te_3\dots\te_m$, which is not steep because $\tf_2$ is descending.
\end{defn}

See Figure~\ref{fig:movefour} for an example of a band-correction move.
Given a full framed turbulence chart such that every edge is steep, we may apply band-correction moves until no band is steep. 
\begin{remk}\label{remk:undocycle}
	We may undo a band-correction move by deleting the edges added during the filling step, then contracting along any two of the edges $\{f_1,f_2,f_3\}$.
\end{remk}

\begin{figure}
	\centering
	\def\svgscale{.21}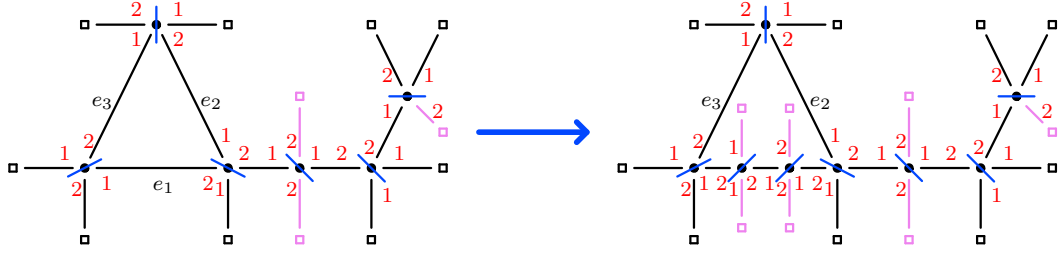
	\caption{Performing a band-correction move.}
	\label{fig:movefour}
\end{figure}

We summarize the process of this subsection with a theorem.
\begin{thm}\label{prop:face-of-gentle}
	Let $(G,\sim,\R)$ be a turbulence chart. There exists a gentle framed turbulence chart $(G',\sim',\R')$ equipped with a set of edges $W$ of $G'$ such that $(G,\sim,\R)$ is obtained by deleting the edges $W$ from $(G',\sim',\R')$ and performing a number of contraction steps.
\end{thm}
\begin{proof}
	Start with the turbulence chart $(G,\sim,\R)$ and repeatedly apply degree-reduction moves to obtain a sub-full framed turbulence chart $(G_1,\sim_1,\R_1)$. Then apply steepening moves to obtain a sub-full turbulence chart $(G_2,\sim_2,\R_2)$ all of whose edges are steep and which has no edges between two fringe vertices. Then apply filling steps to get a full turbulence chart $(G_3,\sim_3,\R_3)$ with only steep edges and no edges between two fringe vertices. Finally, apply band correction moves to reach a gentle turbulence chart $(G_4,\sim_4,\R_4)$.

	Let $W$ be the set of all fringe edges adding during filling steps (including those fringe edges added during band-correction steps).
	Remarks~\ref{remk:undocont},~\ref{remk:undosteep}, and~\ref{remk:undocycle} show that degree-reduction, steepening, and band correction moves may be undone by contractions after removing all edges of $W$. This ends the proof.
\end{proof}

We call $(G',\sim',\R',W)$ as in Theorem~\ref{prop:face-of-gentle} a \emph{gentle envelope} of $(G,\sim,\R)$.

Figures~\ref{fig:moveone},~\ref{fig:movetwo},~\ref{fig:movethree}, and~\ref{fig:movefour} give an example of the full process of applying degree-reduction, steepening, band correction, and filling moves to a general framed turbulence chart to reach a gentle framed turbulence chart with such an edge set $W$. For good measure, Figure~\ref{fig:movefive} shows the resulting fringed algebra.

\begin{figure}
	\centering
	\def\svgscale{.21}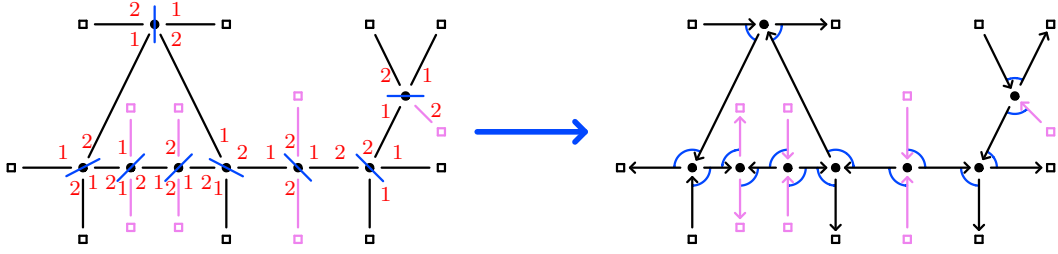
	\caption{Obtaining the gentle algebra of the result from Figure~\ref{fig:movefour}.}
	\label{fig:movefive}
\end{figure}

An immediate corollary of Theorem~\ref{prop:face-of-gentle} is the following.

\begin{cor}\label{cor:face-of-gentle}
	Up to unimodular equivalence, a turbulence polyhedron is precisely a face of a gentle turbulence polyhedron.
\end{cor}
\begin{proof}
	Since contractions do not change the turbulence polyhedron by Lemma~\ref{lem:contractchart}, Theorem~\ref{prop:face-of-gentle} shows that up to unimodular equivalence any turbulence polyhedron is given by taking the turbulence polyhedron of a gentle envelope $(G',\sim',\R')$ and setting equal to zero the flow through a set of edges $W$ of $G'$ in $\F_1(G',\sim',\R')$. This gives a face of $\F_1(G',\sim',\R')$ by Lemma~\ref{lem:gentface}.
	On the other hand, any face of a gentle turbulence polyhedron
	is given by zeroing out the flow through some edge set $W$ of the underlying gentle algebra by Lemma~\ref{lem:gentface}. Deleting these edges from the corresponding turbulence chart retrieves a turbulence chart whose turbulence polyhedron is the desired face.
\end{proof}

\begin{remk}\label{remk:wilt}
	The construction of a gentle envelope is ultimately a generalization of the \emph{ample envelopes} defined by the author and Khrystyna Serhiyenko~\cite{WILT} in the setting of framed DAGs.
	While the present article is concerned with triangulations and subdivisions, the work~\cite{WILT} is concerned with a lattice structure on the maximal simplices of a framing triangulation: starting from a framed DAG $(G,\R)$ under certain \emph{rooted} conditions, the lattice structure defined on the maximal cliques of an ample envelope of $(G,\R)$ (which was proven in the gentle case in~\cite{WIWT}) was shown to descend to a lattice structure on the maximal cliques of the original framed DAG $(G,\R)$.
\end{remk}

\section{Presentations and Subdivisions of Turbulence Polyhedra}
\label{sec:pres-subd}

In this section, we use the construction of gentle envelopes developed in the previous section to transport presentation and subdivision results from turbulence polyhedra of fringed algebras to turbulence polyhedra of arbitrary turbulence charts.

\subsection{Presentations of turbulence polyhedra}

We show that we get vertices of the turbulence polyhedron of $(G,\sim,\R)$ from indicator vectors of ``elementary'' routes, and its extremal rays from indicator vectors of ``elementary'' bands.

\begin{defn}\label{defn:elroute}
	Let $(G,\sim)$ be a turbulence chart.
	We say that a route $p$ of $(G,\sim,\R)$ is \emph{simple} if it does not use the same vertex twice.
			We say that $p$ is a \emph{lollipop} if it is of the form $s\sigma s^{-1}$ for some strings $s$ and $\sigma$ such that every vertex of $s\sigma s^{-1}$ appears exactly twice if it is a vertex of $s$, and once otherwise.
	A self-compatible route is \emph{elementary} if it is simple or it is a lollipop.
\end{defn}

\begin{defn}\label{defn:elband}
	Let $(G,\sim,\R)$ be a framed turbulence chart.
	We say that a band $B$ of $(G,\sim,\R)$ is \emph{simple} if it does not use the same vertex twice. We say that $B$ is a \emph{barbell} if up to cyclic equivalence it is of the form $s\sigma_1 s^{-1}\sigma_2$ for some strings $s$, $\sigma_1$, and $\sigma_2$ such that every vertex used appears exactly twice if it is a vertex of $s$, and once otherwise.
	A band is \emph{elementary} if it is simple or it is a barbell.
\end{defn}

See Figure~\ref{fig:ELEMENTARY} for examples of elementary and nonelementary trails.

\begin{figure}
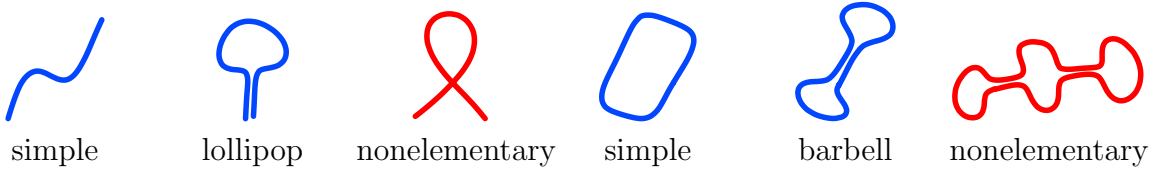

\centering
	\begin{minipage}[b]{.15\textwidth}
  \centering
	\def\svgscale{.21}\input{1a.pdf_tex}

	simple
\end{minipage}
	\begin{minipage}[b]{.15\textwidth}
  \centering
	\def\svgscale{.21}\input{2a.pdf_tex}

	lollipop
\end{minipage}
	\begin{minipage}[b]{.15\textwidth}
  \centering
	\def\svgscale{.21}\input{3a.pdf_tex}

	nonelementary
\end{minipage}
	\begin{minipage}[b]{.15\textwidth}
  \centering
	\def\svgscale{.21}\input{4a.pdf_tex}

	simple
\end{minipage}
	\begin{minipage}[b]{.15\textwidth}
  \centering
	\def\svgscale{.21}\input{5a.pdf_tex}

	barbell
\end{minipage}
	\begin{minipage}[b]{.15\textwidth}
  \centering
	\def\svgscale{.21}\input{6a.pdf_tex}

	nonelementary
\end{minipage}
	\caption{Examples of elementary (blue) and nonelementary (red) routes and bands.}
	\label{fig:ELEMENTARY}
\end{figure}

\begin{remk}\label{remk:band-2to1}
	Note that if $B:=s\sigma_1 s^{-1}\sigma_2$ is a barbell, then $B':=s\sigma_1s^{-1}\sigma_2^{-1}$ is a different barbell with the same indicator vector. Upon choosing a framing $\R$ precisely one of $B$ and $B'$ will be self-compatible.
\end{remk}

\begin{defn}
	If $(G,\sim,\R)$ is a framed turbulence chart and $W$ is a set of edges of $G$, then a trail of $(G,\sim,\R)$ is \emph{$W$-avoiding} if it uses no edges of $W$. A clique of $(G,\sim,\R)$ is \emph{$W$-avoiding} if it consists only of $W$-avoiding trails. A flow $F$ on $(G,\sim,\R)$ is \emph{$W$-avoiding} if $F(e)=0$ for all $e\in W$. The set of $W$-avoiding nonnegative flows make up the face $Q_W$ of $\F_1(G,\sim)$ as in Definition~\ref{defn:qw}.
\end{defn}

\begin{lemma}\label{lem:elemdelete}
	Let $(G,\sim,\R)$ be a framed turbulence chart and let $W$ be a set of edges of $G$. Let $(G',\sim',\R')$ be the framed turbulence chart given by deleting edges of $W$ from $(G,\sim,\R)$. Then the $W$-avoiding elementary trails of $(G,\sim,\R)$ are precisely the elementary trails of $(G',\sim',\R')$.
\end{lemma}
\begin{proof}
	Immediate by definitions.
\end{proof}

\begin{lemma}\label{lem:elemcontract}
	Let $(G,\sim,\R)$ be a framed turbulence chart and let $e$ be an idle edge of $G$. Let $(G',\sim',\R')$ be the framed turbulence chart given by contracting the idle edge $e$. Then the elementary trails of $(G,\sim,\R)$ correspond precisely to the elementary trails of $(G',\sim',\R')$ through removing all copies of (any orientation of) $e$.
\end{lemma}
\begin{proof}
	Immediate by definitions.
\end{proof}

\begin{thm}\label{thm:pres-general}
	Let $(G,\sim,\R)$ be a framed turbulence chart. Then $p\mapsto\I(p)$ is a bijection from elementary routes of $(G,\sim,\R)$ to vertices of $\F_1(G,\sim)$, and $B\mapsto\I(B)$ is a bijection from self-compatible elementary bands of $(G,\sim,\R)$ to the extremal rays of $\F_1(G,\sim)$.
\end{thm}
\begin{proof}
	Choose a gentle envelope $(G',\sim',\R',W)$ of $(G,\sim,\R)$ by Theorem~\ref{prop:face-of-gentle}. 
	Since this framed turbulence chart is gentle, Theorem~\ref{thm:BPres} along with Proposition~\ref{prop:gentandchart}
	shows that taking indicator vectors gives a bijection between elementary routes of $(G',\sim',\R')$ and vertices of $\F_1(G',\sim',\R')$.

	The vertices of the face $Q_W$ of $\F_1(G',\sim',\R')$ given by zeroing out flow through the edges of $W$ are precisely the vertices of $\F_1(G',\sim',\R')$ which are $W$-avoiding flows. It is immediate, then, that these are indicator vectors of $W$-avoiding elementary routes. 
	Similarly, the recession cone of $Q_W$ consists of the $W$-avoiding vortices, and these are minimally generated by the indicator vectors of $W$-avoiding elementary bands.

	The turbulence chart $(G,\sim,\R)$ is obtained by deleting the edges $W$ from $(G',\sim',\R')$ and then contracting idle edges. So, applying Lemmas~\ref{lem:elemdelete} and~\ref{lem:elemcontract} shows the desired result.
\end{proof}

Note that Theorem~\ref{ithm:pres} in the introduction states that taking indicator vectors gives a map from elementary bands to extremal rays which is injective on simple bands but 2-to-1 on barbells, while Theorem~\ref{thm:pres-general} states that taking indicator vectors is a bijection from elementary bands to the extremal rays. These are equivalent by Remark~\ref{remk:band-2to1}.

\begin{cor}\label{cor:acyclic}
	A turbulence chart $(G,\sim)$ is acyclic if and only if $\F_1(G,\sim)$ is bounded.
\end{cor}
\begin{proof}
	Pick any framing $\R$ on $(G,\sim)$.
	Theorem~\ref{thm:pres-general} shows that the indicator vectors of the elementary bands of $(G,\sim)$ generate the recession cone $\text{Rec}(\F_1(G,\sim))$. If $(G,\sim)$ is acyclic, then there are no bands of $(G,\sim)$ and hence $\F_1(G,\sim)$ is bounded.
	On the other hand, if $(G,\sim)$ contains a band $B$ then it is immediate that the indicator vector $\I(B)$ is in $\text{Rec}(\F_1(G,\sim))$ so that $\F_1(G,\sim)$ is unbounded.
\end{proof}

\begin{example}\label{ex:bowtie-pres}
	Let $(G,\sim,\R)$ be the framed turbulence chart of Figure~\ref{fig:bowtie}. 
	Let $v$ be the unique internal vertex. Let $\te$ be the orientation of $e$ with head $v$, and let $\th$ be the orientation of $h$ with tail $v$. Let $\tf$ and $\tg$ be the ``clockwise'' orientations of these edges -- e.g., the tail half-edge of $\tf$ is labelled $2$ by $\R$, and the head labelled $3$.

	Up to equivalence, the three elementary routes are the simple route $\te\th$ and the lollipops $\te\tf(\te)^{-1}$ and $(\th)^{-1}\tg\th$, which correspond to the three vertices as labelled in the figure.
	Up to equivalence, the only two bands are $\tf\tg$ and $\tf(\tg)^{-1}$. The former is a self-compatible barbell, while the latter is a non-self-compatible barbell. The indicator vector of the unique self-compatible elementary band $\tf\tg$ is the ray moving to the right in the figure, which is the extremal ray generating the recession cone of $\F_1(G,\sim)$.
\end{example}

See also Figure~\ref{fig:trapezoid}, where the indicator of the unique nonelementary route is the unique integer point which is not a vertex.

\subsection{Bundle subdivisions of turbulence polyhedra}

We now use gentle envelopes to get bundle subdivisions of turbulence polyhedra.

\begin{defn}\label{defn:bundley}
	Let $\bK=\K\cup \B$ be a bundle of $(G,\sim,\R)$. A \emph{$\bK$-bundle combination} (of $(G,\sim,\R)$) is a linear combination of indicator vectors of trails of $(G,\sim,\R)$
	\[F=\sum_{p\in \bK}a_p\I(p),\]
	such that each $a_p$ is nonnegative. The \emph{strength} of the bundle combination is $\sum_{p\in \K}a_p$ (note we iterate only over routes $\K$, not all trails $\bK$). Note that this is the strength of the resulting flow $F$. The bundle combination is \emph{unit} if its strength is 1. In other words, a unit $\bK$-bundle combination is a convex combination of indicator vectors of routes in ${\K}$ plus a nonnegative combination of indicator vectors of bands in $\B$.
	A bundle combination is
	\emph{positive} if $a_p$ is positive for every $p\in \bK$.
	The \emph{bundle simplihedron} $\D_1(\bK)$ is the polyhedron of unit $\bK$-bundle combinations.
	The \emph{bundle cone} $\D_{\geq0}(\bK)$ is the cone of $\bK$-bundle combinations.
\end{defn}

\begin{lemma}\label{lem:comb-pres}
	Let $\bK=\K\cup\B$ be a bundle. Then a flow $F\in\F_{\geq0}(G,\sim)$ is in $\D_{\geq0}(\bK)$ if and only if $F$ may be realized as a $\bK$-bundle combination.
\end{lemma}
\begin{proof}
	It is immediate that the polyhedron $\D_{\geq0}(\bK)$ is presented as the cone of the vectors $\{\I(p)\ :\ p\in\bK\}$. Hence, realizing an arbitrary flow $F\in\F_{\geq0}(G,\sim)$ as a $\bK$-bundle combination
	\[
		F=\sum_{p\in\K}a_p\I(p)+\sum_{B\in\B}a_B\I(B)
	\]
	is the same as realizing $F$ within the cone $\D_{\geq0}(\bK)$.
\end{proof}

\begin{thm}[{Theorem~\ref{ithm:bundle}}]
	\label{thm:bundle}
	Let $(G,\sim,\R)$ be a framed turbulence chart.
	Any nonnegative flow has at most one representation as a positive bundle combination.
	A dense subset of $\F_{\geq0}(G,\sim)$, including all rational flows, are obtained as a bundle combination. If $F\in\F_{\geq0}(G,\sim)$ is an integer flow, then its bundle combination has integral coefficients.
\end{thm}
\begin{proof}
	Choose by Theorem~\ref{prop:face-of-gentle} a gentle envelope $(G',\sim',\R',W)$ of $(G,\sim,\R)$. 
	Let $(G'',\sim'',\R'')$ be the framed turbulence chart given by deleting the edges of $W$ from $(G,\sim,\R)$ and their fringe vertices.

	Theorem~\ref{thm:BBundle} (translating through Proposition~\ref{prop:gentandchart}) shows the desired decomposition result for $(G',\sim',\R')$:
		namely, that any nonnegative flow $F$ on $(G',\sim',\R')$ has at most one representation as a positive bundle combination, that all rational nonnegative flows have such a representation, and that this combination has integer coefficients when $F$ is an integer flow.

	In particular, this decomposition result holds for the $W$-avoiding nonnegative flows of $(G',\sim',\R')$, which are precisely the nonnegative flows of $(G'',\sim'',\R'')$. Moreover, if $F$ is a $W$-avoiding nonnegative flow, then its bundle combination must consist of only $W$-avoiding routes -- i.e., it must be a bundle combination of $(G'',\sim'',\R'')$. This shows the desired decomposition result for $(G'',\sim'',\R'')$. Finally, contracting $(G'',\sim'',\R'')$ to $(G,\sim,\R)$ and applying Lemma~\ref{lem:contractchart} gives the desired decomposition result for $(G,\sim,\R)$.
\end{proof}

An immediate consequence of Theorem~\ref{thm:bundle} is that bundle simplihedra and bundle cones are simplihedra presented by the indicator vectors of routes and bands of their bundles:

\begin{cor}\label{cor:bs-simple}
	Let $\bK=\K\cup\B$ be a bundle.
	Then $\D_1(\bK)$ and $\D_{\geq0}(\bK)$ are simplicial polyhedra with the following minimal presentations:
	\begin{align*}
		\D_1(\bK)&=\textup{conv}\big(\{\I(p)\ :\ p\in\K\}\big)+\textup{cone}\big(\{\I(B)\ :\ B\in\B\}\big)\\
		\D_{\geq0}(\bK)&=\mathbf{0}+\textup{cone}\big(\{\I(p)\ :\ p\in\K\}\big)
	\end{align*}
\end{cor}
\begin{proof}
	It is immediate by definition that $\D_1(\bK)$ is a pointed polyhedron whose vertices are contained in the set $\{\I(p)\ :\ p\in\K\}$ and whose recession cone is generated by $\{\I(B)\ :\ B\in\B\}$.
	Suppose for some $p'\in\K$ that $\I(p)$ is not a vertex; by the above presentation of $\D_1(\bK)$ we may present $\I(p')$ as a nonnegative linear combination
	\[
		\I(p')=\sum_{p'\neq p\in\K}a_p\I(p)+\sum_{B\in\B}a_B\I(p),
	\]
	where $\sum_{p'\neq p\in\K}a_p\I(p)=1$. Then the two sides of the above equation represent different bundle combinations for the same flow $\I(p')$, contradicting Theorem~\ref{thm:bundle}. This shows that the vertices of $\D_1(\bK)$ are precisely the indicator vectors of routes of $\bK$. Similarly, suppose that $\{\I(B)\ :\ B\in\B\}$ is not a minimal generating set for $\text{Rec}(\D_1(\bK))$. Then there exists $B'\in\B$ with
	\[
		\I(B')=\sum_{B'\neq B\in\B}a_B\I(B),
	\]
	contradicting uniqueness of bundle combinations as in Theorem~\ref{thm:bundle}. This shows that $\{I(B)\ :\ B\in\B\}$ is a minimal generating set for the recession cone, showing that $\D_1(\bK)$ is minimally presented as in the corollary statement. Then uniqueness of bundle combinations given in Theorem~\ref{thm:bundle} shows that $\D_1(\bK)$ satisfies condition (1) of Proposition~\ref{prop:simplicial}, hence $\D_1(\bK)$ is simplicial.

	This shows the desired statement for $\D_1(\bK)$; the proof for $\D_{\geq0}(\bK)$ is similar.
\end{proof}

In fact, Theorem~\ref{thm:bundle} is a simplicial subdivision result on the turbulence polyhedron of an arbitrary turbulence chart, and on the cone of nonnegative flows.

\begin{cor}
	\label{cor:subd}
	For a framed turbulence chart $(G,\sim,\R)$, the \emph{bundle subdivisions}
	\begin{align*}
		\mathcal S_{\geq0}(G,\sim,\R)&:=\{\D_{\geq0}(\bK)\ :\ \bK\text{ is a maximal bundle of }(G,\sim,\R)\} \text{ and} \\
		\mathcal S_1(G,\sim,\R)&:=\{\D_1(\bK)\ :\ \bK\text{ is a maximal bundle of }(G,\sim,\R)\text{ with at least one route}\} 
	\end{align*}
	are simplicial subdivisions of $\F_{\geq0}(G,\sim)$ and $F_1(G,\sim)$, respectively, covering every rational point.
\end{cor}
\begin{proof}
	Define 
	$\mathcal S_{1}(G,\sim,\R):=\{\D_1(\bK)\ :\ \bK\text{ is a maximal bundle of }(G,\sim,\R)\}$. We will show that this is a subdivision of $\F_1(G,\sim)$ covering every rational point. By Corollary~\ref{cor:bs-simple}, each member $\D_1(\bK)$ of this set is a simplicial polyhedron.
	Theorem~\ref{thm:bundle} shows that any rational flow $F\in\F_{1}(G,\sim,\R)$ is achieved through a $\bK$-bundle subdivision for some bundle $\bK$; since $F$ is unit, the bundle $\bK$ must contain a route. Then $F\in\D_1(\bK')$, where $\bK'$ is any maximal bundle containing $\bK$. This shows that $\mathcal S_1(G,\sim,\R)$ satisfies density. It remains to show that it satisfies the strong intersection property.
	Indeed, let $\bK_1$ and $\bK_2$ be distinct maximal bundles of $(G,\sim,\R)$ containing at least one route. 
	Note that neither $\D_1(\bK_1)$ nor $\D_1(\bK_2)$ is empty because both bundles contain a route.
	Lemma~\ref{lem:comb-pres} shows that $\D_1(\bK_1)\cap\D_1(\bK_2)=\D_1(\bK_1\cap\bK_2)$, and Corollary~\ref{cor:bs-simple} with Proposition~\ref{prop:simplicial} shows that this is either empty or a face of both.
	This ends the proof for $\mathcal S_1(G,\sim,\R)$. The proof for $\mathcal S_{\geq0}(G,\sim,\R)$ is similar.
\end{proof}

Note that the stipulation that the bundles giving cells of the bundle subdivision on $\F_1(G,\sim)$ must contain a route is not necessary for fringed algebras (equivalently, gentle turbulence charts) because every maximal bundle of a fringed algebra contains all straight routes.

We state this subdivision result as a triangulation result in the acyclic case.

\begin{cor}
	\label{cor:tri}
	For an acyclic framed turbulence chart $(G,\sim,\R)$, the \emph{clique triangulations}
	\begin{align*}
		\mathcal T_{\geq0}(G,\sim,\R)&:=\{\D_{\geq0}(K)\ :\ \K\text{ is a maximal clique of }(G,\sim,\R)\} \text{ and}\\
		\mathcal T_1(G,\sim,\R)&:=\{\D_1(K)\ :\ \K\text{ is a maximal clique of }(G,\sim,\R)\}
	\end{align*}
	are complete unimodular triangulations of $\F_{\geq0}(G,\sim)$ and $\F_1(G,\sim)$, respectively.
\end{cor}

See Figure~\ref{fig:trapezoid} for the clique triangulations from two framings of the same acyclic turbulence chart. 
Since both sides of this figure use the clockwise framing, the cells of this triangulation are represented by maximal collections of pairwise noncrossing routes.
Note that when $(G,\sim)$ is acyclic, the turbulence polyhedron $\F_1(G,\sim)$ is a bounded polytope and the bundle subdivision is equal to the clique triangulation.

The following lemma is a useful observation which we will use for computation purposes in Example~\ref{ex:bowtie-subd}.

\begin{lemma}\label{lem:ind-nob}
	Let $(G,\sim,\R)$ be a framed turbulence chart and let $F$ be an integer unit nonnegative flow. Then $F$ is either realized as a positive bundle combination using a band, or $F$ is the indicator vector of a self-compatible route $p_F$.
\end{lemma}
\begin{proof}
	Theorem~\ref{thm:bundle} shows that $F$ is realized as a positive bundle combination
	\[F=\sum_{p\in \K}a_p\I(p)+\sum_{B\in\mathcal B}a_B\I(B)\]
	where each coefficient $a_p$ and $a_B$ is a positive integer and $\sum_{p\in \K}a_p=1$. In particular, these conditions imply that $\K=\{p_F\}$ for some self-compatible route ${p_F}$ with $a_{p_F}=1$. If $\mathcal B$ is nonempty empty, then $F$ is realized as a positive bundle combination using the bands of $\B$; otherwise, $F=\I(p_F)$.
\end{proof}

\begin{example}\label{ex:bowtie-subd}
	Let $(G,\sim,\R)$ be the framed turbulence chart of Figure~\ref{fig:bowtie}. Since this figure uses the clockwise framing, the maximal bundles are maximal collections of pairwise noncrossing trails. On the right of the figure, vertices are labelled by the self-compatible routes with the relevant indicator vector and the bundle subdivision is indicated with dotted lines.

	We will prove that the bundle subdivision is as depicted in the figure.
	As in Example~\ref{ex:bowtie-pres}, let $v$ be the unique internal vertex. Let $\te$ be the orientation of $e$ with head $v$, and let $\th$ be the orientation of $h$ with tail $v$. Let $\tf$ and $\tg$ be the ``clockwise'' orientations of these edges -- e.g., the tail half-edge of $\tf$ is labelled $2$ by $\R$, and the head labelled $3$.

	First, note that $\te\th$ is a self-compatible route which is compatible with every trail.  Observe that there are only two bands of $(G,\sim,\R)$ up to equivalence: $\tf\tg$ and $\tf(\tg)^{-1}$. The latter band is not self-compatible, but the former is both self-compatible and elementary. 

	We will show that $\tf\tg$ is compatible with no self-compatible route other than $\te\th$.
	Let $p$ be a route beginning with $\te$ such that $p$ is compatible with $\tf\tg$ but $p\neq\te\th$. Since $\te\tf^{-1}$ is not compatible with $\tf\tg$, the route must begin with $\te\tf$. Similarly, since $\te\tf\tg^{-1}$ and $\te\tf\te^{-1}$ are incompatible with $\tf\tg$, $p$ must begin with $\te\tf\tg$. Repeating this logic, one observes that we must have $p=\te(\tf\tg)^a\th$ for some $a\geq1$. Note that any such route crosses itself and hence is not self-compatible; this shows that there is no self-compatible route beginning with $\te$ which is compatible with $\tf\tg$. Symmetrically, there is no self-compatible route beginning with $\th^{-1}$ other than $\te\th\equiv\th^{-1}\te^{-1}$ which is compatible with $\tf\tg$, hence $\te\th$ is the only self-compatible route which is compatible with $\tf\tg$.
	This shows that $\{\te\th,\tf\tg\}$ is a maximal bundle; since $\tf\tg$ is the only self-compatible band, this is the only maximal bundle containing a band. Its bundle simplihedron is represented in Figure~\ref{fig:bowtie} by the blue dotted ray beginning at $\I(\te\th)$ and continuing to the right.

	We wish to describe the rest of the self-compatible routes and maximal bundles.
	We could work them out using compatibility conditions manually, but instead we will find them by using the integer unit flows, representing a flow as an ordered pair of coefficients of $e,f,g,h$ respectively.
	Note that a unit integer nonnegative flow $F$ must have either $F(\te)=F(\th)=1$, or $F(\te)=2$ and $F(\th)=0$, or $F(\te)=0$ and $F(\th)=2$. It is then not hard to see that the complete list of unit integer nonnegative flows is:
	\begin{enumerate}
		\item $B_a:=(1,a,a,1)$ for $a\in\mathbb Z_{\geq0}$,
		\item $E_a:=(2,a+1,a,0)$ for $a\in\mathbb Z_{\geq0}$, and
		\item $H_a:=(0,a,a+1,2)$ for $a\in\mathbb Z_{\geq0}$.
	\end{enumerate}
	Note that $H_a$ is realized as the bundle combination $\I(\te\th)+a\I(\tf\tg)$ for $a\geq0$, hence uniqueness of Theorem~\ref{thm:bundle} shows that no self-compatible route returns $H_a$ with the exception of $\I(\te\th)=H_0$.
	Hence, every self-compatible route $p\neq\te\th$ returns through its indicator vector a unit integer nonnegative flow $\I(p)$ of the form $B_a$ or $E_a$. 
	Since $\tf\tg$ is compatible only with the trail $\te\th$, it follows that $B_a$ and $H_a$ may not be realized as a bundle combination using a band, so by Lemma~\ref{lem:ind-nob} there exist self-compatible routes $p_{B_a}$ and $p_{E_a}$ with $\I(p_{B_a})=B_a$ and $\I(p_{E_a})=E_a$ for every $a\in\mathbb Z_{\geq0}$.
	We have shown that the complete list of self-compatible trails is
	\[
		\{\te\th\}\cup\{\tf\tg\}\cup\{p_{B_a}\ :\ a\in\mathbb Z_{\geq0}\}\cup\{p_{E_a}\ :\ a\in\mathbb Z_{\geq0}\}.
	\]
	We have already shown that $\{\te\th,\tf\tg\}$ is a maximal bundle whose bundle simplihedron is the dotted blue ray bisecting $\F_1(G,\sim)$ into its top and bottom half.
	(All rational points of) the top half must covered by maximal bundle simplihedra of cliques, all of which must contain $\I(\te\th)$. The only possibility, then, is for $\{\te\th,p_{E_a},p_{E_{a+1}}\}$ to be a maximal clique for every $a\in\mathbb Z_{\geq0}$.
	Similarly, $\{\te\th,p_{B_a},p_{B_{a+1}}\}$ is a maximal clique for every $a\in\mathbb Z_{\geq0}$. This completes the proof that the maximal bundles of $(G,\sim,\R)$ are
	\[
		\{\te\th,\tf\tg\}\cup\left\{\{\te\th,p_{E_a},p_{E_{a+1}}\}\ :\ a\in\mathbb Z_{\geq0}\right\}\cup\left\{\{\te\th,p_{B_a},p_{B_{a+1}}\}\ :\ a\in\mathbb Z_{\geq0}\right\},
	\]
	showing that the bundle subdivision of $\F_1(G,\sim)$ induced by $\R$ is as pictured in Figure~\ref{fig:bowtie}. 
\end{example}

\begin{remk}
	Even in the gentle case, the bundle subdivision of the turbulence polyhedron of a turbulence chart may fail to be complete; see~\cite[Example 5.22 and \S10]{BERG}.
	In~\cite{BERG}, this was addressed by the addition of vortex dissections of turbulence polyhedra, which are complete but may fail the strong intersection property of Definition~\ref{defn:subd} and may have cells which are not simplicial. Cells of a vortex dissection are indexed by certain cliques called band-stable cliques. In fact, the theory of vortex dissections may be extended to general turbulence charts, but we omit this theory for the sake of space.
\end{remk}

\section{Specialization to framed directed graphs and gentle algebras}
\label{sec:specialization}

In this section, we reiterate that our results on framed turbulence charts generalize known results simultaneously in the setting of framed DAGs and gentle algebras. Moreover, we show that our results imply new results in the setting of framed directed graphs which may have cycles.

\subsection{Gentle algebras and framed DAGs}

Recall Definition~\ref{defn:genttochart} and Proposition~\ref{prop:gentandchart}: for a gentle algebra $\L$, one obtains a gentle framed turbulence chart $(G_\tL,\sim_\tL,\R_\tL)$ with $\F_1(\tL)=\F_1(G_\tL,\sim_\tL)$ and appropriate bijections on trails.
Moreover, it is immediate that the bundle subdivision of $(G_\tL,\sim_\tL,\R_\tL)$ is the same as the bundle subdivision of $\tL$.
Hence, our theory of presentations and bundle subdivisions of framed turbulence charts generalizes the theory on fringed algebras developed in~\cite{BERG}.

Similarly, recall Definition~\ref{defn:chart-from-dg-} and Proposition~\ref{prop:dg-to-chart}: given a framed DAG $\Gamma=(G,\R)$ one obtains a directable framed turbulence chart $(G_\Gamma,\sim_\Gamma,\R_\Gamma)$ with $\F_1(\Gamma)=\F_1(G_\Gamma,\sim_\Gamma)$ and appropriate bijections on trails.
Because $(G_\Gamma,\sim_\Gamma)$ is acyclic, the clique triangulation and bundle subdivision of $(G_\Gamma,\sim_\Gamma,\R_\Gamma)$ agree; it is immediate from the definitions that these are the same as the framing triangulation of $\Gamma$. Hence, our theory of presentations and bundle subdivisions generalizes the theory of framing triangulations of framed DAGs.

For an example of the above two paragraphs, recall Figure~\ref{fig:turbchartintro} showing a directable gentle framed turbulence chart along with a corresponding framed DAG and gentle algebra. Figures~\ref{fig:sqdag},~\ref{fig:sqturb}, and~\ref{fig:sqgent} show the resulting triangulations from the perspective of framed DAGs, turbulence charts, and gentle algebras.

This justifies our claim that framed turbulence charts are the common generalization of the theories of fringed algebras and framed DAGs, e.g., Figure~\ref{fig:CUBEOFGEN}.

\subsection{Framed directed graphs}

Recall Section~\ref{ssec:directable}, where we connected framed directed graphs to framed directable turbulence charts. We now use this connection to specialize definitions and results from the setting of framed turbulence charts to framed directed graphs.

\begin{lemma}
	The contraction, degree-reduction, steepening, filling, and band-correction moves of Section~\ref{sec:ge} all preserve directability of a framed turbulence chart. 
\end{lemma}
\begin{proof}
	It is immediate to see that if $(G,\sim,\R)$ is a directable framed turbulence chart and $(G',\sim',\R')$ is obtained from $(G,\sim,\R)$ by contracting an edge or deleting a fringe edge, then $(G',\sim',\R')$ is directable if and only if $(G,\sim,\R)$ is directable. In particular, this shows that contractions preserve directability.
	The result for degree-reduction, steepening, filling, and band-correction moves follows because all of these moves are undone by a combination of deleting fringe edges and contracting.
\end{proof}

The above lemma implies that the process outlined in Section~\ref{sec:ge} to find a gentle envelope of a turbulence chart specializes without issue to framed directed graphs (see Remark~\ref{remk:wilt}). This gives us the following variant of Theorem~\ref{prop:face-of-gentle}:

\begin{thm}\label{thm:faces-of-gentle-dg}
	If $(G,\R)$ is a framed directed graph, then there exists a gentle framed directed graph $(G',\R')$ with a set $W$ of source/sink edges such that deleting all edges of $W$ from $(G',\R')$ and performing some contraction steps yields $(G,\R)$.
\end{thm}
We call $(G',\R',W)$ a \emph{gentle envelope} of $(G,\R)$.
One may now prove the following exactly as we proved Corollary~\ref{cor:face-of-gentle}:
\begin{cor}\label{cor:face-of-gentle-dag}
	Up to unimodular equivalence, a flow polyhedron is precisely a face of a gentle flow polyhedron.
\end{cor}

We now translate our presentation result Theorem~\ref{thm:pres-general} to framed directed graphs. Recall Definitions~\ref{defn:elroute} and~\ref{defn:elband} of elementary routes and bands of a framed directed graph, and note that lollipops and barbells are impossible in a directable turbulence chart. Then, defining a \emph{simple} route or band of a directed graph to be one which does not use the same vertex multiple times, we get the following translation of Theorem~\ref{thm:pres-general} through Proposition~\ref{prop:dg-to-chart}:

\begin{thm}[{Theorem~\ref{thm:pres-general}}]
	Let $(G,\R)$ be a framed directed graph. Then $p\mapsto\I(p)$ is a bijection from elementary routes of $(G,\R)$ to vertices of $\F_1(G)$, and $B\mapsto\I(B)$ is a bijection from self-compatible elementary bands of $(G,\R)$ to the extremal rays of $\F_1(G)$.
\end{thm}

Recall the definition of the bundle complex of a framed directed graph from Section~\ref{ssec:frameddg}.
We copy verbatim the following definitions from Definition~\ref{defn:bundley} but phrase them for framed directed graphs.
\begin{defn}\label{defn:bundley}
	Let $\bK=\K\cup \B$ be a bundle of a framed directed graph $(G,\R)$. A \emph{$\bK$-bundle combination} (of $(G,\R)$) is a linear combination of indicator vectors of trails of $(G,\R)$
	\[F=\sum_{p\in \bK}a_p\I(p)\]
	such that each $a_p$ is nonnegative. The \emph{strength} of the bundle combination is $\sum_{p\in \K}a_p$ (note we iterate only over routes $\K$, not all trails $\bK$). The bundle combination is \emph{unit} if its strength is 1. It is \emph{positive} if $a_p$ is positive for every $p\in \bK$.
	The \emph{bundle simplihedron} $\D_1(\bK)$ is the polyhedron of unit $\bK$-bundle combinations. The \emph{bundle cone} $\D_{\geq0}(\bK)$ is the cone of $\bK$-bundle combinations.
\end{defn}

Proposition~\ref{prop:dg-to-chart} now allows us to phrase Theorem~\ref{thm:bundle} and Corollary~\ref{cor:subd} for framed directed graphs:

\begin{thm}[{Theorem~\ref{thm:bundle}, Corollary~\ref{cor:subd}}]
	\label{thm:bundledag}
	Let $(G,\R)$ be a framed directed graph.
	Any nonnegative flow has at most one representation as a positive bundle combination.
	A dense subset of $F\in\F_{\geq0}(G)$, including all rational flows, are obtained as a bundle combination. If $F\in\F_{\geq0}(G)$ is an integer flow, then its bundle combination has integral coefficients.
	Consequently, the \emph{bundle subdivisions}
	\begin{align*}
		\mathcal S_{\geq0}(G,\R)&:=\{\D_{\geq0}(\bK)\ :\ \bK\text{ is a maximal bundle of }(G,\R)\} \text{ and} \\
		\mathcal S_1(G,\R)&:=\{\D_1(\bK)\ :\ \bK\text{ is a maximal bundle of }(G,\R)\text{ with at least one route}\} 
	\end{align*}
	are simplicial polyhedral subdivisions of $\F_{\geq0}(G)$ and $F_1(G)$, respectively, covering every rational point.
\end{thm}

Theorem~\ref{thm:bundledag} restricts to the framing triangulation result of Danilov, Karzanov, and Koshevoy~\cite{DKK} when $G$ has no cycles, and it restricts to a triangulation result of 
Abram, Bastidas, Dequ\^{e}ne, Morales, Park, and Thomas~\cite{UQAM} in the case of \emph{cyclic amply framed DAGs} (see Remark~\ref{remk:uqam1}).

\section{Examples Inspired by the Kronecker Quiver}
\label{sec:examples}

The framed turbulence charts of Figures~\ref{fig:trapezoid} and~\ref{fig:bowtie} have served as the running examples of this article, and their presentations and subdivisions were discussed in the previous section.
In the interest of providing more examples, we conclude this article with a few framed turbulence charts born of the Kronecker quiver and their subdivided turbulence polyhedra.
The Kronecker quiver is a classic example in the representation theory of algebras consisting of two parallel arrows between two vertices. Figure~\ref{fig:kronhor} shows the fringed algebra of (the gentle algebra of) the Kronecker quiver and the corresponding framed turbulence chart and framed directed graph.

\begin{figure}
	\centering
	\def\svgscale{.25}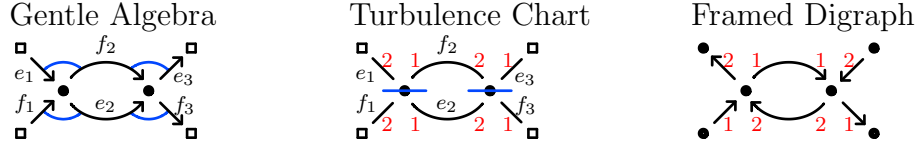
	\caption{The fringed algebra of the Kronecker quiver (left) with the corresponding framed turbulence chart (middle) and framed directed graph (right).}
	\label{fig:kronhor}
\end{figure}

\begin{example}\label{ex:kron1}
	The Kronecker framed turbulence chart of Figure~\ref{fig:kronfinal} is not drawn using the clockwise framing. Figure~\ref{fig:kronclock} redraws the same graph using the clockwise framing and portrays two maximal bundles.
	Consider the oriented edges $\te_1$ and $\tf_1$ to begin at fringed edges and consider the oriented edges $\te_3$ and $\tf_3$ to end at fringed edges. Consider $\te_2$ and $\tf_2$ oriented left to right in the drawing.

	Figure~\ref{fig:kronfinal} shows its three-dimensional turbulence polyhedron. The lattice points are labelled by the corresponding self-compatible route when possible.
	The simplex of the maximal clique in the middle of Figure~\ref{fig:kronclock} is highlighted and outlined in red in Figure~\ref{fig:kronclock}.

	The routes $p_e:=\te_1\te_2\te_3$ and $p_f:=\tf_1\tf_2\tf_3$ are exceptional (i.e., they are compatible with every trail).
	For $a\in\mathbb Z_{\geq0}$, define the self-compatible routes $l_a:=\te_1(\te_2\tf_2^{-1})^a\tf_1^{-1}$ and $r_a:=\te_3^{-1}(\te_2^{-1}\tf_2)^a\tf_3$. Then the self-compatible trails of $(G,\sim,\R)$ are
	\[
		\{p_e,p_f,\te_2\tf_2^{-1}\}\cup\{r_a\ :\ a\in\mathbb Z_{\geq0}\}\cup\{l_a\ :\ a\in\mathbb Z_{\geq0}\}.
	\]
	The elementary routes are $\{p_e,p_f,r_0,l_0\}$.
	There is only one band $\te_2\tf_2^{-1}$, which is elementary.
	The maximal bundles of $(G,\sim,\R)$ are:
	\[
		\{p_e,p_f,\te_2\tf_2^{-1}\}\cup
		\{p_e,p_f,r_0,l_0\}\cup
		\big\{\{p_e,p_f,r_a,r_{a+1}\}\ :\ a\in\mathbb Z_{\geq0}\big\}\cup
		\big\{\{p_e,p_f,l_a,l_{a+1}\}\ :\ a\in\mathbb Z_{\geq0}\big\}
	\]

	\begin{figure}
		\centering
		\def\svgscale{.25}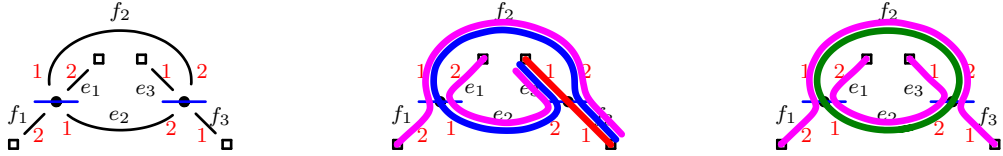
		\caption{The Kronecker framed turbulence chart drawn using the clockwise framing.}
		\label{fig:kronclock}
	\end{figure}
	\begin{figure}
		\centering
		\def\svgscale{.25}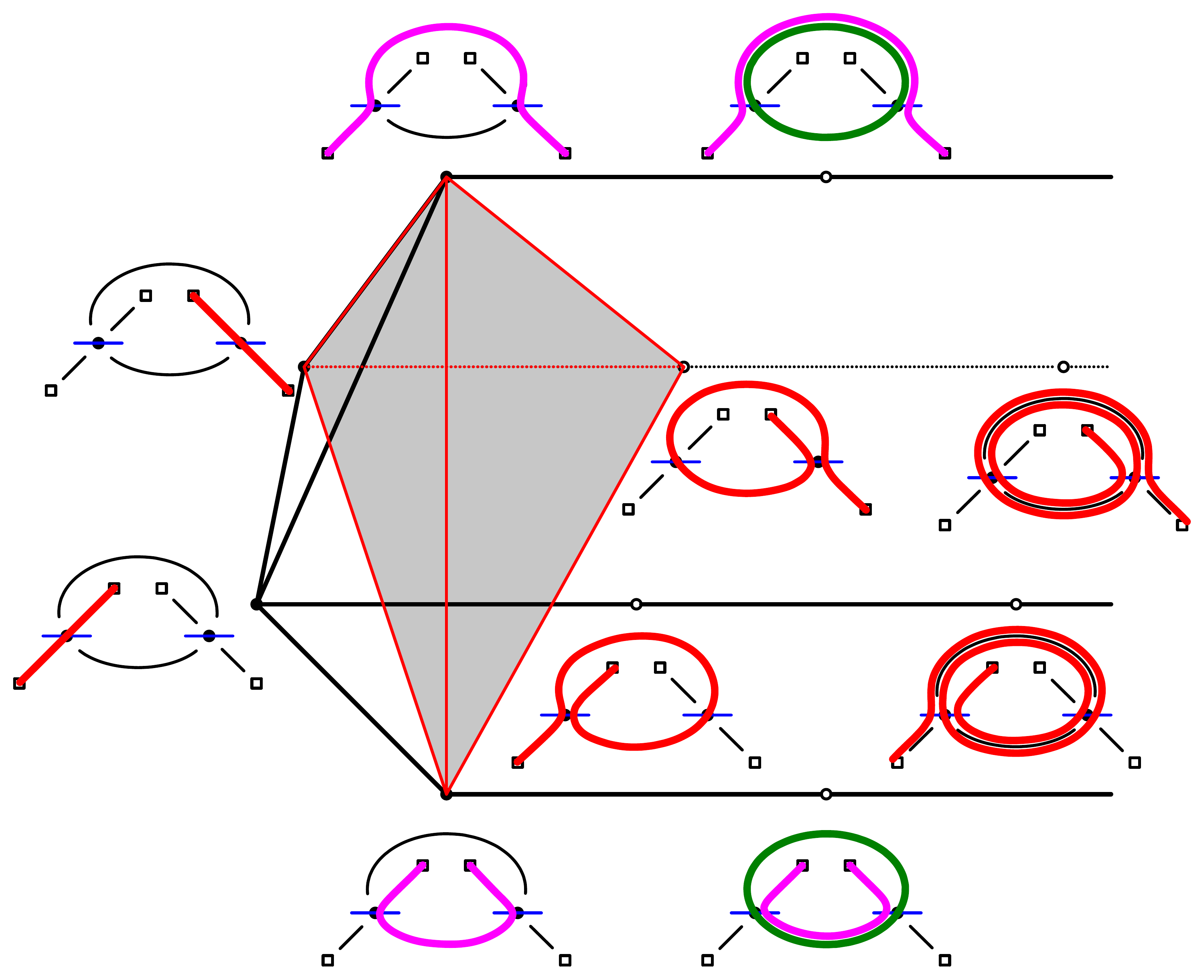
		\caption{The turbulence polyhedron of the Kronecker framed turbulence chart.}
		\label{fig:kronfinal}
	\end{figure}
\end{example}

\begin{example}\label{ex:face}
	Figure~\ref{fig:kronface} shows the framed turbulence chart $(G',\sim',\R')$ obtained by deleting the edges $e_1$ and $e_2$ from the Kronecker framed turbulence chart $(G,\sim,\R)$ of Figure~\ref{fig:kronclock}.
	The turbulence polyhedron is a ray, which is the face of $\F_1(G,\sim)$ obtained by zeroing out flow through the edges $e_1$ and $f_1$.
	Accordingly, the bundle subdivision of $(G',\sim',\R')$ is the bundle subdivision of $(G,\sim,\R)$ restricted to this face.
	The maximal bundles are
	\[
		\big\{e_2f_2^{-1}\big\}\cup\big\{\{r_a,r_{a+1}\}\ :\ a\geq0\big\},
	\]
	where the paths $r_a$ are defined as in Example~\ref{ex:kron1}.
	The maximal bundle $\{e_2f_2^{-1}\}$ contains no route and hence does not index a cell of the bundle subdivision on $\F_1(G,\sim)$ induced by $\R$. It will, however, index a cell of the bundle subdivision on the cone of nonnegative flows $\F_{\geq0}(G,\sim)$.
	\begin{figure}
		\centering
		\def\svgscale{.25}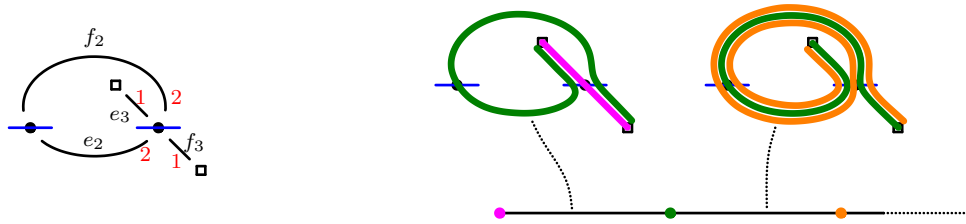
		\caption{A face of the Kronecker framed turbulence chart.}
		\label{fig:kronface}
	\end{figure}
\end{example}

We now give a framing on the Kronecker turbulence chart with a particularly simple bundle subdivision.

\begin{example}\label{ex:kron2}
	Figure~\ref{fig:kron-alt} shows the embedding of the Kronecker turbulence chart $(G,\sim)$ from Figure~\ref{fig:kronhor} equipped with its clockwise framing $\R'$.
	Since $(G,\sim)$ is directable, this turbulence chart may still be described by a framed directed graph with an oriented cycle, but $(G,\sim,\R')$ is not gentle so there is no fringed algebra.
	Consider all oriented edges $\te_i$ and $\tf_j$ left to right.
	With the new framing $\R'$, the band $e_2f_2^{-1}$ is compatible with every trail. The only self-compatible routes are $\{p_e,p_f,r_0,l_0\}$, which form two maximal bundles as pictured.
	We finally remark that the framed directed graph is an example of a directed graph with an ample framing all of whose cycles are monochromatic (i.e., all half edges have the same label in $\{1,2\}$). Such framings on directed graphs are called \emph{cyclic ample framings} in~\cite[Definition 6.7]{UQAM} and studied in greater depth.
	In particular, the amply framed directed graph whose internal subgraph is a monochromatic cycle of length $n$ gives a bundle subdivision whose dual graph is the 1-skeleton of the cyclohedron $\mathcal C_n$ encoding the centrally symmetric triangulations of a regular $2n$-gon~\cite[Proposition 7.6]{UQAM} (the Kronecker framed directed graph of Figure~\ref{fig:kron-alt} being the $n=2$ example).
	\begin{figure}
		\centering
		\def\svgscale{.25}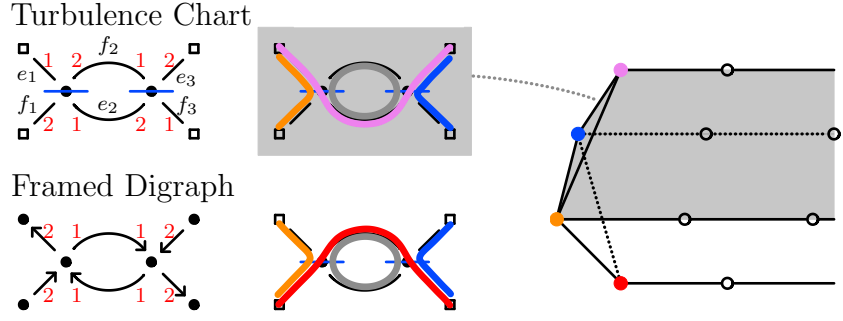
		\caption{The Kronecker turbulence chart with an altered framing.}
		\label{fig:kron-alt}
	\end{figure}
\end{example}

Examples~\ref{ex:kron1} and~\ref{ex:kron2} showed the bundle subdivisions from two different framings on the same Kronecker turbulence chart. Both examples featured exceptional routes, corresponding to points of the turbulence polyhedron which appear in every maximal bundle simplihedron.
The former example had an infinite number of simplices with one lower-dimensional bundle simplihedron using the band $\te_2\tf_2^{-1}$, while the latter example had only two bundles, both containing $\te_2\tf_2^{-1}$. We now give a framing on this Kronecker turbulence chart whose bundle subdivision is a true triangulation of the turbulence polyhedron into full-dimensional simplices because the band $\te_2\tf_2^{-1}$ is not compatible with any route.

\begin{example}\label{ex:kron3}
	Figure~\ref{fig:kronnoband} gives a final framing $\R''$ of the Kronecker turbulence chart.
	Observe that the unique band $\{\te_2\tf_2\}$ of this chart is not compatible with any route. Hence, all maximal bundle simplihedra of the bundle subdivision of $\F_1(G,\sim)$ induced by $\R''$ are simplices, meaning that this bundle subdivision is a unimodular triangulation of $\F_1(G,\sim)$.
	\begin{figure}
		\centering
		\def\svgscale{.25}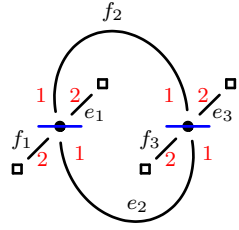
		\caption{A final framing of the Kronecker turbulence chart.}
		\label{fig:kronnoband}
	\end{figure}
\end{example}

\bibliographystyle{alphaurl}
\bibliography{biblio} 

\end{document}